\documentclass[11pt,a4paper]{article}
\usepackage{amsmath,amssymb,amsthm,amsfonts,latexsym,graphicx,subfigure}
\usepackage[all,import]{xy}
\usepackage{indentfirst,cite}
\usepackage{tabularx,booktabs}
\usepackage{caption,amscd}
\usepackage{hyperref}

\usepackage{fancyhdr,enumerate}
\usepackage{bbm,color}
\usepackage{tikz}
\usetikzlibrary{shapes.geometric, arrows}
\usepackage{accents,cases}
\usepackage{multirow}
\usepackage{euscript}\usepackage{wasysym}\usepackage{pdfsync}\usepackage{comment}
\usepackage{extarrows}
\newtheorem{introthm}{Theorem}

\usepackage{array,float}
\newcommand{\PreserveBackslash}[1]{\let\temp=\\#1\let\\=\temp}
\newcolumntype{C}[1]{>{\PreserveBackslash\centering}p{#1}}
\newcolumntype{R}[1]{>{\PreserveBackslash\raggedleft}p{#1}}
\newcolumntype{L}[1]{>{\PreserveBackslash\raggedright}p{#1}}

\DeclareMathOperator*{\argmin}{\ensuremath{arg\,min}}

\DeclareMathOperator*{\argopti}{\ensuremath{arg\,opti}}

\DeclareMathOperator*{\Sgn}{\ensuremath{Sgn}}

\DeclareMathOperator*{\sgn}{\ensuremath{sign}}

\DeclareMathOperator*{\vol}{\ensuremath{vol}}

\newcommand{\dist}{\mathrm{dist}}

\newcommand{\E}{\ensuremath{\mathcal{E}}}
\newcommand{\C}{\ensuremath{\mathcal{C}}}

\makeatletter
\def\wbar{\accentset{{\cc@style\underline{\mskip8mu}}}}

\makeatother

\renewcommand{\vec}[1]{\mbox{\boldmath \small $#1$}}

\newcommand{\power}{\ensuremath{\mathcal{P}}}

\newcommand{\supp}{\ensuremath{\mathrm{supp}}}
\newcommand{\Ch}{\ensuremath{\mathrm{Ch}}}
\newcommand{\R}{\ensuremath{\mathbb{R}}}
\newcommand{\A}{\ensuremath{\mathcal{A}}}

\theoremstyle{plain}
\newtheorem{theorem}{Theorem}[section]

\newtheorem{defn}{Definition}[section]
\newtheorem{notification}{Convention}

\newtheorem{remark}{Remark}

\newtheorem{cor}{Corollary}[section]
\newtheorem{pro}{Proposition}[section]
\newtheorem{example}{Example}[section]

\def\D{{\mathcal D}}

\allowdisplaybreaks[4]
\begin{document}
\bibliographystyle{unsrt}
\title{Discrete-to-Continuous Extensions: Lov\'asz extension, optimizations and eigenvalue problems
}
\author{J\"urgen Jost\footnotemark[1]\footnotemark[3], \and Dong Zhang\footnotemark[2]}
\footnotetext[1]{Max Planck Institute for Mathematics in the Sciences, Inselstrasse 22, 04103 Leipzig,
Germany. \\Email address:
{\tt  jost@mis.mpg.de}  (J\"urgen Jost).
}
\footnotetext[3]{Santa Fe Institute for the Sciences of Complexity, Santa Fe, NM 87501, USA}
\footnotetext[2]{LMAM and School of Mathematical Sciences, 
        Peking University,  
      100871 Beijing, China
\\
Email address:  
{\tt dongzhang@math.pku.edu.cn}  
(Dong Zhang).
}
\date{}
\maketitle

\begin{abstract}

\small
In this paper, we use various versions of Lov\'asz extension 
to systematically derive  continuous formulations of  problems from discrete mathematics.  This will take place in the following context:
\begin{enumerate}[-]
\item  For combinatorial optimization problems  in  quotient  form, we  systematically develop equivalent continuous versions, thereby making tools from convex optimization, fractional programming and more general continuous algorithms like the stochastic subgradient method available for such optimization problems.  Among other applications, we present an  
iteration scheme combining  the inverse power  and the steepest descent method to relax a Dinkelbach-type scheme for solving the equivalent continuous optimization. 
These results are natural and nontrivial generalizations of the related works by Hein   et al 
\cite{HeinBuhler2010,HS11,TVhyper-13}.

\item 

For some combinatorial quantities like Cheeger-type constants, 
we suggest a  nonlinear eigenvalue problem for a pair of Lov\'asz extensions of certain  functions, which encodes  certain  combinatorial structures. 
{This helps us to understand  the data generated by a pair of functions on a power set from a geometric point of view.}
\end{enumerate}

This theory has several applications to quantitative and combinatorial problems, including
\begin{enumerate}[(1)]
\item 
The equivalent continuous representations for the max $k$-cut problem, various Cheeger sets and isoperimetric constants are constructed.
This also initiates a study of Dirichlet and Neumann 1-Laplacians on graphs, in which the nodal domain property and Cheeger-type equalities are presented.
Among them, some Cheeger constants using different versions of vertex-boundary introduced in expander graph theory \cite{BHT00},
    are transformed into continuous forms, which recover the inequalities and identities on graph Poincare profiles proposed by Hume et al \cite{Hume17,HMT19,Hume19arxiv,HMT22}. { Also, we find that  the min-cut and max-cut problems are equivalent to solving  the first nontrivial  eigenvalue and the largest eigenvalue of a certain nonlinear eigenvalue  problem provided by the Lov\'asz extension,  respectively. {This leads  to one of the best continuous algorithms for the max-cut problem \cite{SZZmaxcut}, as recognized in the field of graph optimization.} }

\item Also, we derive a new equivalent continuous
    representation of the graph independence number, which can be compared with the Motzkin-Straus theorem. More importantly, an equivalent continuous optimization for the chromatic number is provided, which seems to be the first continuous representation of the graph vertex coloring number. We provide the first continuous reformulation of the frustration index in signed networks, and we find a connection to the so-called modularity measure. Graph matching numbers, submodular vertex covers and multiway partition problems can also be studied in our framework.
\end{enumerate}

\vspace{0.2cm}

\end{abstract}
\tableofcontents

\tikzstyle{startstop} = [rectangle, rounded corners, minimum width=1cm, minimum height=1cm,text centered, draw=black, fill=red!0]
\tikzstyle{io1} = [rectangle, trapezium left angle=80, trapezium right angle=100, minimum width=1cm, minimum height=1cm, text centered, draw=black, fill=blue!0]
\tikzstyle{io2} = [trapezium,  rounded corners, trapezium left angle=100, trapezium right angle=100, minimum width=1cm, minimum height=1cm, text centered, draw=black, fill=yellow!0]
\tikzstyle{process} = [rectangle, minimum width=1cm, minimum height=1cm, text centered, draw=black, fill=orange!0]
\tikzstyle{decision} = [circle, minimum width=1cm, minimum height=1cm, text centered, draw=black, fill=green!0]
\tikzstyle{decision2} = [ellipse, rounded corners=10mm, minimum width=2cm, minimum height=2cm, text centered, draw=black, fill=green!0]
\tikzstyle{arrow} = [thick,->,>=stealth]

\section{Introduction and Background}\label{sec:introduction}

As a fundamental tool in discrete mathematics,  Lov\'asz extension  has been deeply connected to  submodular analysis \cite{Choquet54,Lovasz}, and has been applied in many  areas like  combinatorial  optimization, game  theory, matroid theory, stochastic processes, electrical networks, computer vision and machine learning \cite{F05-book}. There are many generalizations, such as the disjoint-pair Lov\'asz extension and the Lov\'asz extension on distributive lattices \cite{F05-book,Murota03book}.  Recent developments include quasi-Lov\'asz extension on some algebraic structures and fuzzy mathematics  \cite{CouceiroMarichal11,CouceiroMarichal13}, applications of Lov\'asz extensions to  graph cut problems and computer science \cite{CSZ18,PRK19}, as well as Lov\'asz-softmax loss in deep learning \cite{BTB18}.


\vspace{0.16cm}

We shall start by looking  at the original Lov\'asz extension. For simplicity, we shall work throughout this paper with a finite and nonempty set $V=\{1,\cdots,n\}$ and its power set $\mathcal{P}(V)$. Also, we shall sometimes work on $\mathcal{P}(V)^k:=\{(A_1,\cdots,A_k):A_i\subset V,\,i=1,\cdots,k\}$ and $\mathcal{P}_k(V):=\{(A_1,\cdots,A_k)\in\power(V)^k:A_i\cap A_j=\varnothing,\,\forall i\ne j\}$, as well as some restricted family $\A\subset \mathcal{P}(V)^k$. We denote the cardinality of a set $A$ by $\#A$, and  identify every $A\in \mathcal{P}(V){\setminus\{\varnothing\}}$ with its indicator vector $\vec1_A\in \R^V=\R^n$. The Lov\'asz extension  extends the domain of $f$ to the whole Euclidean space\footnote{Some other versions in the literature only extend the domain to the cube $[0,1]^V$ or the nonnegative orthant $\R_{\ge0}^V$. In fact, many works on Boolean lattices  identify  $\power(V)$ with the discrete cube $\{0,1\}^n$.} $\R^V$. There are several equivalent
expressions:

\begin{itemize}
\item For $\vec x =(x_1,\dots ,x_n)\in \mathbb{R}^n$, let $\sigma:V\cup\{0\}\to V\cup\{0\}$ be a bijection such that $ x_{\sigma(1)}\le x_{\sigma(2)} \le \cdots\le x_{\sigma(n)}$ and $\sigma(0)=0$, where $x_0:=0$. The Lov\'asz extension of $f$ is defined by
\begin{equation}\label{eq:Lovasum}
f^{L}(\vec x)=\sum_{i=0}^{n-1}(x_{\sigma(i+1)}-x_{\sigma(i)})f(V^{\sigma(i)}(\vec x)),
\end{equation}
where   $V^0(\vec x)=V$ and $V^{\sigma(i)}(\vec x):=\{j\in V: x_{j}> x_{\sigma(i)}\},\;\;\;\; i=1,\cdots,n-1$. We can  write \eqref{eq:Lovasum} in an integral form as
\begin{align}\label{eq:Lovaintegral}
f^{L}(\vec x)&=\int_{\min\limits_{1\le i\le n}x_i}^{\max\limits_{1\le i\le n}x_i} f(V^t(\vec x))d t+f(V)\min_{1\le i\le n}x_i
\end{align}

where $V^t(\vec x)=\{i\in V:  x_i>t\}$. If we apply the   M\"obius transformation, this becomes
 \begin{equation}\label{eq:LovaMobuis} f^{L}(\vec x)=\sum\limits_{A\subset V}\sum\limits_{B\subset A}(-1)^{\#A-\#B}f(B)\bigwedge\limits_{i\in A} x_i,\end{equation}
where $\bigwedge\limits_{i\in A} x_i$ is the minimum over $\{x_i:i\in A\}$.
\end{itemize}
It is easy to see  that $f^L$ is positively one-homogeneous, PL (piecewise linear) and Lipschitz continuous \cite{Lovasz,Bach13}. Also, $f^L(\vec x+t\vec 1_V)=f^L(\vec x)+tf(V)$, $\forall t\in\R$, $\forall \vec x\in\R^V$, {and $f^L(\vec1_A)=f(A)$ for any $A\in\power(V)\setminus\{\varnothing\}$. The definition of $f^L$ does not involve the datum $f(\emptyset)$, and thus by convention, it is natural to reset $f(\emptyset)=0$ to match the equality  $f^L(\vec0)=0$, unless stated otherwise. For convenience, we say that $f:\power(V)\to\R$ is a constant (resp., positive) function if $f$ is constant (resp., positive) on $\power(V)\setminus\{\varnothing\}$}. 
Moreover, a continuous function $F:\R^V\to \R$ is the Lov\'asz extension of some $f:\power(V)\to\R$ if and only if $F(\vec x+\vec y)=F(\vec x)+F(\vec y)$ whenever $(x_i-x_j)(y_i-y_j)\ge0$, $\forall i,j\in V$.


\vspace{0.13cm}

In this paper, we shall use the Lov\'asz extension and its variants to study
the interplay between discrete and continuous aspects in  topics such as  convexity, optimization and spectral  theory.



\vspace{0.13cm}

\textbf{Submodular and convex functions}

Submodular function have emerged as a powerful concept in discrete optimization, see  Fujishige's monograph \cite{F05-book} and {Bach's works \cite{Bach13,Bach19}}.   We also refer the readers to some recent related works regarding  
submodular functions on  hypergraphs  \cite{LM18,LM18-,LHM20}.  We recall
  that a discrete function $f:\A\to \R$  defined on  an algebra $\A\subset\mathcal{P}(V)$ (i.e., $\A$ is  closed under  union and  intersection)  is submodular if $f(A)+f(B)\ge f(A\cup B)+f(A\cap B)$, $\forall A,B\in\A$. The Lov\'asz extension turns a  submodular into a convex function, and we can hence minimize the former by minimizing the latter:
\begin{theorem}[Lov\'asz \cite{Lovasz}]
	$f:\mathcal{P}(V)\to\mathbb{R}$ is submodular if and only if $f^L$ is convex.
\end{theorem}

\begin{center}
	\begin{tikzpicture}[node distance=6cm]
	
	\node (convex) [startstop] {  Submodularity };
	
	\node (submodular) [startstop, right of=convex, xshift=1.6cm]  { Convexity  };
	
	\draw [arrow](convex) --node[anchor=south] { \small Lov\'asz extension } (submodular);
	\draw [arrow](submodular) --node[anchor=north] {  } (convex);
	\end{tikzpicture}
\end{center}

\begin{theorem}[Lov\'asz \cite{Lovasz}]If $f:\mathcal{P}(V)\to\mathbb{R}$ is submodular with $f(\varnothing)=0$, then
	$$\min\limits_{A\subset V}f(A)=\min\limits_{\vec x\in [0,1]^V}f^L(\vec x).$$
\end{theorem}

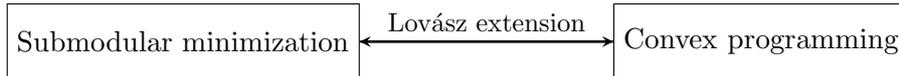
\begin{center}
	\begin{tikzpicture}[node distance=6cm]
	
	\node (convex) [process] {
		Submodular minimization  };
	
	\node (submodular) [process, right of=convex, xshift=1.6cm]  {
		Convex programming   };
	
	\draw [arrow](convex) --node[anchor=south] { \small Lov\'asz extension } (submodular);
	\draw [arrow](submodular) --node[anchor=north] { \small} (convex);
	\end{tikzpicture}
\end{center}

Thus, submodularity can be seen  as some kind of
 `discrete convexity', and that naturally lead to  many generalizations, such as bisubmodular, $k$-submodular, L-convex and M-convex, see \cite{F05-book,Murota03book}.
Moreover, the following classical result characterizes the class of all functions which can be expressed as Lov\'asz extensions of  submodular functions. 
 \begin{theorem}[
  Theorem 7.40 in \cite{Murota03book}]
 	A one-homogeneous function $F:\R^V\to \R$ is a Lov\'asz extension of some submodular function if and only if $F(\vec x+t\vec 1_V)=F(\vec x)+tF(\vec 1_V)$, $\forall t\in\R$, $\forall \vec x\in\R^V$, and $F(\vec x)+F(\vec y)\ge F(\vec x\vee \vec y)+F(\vec x\wedge \vec y)$, where the $i$-th components of $\vec x\vee \vec y$ and $\vec x\wedge \vec y$ are
 	$(\vec x\vee \vec y)_i=\max\{x_i,y_i\}$ and $(\vec x\wedge \vec y)_i=\min\{x_i,y_i\}$.
 \end{theorem}
 One may want to extend such a result to the  bisubmodular  or more general cases. 
 In that direction, we shall obtain some   results such as Proposition \ref{pro:bisubmodular-continuous} and Theorem \ref{thm:submodular-L-equivalent} in Section \ref{sec:SubmodularityConvexity}. It is also worth noting that Bach investigated an interesting generalization of submodular functions by a generalized Lov\'asz extension \cite{Bach19}.

So far,  research has mainly  focused  on `discrete convex' functions, leading to
`Discrete Convex Analysis' \cite{Murota98,Murota03book}, whereas the  discrete non-convex setting which is quite  popular in modern sciences has not yet received that much attention.

\vspace{0.16cm}

\textbf{Non-submodular cases}

Obviously, the non-convex case is so diverse and general that it cannot be directly studied by standard submodular tools.  Although some publications show several results on non-submodular (i.e., non-convex) minimization based on Lov\'asz extension \cite{HS11}, so far, these only work for special minimizations over the whole power set.  Here, we shall find applications for discrete optimization and nonlinear spectral graph theory by employing the multi-way Lov\'asz extension on enlarged and restricted  domains.

\vspace{0.16cm}

In summary, we are going to initiate the study of diverse continuous extensions in non-submodular settings. This paper develops a systematic framework for many aspects around the topic. We  establish a universal discrete-to-continuous framework via multi-way extensions, by systematically utilizing integral representations.  
In \cite{JZ-prepare21}, we establish the links between discrete Morse theory and continuous Morse theory via the original Lov\'asz extension. 
We shall now discuss some connections with other various fields.

\vspace{0.19cm}

\textbf{Connections with  combinatorial optimization}

Because of the wide range of applications of discrete mathematics in computer science,  combinatorial optimization has been much studied from the mathematical perspective.
It is known that any combinatorial optimization can be equivalently expressed as a continuous optimization via convex (or concave) extension, but often, 
there is the difficulty that one cannot  write down an equivalent continuous
object function in closed form. 
For practical purposes, it would be very helpful if one could  transfer a
combinatorial optimization problem to an explicit and simple equivalent
continuous optimization problem in closed form. Formally, in many concrete situations, it would be useful if one could get an identity of the  form
\begin{equation}\label{eq:D-to-C-formal}\min\limits_{(A_1,\cdots,A_k)\in \A\cap \supp(g)}\frac{f(A_1,\cdots,A_k)}{g(A_1,\cdots,A_k)}=\inf\limits_{\psi\in \D(\A)}\frac{\widetilde{f}(\psi)}{\widetilde{g}(\psi)}.\end{equation}
where $f,g:\A\to [0,\infty)$,  $\D(\A)$ is a feasible domain determined by $\A$ only, $\mathrm{supp}(g)$ is the support of $g$, and $\widetilde{f}$ and $\widetilde{g}$ are suitable continuous extensions of $f$ and $g$.

So far, only situations where  $f:\power(V)\to \R$ or $f:\power_2(V)\to\R$ have been investigated systematically \cite{HS11,CSZ18}, and what is lacking are situations with restrictions, that is,  incomplete data. 

Also, to the best of our knowledge, the known results in the literature do not work for  combinatorial optimization directly on set-tuples.  But most of combinatorial optimization problems should be formalized in the form of set-tuples, and only a few can be represented in set form or disjoint-pair form. Whenever one can find an equivalent Lipschitz function for a combinatorial problem in the field of discrete optimization, this makes  useful tools available and leads to new connections.
That is, one wishes to establish   a {\sl discrete-to-continuous transformation} like the operator $\sim$ in \eqref{eq:D-to-C-formal}.
We will show in Section \ref{sec:CC-transfer} that the Lov\'asz extension and its variants  
are suitable choices for such a transformation
(see Theorems \ref{thm:tilde-fg-equal}, \ref{thm:tilde-H-f} and Proposition \ref{pro:fraction-f/g} for details).

\vspace{0.15cm}

To reach these goals, we need to systematically study various generalizations  of the  Lov\'asz extension. More precisely, we shall work with the following  two different multi-way forms:

\begin{enumerate}[(1)]

\item Disjoint-pair version: 
    for a function $f:\power_2(V)\to\R$, its disjoint-pair
  Lov\'asz extension is defined as
\begin{equation}\label{eq:disjoint-pair-Lovasz-def-integral}
f^{L}(\vec x)=\int_0^{\|\vec x\|_\infty} f(V_+^t(\vec x),V_-^t(\vec x))dt,
\end{equation}
where $V_\pm^t(\vec x)=\{i\in V:\pm x_i>t\}$, $\forall t\ge0$. For $\A\subset\power_2(V)$ and $f:\A\to\R$, the feasible domain $\D_\A$ of the disjoint-pair Lov\'asz extension is $\{\vec x\in\R^V:(V_+^t(\vec x),V_-^t(\vec x))\in\A,\forall t\ge0\}$.  We simply use $\vec1 _{A,B}$ to represent the  indicator vector $\vec 1_A-\vec 1_B\in\D_\A$ of the disjoint set-pair $(A,B)\in\A$.

It should be noted that the disjoint-pair Lov\'asz extension introduced by
Qi \cite{Qi88} has been systematically  investigated by Fujishige \cite{Fujishige14,F05-book} and Murota \cite{Murota03book} in the context of  discrete convex analysis (or the theory of submodular functions).  They  defined and investigated the disjoint-pair  Lov\'asz extension in a summation form. 
 The 
integral formulation \eqref{eq:disjoint-pair-Lovasz-def-integral}, however,   is 
more convenient to obtain a closed formula of the equivalent continuous
optimization problem for a combinatorial optimization problem. Moreover, the references and the present paper 
focus on different aspects, with the exception of the submodularity theorem (i.e., $f$ is bisubmodular iff $f^L$ is convex). 

\item $k$-way version:
for a function $f:\mathcal{P}(V)^k\to \R$, the {\sl simple $k$-way Lov\'asz extension} $f^L:\R^{kn}\to \R$ is defined as
\begin{equation}\label{eq:Lovasz-Form-1}
f^L(\vec x^1,\cdots,\vec x^k)=\int_{\min \vec x}^{\max \vec x}f(V^t(\vec x^1),\cdots,V^t(\vec x^k))dt+ f(V,\cdots,V)\min\vec x,
\end{equation}
where $V^t(\vec x^i)=\{j\in V:x^i_j>t\}$, $\min\vec x=\min\limits_{i,j} x^i_j$ and $\max\vec x=\max\limits_{i,j} x^i_j$.  For $\A\subset\power^k(V)$ with $(\varnothing,\cdots,\varnothing),(V,\cdots,V)\in\A$  and $f:\A\to\R$, we take  $\D_\A=\{\vec x\in\R^{kn}_{\ge0}:(V^t(\vec x^1),\cdots,V^t(\vec x^k))\in\A,\forall t\in\R\}$ as a feasible domain of the $k$-way Lov\'asz extension $f^L$.  For convenience, we will simply use $\vec1 _{A_1,\cdots,A_k}$ to represent the  indicator vector $(\vec1_{A_1},\cdots,\vec1_{A_k})\in\D_\A$ of the  set-tuple $(A_1,\cdots,A_k)\in\A$. 

By the Lov\'asz extension of  submodular functions on distributive
  lattices \cite{F05-book,Murota03book},  our $k$-way version
  \eqref{eq:Lovasz-Form-1} can be reduced to the classical version on
  distributive lattices. Our main purposes and key results, however,  are
  different from that approach. In fact, we mainly aim to deal with discrete
  fractional programming by the $k$-way  Lov\'asz extension, while those references concentrate on  submodularity and  convex optimization. 
\end{enumerate}

 All these multi-way Lov\'asz extensions satisfy the optimal identity Eq.~\eqref{eq:D-to-C-formal}:

\begin{introthm}[Theorem \ref{thm:tilde-H-f} and Proposition \ref{pro:fraction-f/g}]\label{thm:tilde-fg-equal}
Given two  functions $f,g:\A\to [0,+\infty)$, let $\tilde{f}$ and $\tilde{g}$ be two real 
functions on $\D_\A$ satisfying $\tilde{f}(\vec1_{A_1,\cdots,A_k})=f(A_1,\cdots,A_k)$ and $\tilde{g}(\vec1_{A_1,\cdots,A_k})=g(A_1,\cdots,A_k)$,   where $\vec1_{A_1,\cdots,A_k}\in\D_\A$ is the indicator vector of the set-tuple $(A_1,\cdots,A_k)\in\A$. Then Eq.~\eqref{eq:D-to-C-formal} holds if $\tilde{f}$ and $\tilde{g}$ further possess  (P1) or (P2) below. Correspondingly, if $\tilde{f}$ and $\tilde{g}$ fulfil (P1') or (P2),
there similarly holds
$$\max\limits_{(A_1,\cdots,A_k)\in \A\cap\supp(g)}\frac{f(A_1,\cdots,A_k)}{g(A_1,\cdots,A_k)}=\sup\limits_{\psi\in \D_\A\cap\supp(\widetilde{g})}\frac{\widetilde{f}(\psi)}{\widetilde{g}(\psi)}.$$
Here the optional additional conditions of $\tilde{f}$ and $\tilde{g}$ are:

(P1) $\tilde{f}\ge f^L$ and $\tilde{g}\le g^L$.\;\;\;  (P1') $\tilde{f}\le f^L$ and $\tilde{g}\ge g^L$.

(P2) 
$\tilde{f}=((f^\alpha)^L)^{\frac1\alpha}$ and $\tilde{g}=((g^\alpha)^L)^{\frac1\alpha}$ for some  $\alpha>0$.

Here $f^L$ is either the original or the disjoint-pair or the $k$-way Lov\'asz extension. 
\end{introthm}

Theorem \ref{thm:tilde-fg-equal}
shows that by the multi-way Lov\'asz extension, the  combinatorial
optimization in quotient form can be transformed to  fractional
programming. And based on this fractional optimization, we propose an
effective local convergence scheme, which relaxes the Dinkelbach-type
iterative scheme and mixes the inverse power method and the steepest descent
method. Furthermore, many other  continuous iterations, such as
Krasnoselski-Mann iteration, and the stochastic subgradient method, could be directly applied here. We refer the readers to \cite{JostZhang} for another development on  equalities between  discrete and continuous optimization problems via various generalizations of Lov\'asz extension.

The power of Theorem \ref{thm:tilde-fg-equal} is embodied in many new examples and applications including  Cheeger-type problems, various
isoperimetric constants and max $k$-cut problems (see Subsections \ref{sec:max-k-cut}, \ref{sec:boundary-graph-1-lap} and \ref{sec:variantCheeger}). 
And moreover, we find that not only combinatorial optimization,
but also some combinatorial invariants like the independence number and the chromatic number, can 
be transformed into a continuous representation by this scheme.

\begin{introthm}[Sections \ref{sec:independent-number} and \ref{sec:chromatic-number}]
\label{thm:graph-numbers}
For an unweighted and undirected simple graph $G=(V,E)$ with $\#V=n$, its independence number can be represented as
$$\alpha(G)=\max\limits_{\vec x\in \R^n\setminus\{\vec 0\}}\frac{\sum\limits_{\{i,j\}\in E}(|x_i-x_j|+|x_i+x_j|)- 2\sum\limits_{i\in V}(\deg_i-1)|x_i|}{2\|\vec x\|_\infty},$$
where $\deg_i=\#\{j\in V:\{j,i\}\in E\}$, $i\in V$, and its chromatic number is
\begin{equation*}
\gamma(G)= n^2-\max\limits_{\vec x\in\R^{n^2}\setminus\{\vec 0\}}\sum\limits_{k\in V}\frac{n\sum\limits_{\{i,j\}\in E}(|x_{ik}-x_{jk}|+|x_{ik}+x_{jk}|)+2n\|\vec x^{,k}\|_{\infty}-2n\deg_k\|\vec x^{,k}\|_1- 2\|\vec x^{k}\|_{\infty}}{2\|\vec x\|_\infty},
\end{equation*}
where $\vec x=(x_{ki})_{k,i\in V}$, $\vec x^{k}=(x_{k1},\cdots,x_{kn})$ and $\vec x^{,k}=(x_{1k},\cdots,x_{nk})
$.
The maximum matching number of $G$ can be expressed as
$$\max\limits_{\vec y\in\R^E\setminus\{\vec 0\}}\frac{\|\vec y\|_1^2}{\|\vec y\|_1^2-2\sum_{e\cap e'=\varnothing}y_ey_{e'}}.$$
\end{introthm}

 \vspace{0.15cm}




There are some equivalent continuous  reformulations of the maxcut problem and
the  independence number of a graph in the literature. However, a continuous  reformulation of the coloring number has not yet been proposed. The main reason seems to be the complexity of coloring a graph. Hence, it is very difficult to discover a continuous form of the coloring number by direct  observation. 

\begin{introthm}[Theorem \ref{thm:tilde-H-f}]\label{thm:tilde-fg-equal-PQ}
Given   functions $f_1,\cdots,f_n:\A\to[0,+\infty)$,and  $p$-homogeneous functions $P,Q:[0,+\infty)^n\to[0,+\infty)$, we have
$$\max\limits_{A\in\A}\frac{P(f_1(A),\cdots,f_n(A))}{Q(f_1(A),\cdots,f_n(A))}=\sup\limits_{x\in\D_\A}\frac{P(f_1^L(\vec x),\cdots,f_n^L(\vec x))}{Q(f_1^L(\vec x),\cdots,f_n^L(\vec x))}$$
if $P^{\frac1p}$ is 
subadditive  and  $Q^{\frac1p}$ is superadditive. One can replace  `max' by `min' if $P^{\frac1p}$ is 
superadditive  and  $Q^{\frac1p}$ is subadditive.
\end{introthm}

Theorems \ref{thm:tilde-fg-equal},  \ref{thm:tilde-fg-equal-PQ} and \ref{thm:tilde-H-f}  can be seen as natural and nontrivial generalizations of the related original works by   Hein and Setzer   \cite{HS11}.

\vspace{0.1cm}

\textbf{Connections with  spectral graph theory}

Spectral graph theory aims to derive properties of a (hyper-)graph from its eigenvalues and eigenvectors. Going beyond the linear case, nonlinear spectral graph theory is developed in terms of discrete geometric analysis and difference equations on (hyper-)graphs.
Every discrete eigenvalue problem can be formulated as a variational problem  for an
objective functional, a Rayleigh-type quotient.
In some cases, this
functional is natural and easy to obtain,
since one may compare the discrete version with its original continuous analog in geometric analysis. However, in other
situations, there is no such analog. Fortunately, we find a unified framework based on multi-way Lov\'asz extension to produce appropriate objective functions from a combinatorial problem (see Sections 
\ref{sec:CC-transfer} and \ref{sec:eigenvalue}). 

More precisely, for a combinatorial problem with a  discrete objective function of the form $\frac{f(A)}{g(A)}$, we might obtain some correspondences by studying the   
set-valued eigenvalue problem
$$
 \nabla f^L (\vec x)\bigcap \lambda\nabla  g^L (\vec x) \ne\varnothing
$$
which is simply called the eigenvalue problem of the function pair $(f^L,g^L)$.  
Hereafter we use $\nabla$ to denote the (Clarke) sub-gradient operator acting on Lipschitz functions.
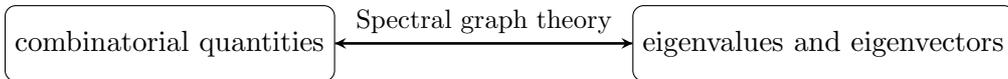
\begin{center}
\begin{tikzpicture}[node distance=6cm]

\node (graph) [startstop] {  combinatorial  quantities };

\node (spectrum) [startstop, right of=convex, xshift=2.6cm]  { eigenvalues and eigenvectors }; 

\draw [arrow](graph) --node[anchor=south] { \small Spectral graph theory} (spectrum);
\draw [arrow](spectrum) --node[anchor=north] { } (graph);
\end{tikzpicture}
\end{center}



\vspace{0.16cm}
We shall consider the following three concepts:
\begin{itemize}
\item {\sl Eigenvectors and eigenvalues}:\;\;
The set-valued eigenvalue problems above are usually written as $\vec0\in \nabla f^L (\vec x)-\lambda\nabla  g^L (\vec x)$ by using the Minkowski summation of convex sets. We call $\lambda$ an eigenvalue and $\vec x$ an eigenvector associated to $\lambda$.

\item {\sl Critical points and critical values}:\;\;
The set of critical points
$\left\{\vec x\left|0\in\nabla \frac{f^L(\vec x)}{g^L(\vec x)}\right.\right\}$
 and the corresponding critical values. 

\item {\sl Minimax critical values} (i.e., {\sl variational eigenvalues in Rayleigh quotient form}):\;\;
The Lusternik-Schnirelman theory tells us that the min-max
values 
\begin{equation}\label{eq:def-c_km}
\lambda_{m}=\inf_{S\in \Gamma_m}\sup\limits_{\vec x \in S}\frac{f^L(\vec x)}{g^L(\vec x)},\;\;m=1,2,\ldots,n,
\end{equation}
are critical values of $f^L(\cdot)/g^L(\cdot)$. Here $\Gamma_m$ is a class of certain topological objects at level $m$, e.g., the family of subsets with Krasnoselskii's $\mathbb{Z}_2$-genus  (or Lusternik-Schnirelman category)  not smaller than $m$.   Since this paper does not focus on the min-max critical values, we will not say more about  Krasnoselskii's $\mathbb{Z}_2$-genus and the class $\Gamma_m$.  We  refer the interested readers to \cite{JostZhang} for systematic studies on this topic.
\end{itemize}

  There are the following relations between these three classes:
$$\{\text{Eigenvalues in Rayleigh quotient}\}\subset\{\text{Critical values}\}\subset\{\text{Eigenvalues}\}.$$
For linear spectral theory, the above three classes  coincide. However, for the non-smooth spectral theory derived by Lov\'asz extension, we only have the inclusion relations.


We have the following result on the eigenvalue problem for the  disjoint-pair Lov\'asz extension, while  for the results on the original Lov\'asz extension, we refer to  Section \ref{sec:eigenvalue} for details. 
\begin{introthm}
\label{introthm:eigenvalue}
Given $f,g:\power_2(V)\to\R$, then every eigenvalue of $(f^L,g^L)$ has an eigenvector of the form $\vec 1_A-\vec1_B$. Moreover, we have the following claims:
\begin{itemize}
\item If $2f(A,B)=f(A,V\setminus A)+f(V\setminus B,B)$ and $2g(A,B)=g(A,V\setminus A)+g(V\setminus B,B)$  for any $(A,B)\in \power_2(V)\setminus\{(\varnothing,\varnothing)\}$, then every eigenvalue of $(f^L,g^L)$ has an eigenvector of the form  $\vec 1_A-\vec1_{V\setminus A}$.
\item If $g=\mathrm{Const}$,  then for any $A\subset V$, $\vec 1_A-\vec1_{V\setminus A}$ is an eigenvector.
\item If $f(A,B)=\hat{f}(A)+\hat{f}(B)$ and $g(A,B)=\hat{g}(A)+\hat{g}(B)$  for some  symmetric function  $\hat{f}:\power(V)\to\R$ (i.e., $\hat{f}(A)=\hat{f}(V\setminus A)$, $\forall A$) and non-decreasing submodular function  $\hat{g}:\power(V)\to\R_+$,  then the second 
eigenvalue $\lambda_2$ of $(f^L,g^L)$ equals 
$$\min\limits_{\vec x\bot\vec 1}\frac{f^L(\vec x)}{\min\limits_{t\in\R}g^L(\vec x-t\vec 1)}=\min\limits_{A\in\mathcal{P}(V)\setminus\{\varnothing,V\}}\frac{\hat{f}(A)}{\min\{\hat{g}(A),\hat{g}(V\setminus A)\}}=\min\limits_{(A,B)\in\mathcal{P}_2(V)\setminus\{(\varnothing,\varnothing )\}}\max\{\frac{\hat{f}(A)}{\hat{g}(A)},\frac{\hat{f}(B)}{\hat{g}(B)}\}.$$
\end{itemize}

\end{introthm}

This generalizes  recent results on the 
graph 1-Laplacian and Cheeger's constant \cite{HeinBuhler2010,TVhyper-13,Chang16,CSZ15,CSZ17}. { And as a new application, we show  that the   min-cut problem and the   max-cut problem are equivalent to solving  the 
smallest  nontrivial (i.e., the second) eigenvalue and the largest eigenvalue of a certain nonlinear eigenvalue problem (see Theorem \ref{thm:mincut-maxcut-eigen}).
}

\vspace{0.16cm}

\textbf{Applications to   frustration in signed network }

As a key measure for analysing signed networks, the frustration index on a signed graph  quantifies  how far a signature  is
from being balanced (see Section \ref{sec:frustration}). Computing the frustration index  is NP-hard, and few algorithms have been proposed  \cite{ArefWilson19,ArefMasonWilson20}.

Considering a signed graph $(V,E_+\cup E_-)$ with  $E_+$ (resp. $E_-$) 
the set of positive (resp. negative)  edges,   based on the  disjoint-pair Lov\'asz  extension, we obtain an equivalent continuous optimization of the frustration index (or the line index of balance \cite{Harary59}):
$$\#E_-+\min\limits_{ x\ne0}\frac{\sum_{\{i,j\}\in E_+}|x_i-x_j|-\sum_{\{i,j\}\in E_-}|x_i-x_j|}{2\|\vec x\|_\infty}.$$
This new reformulation can be computed via typical algorithms in continuous optimization. 

Also, we propose the  eigenvalue problem
\begin{equation}\label{eq:frustration-eigen}
\nabla\left(\sum_{\{i,j\}\in E_+}|x_i-x_j|+\sum_{\{i,j\}\in E_-}|x_i+x_j|\right)\bigcap\lambda \nabla\|\vec x\|_\infty \ne\varnothing    
\end{equation}
and we show an iterative scheme for searching the frustration index based on the smallest eigenvalue of the nonlinear eigenvalue problem \eqref{eq:frustration-eigen}. See Section \ref{sec:frustration} for details and  more results.  

\vspace{0.16cm}

 Since the  transformation of a combinatorial optimization to a
continuous optimization or a nonsmooth  eigenvalue problem usually leads
  to a quotient, 
the task for  fractional programming then becomes to compute  an optimal value or an eigenvector. In Section \ref{sec:algo}, we present a general  algorithm which is available to compute the resulting continuous reformulations arising in Theorems \ref{thm:tilde-fg-equal}, \ref{thm:graph-numbers}, \ref{thm:tilde-fg-equal-PQ} and \ref{introthm:eigenvalue}. 

 In another paper \cite{JostZhang}, we present  a systematic study of  general function pairs $(F,G)$, in which $F$ and $G$ can be piecewise multilinear   or other general extensions of certain discrete functions. The papers can be read independently of each other.

 In summary, we present a systematic study for  constructing  nonlinear
eigenvalue problems and equivalent continuous  reformulations for
combinatorial quantities, which capture the key properties of the original   combinatorial problems. This is helpful to increase understanding of certain combinatorial problems by the corresponding  eigenvalue problems and the  equivalent continuous  reformulations. The following picture summarizes the relations between the various concepts developed and studied in this paper.

\begin{figure}[H]
\centering
\begin{tikzpicture}[node distance=4.5cm]

\node (CQ) [startstop] {\begin{tabular}{l}
Combinatorial\\ Quantities
\end{tabular}};

\node (DO) [process, right of=CQ, xshift=-1cm]  { \begin{tabular}{l}
 Discrete\\
 Optimization
\end{tabular}};

\node (CO) [startstop, right of=DO, yshift=0cm, xshift=2.6cm] {\begin{tabular}{l}
 Continuous\\
 Optimization
\end{tabular} };
\node (DM) [io1, below of=CQ, xshift=0.1cm,yshift=2cm] {\begin{tabular}{l}
 Discrete Morse theory
\end{tabular}
 };
\node (F/G) [process, right of=DM, xshift=1.9cm] {\begin{tabular}{l}
(topological) Morse theory    \\
  (metric) critical point theory
\end{tabular}
};
\node (FG) [startstop, below of=F/G, yshift=2.6cm,xshift=3.3cm] {
\begin{tabular}{r}
(Nonlinear) Spectral theory\\
\end{tabular}};

\node (algorithm) [io1, right of=F/G, xshift=2cm] {\begin{tabular}{l}
 Continuous\\
 Programming\\
\& Algorithm
\end{tabular}};
\node (submodular) [io2, above of=DO, yshift=-2.6cm] {Submodularity};
\node (convexity) [io2, right of=submodular, xshift=5.2cm] {Convexity};
\draw [arrow](CQ) --node[anchor=south] { } (DO);
\draw [arrow](DO) --node[anchor=south] { } (CQ);
\draw [arrow](DO) --node[anchor=south]{
\small Section  \ref{sec:CC-transfer} }node[anchor=north] {  \cite{JostZhang}} (CO);
\draw [arrow](DM) -- node[anchor=south] {}node[anchor=north] {\cite{JZ-prepare21}} (F/G);
\draw [arrow](FG) -- node[anchor=south] {} (algorithm);
\draw [arrow](F/G) -- node[anchor=south] {  } (FG);
\draw [arrow](CO) -- node[anchor=north] {  \small Section \ref{sec:algo}} (algorithm); 
\draw [arrow](CO) -- node[anchor=south] {  } (F/G);
\draw [arrow](CO) -- node[anchor=north] {  \small Section \ref{sec:eigenvalue}} (FG);
\draw [arrow](submodular) -- node[anchor=south] {  Lov\'asz extension } (convexity);
\draw [arrow](submodular) -- node[anchor=south] { } (DO);
\draw [arrow](convexity) -- node[anchor=south] { } (submodular);
\draw [arrow](convexity) -- node[anchor=south] { } (algorithm);
\end{tikzpicture}
\caption{\label{fig:flowchart} The relationship between the aspects 
studied in our work.}
\end{figure}
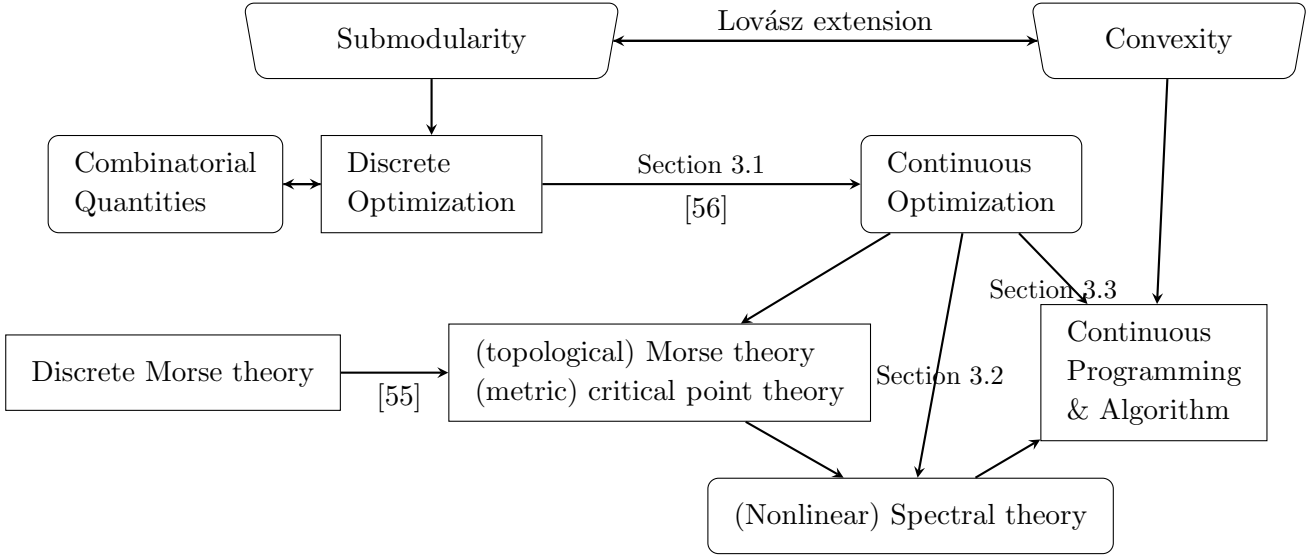
{We shall now  brief discuss how to apply this scheme.  
Our framework gives new  continuous formulations and eigenvalue 
representations for certain combinatorial optimization and related  discrete quantities.  Compared  to other  formulations of those combinatorial  problems, the main advantage of our   formulation is that the critical data (including min-max data, saddle points, and optimal values
) 
of the continuous representations 
 incorporate all the  key information of the original  combinatorial problems (see Sections \ref{sec:CC-transfer}
 and \ref{sec:eigenvalue}). 
For example, by the results in \cite{DPT22}, the $k$-way Cheeger constant on a tree graph  agrees with the $k$-th eigenvalue of the graph 1-Laplacian,  which can be   subsumed into  the above framework. 

Restricted onto  optimization problems, the  continuous representation obtained by Lov\'asz-type extension  leads to an iterative 
algorithm based on  
fractional programming, but we should point out that this is not the main 
focus of the present paper. 

Although the associated  algorithms are  not the main contribution and focus of this work, in Section \ref{sec:algo} we review  fractional programming and  explore more in this direction. A remarkable theoretical advantage we proved in this paper  is that our scheme  provides  an iterative solution  {\sl without}  rounding, and can be used  to improve any  initially given data.   Moreover, just to explain  the applicability, 
we should point out that this framework already  performs well on the Cheeger cut problem (see Sections \ref{sec:boundary-graph-1-lap} and \ref{sec:variantCheeger}), and the maxcut problem (see Section \ref{sec:maxcut}
 for details). One can expect a good  performance of this framework also on other combinatorial problems, such as the frustration set problem, the independent set problem and  the coloring problem.  

\vspace{0.16cm}

\noindent\textbf{Related works}.\;   The present paper is the second one in a series that  develops a systematic
  bridge between constructions in  discrete mathematics and the corresponding   continuous analogs via Lov\'asz type extensions,  
where the other two parts \cite{JZ-prepare21,JostZhang} are concerned with different aspects. Let us  briefly describe their contents and put them into perspective.  The series is motivated by recent developments on  Cheeger inequalities,  Lov\'asz extensions, expander graphs, spectral graph theory and practical applications.  
We  focus on the Lov\'asz extension and introduce some useful generalizations, including the multi-way Lov\'asz extension in this paper, and the piecewise multilinear extension in \cite{JostZhang}, which we simply  call  discrete-to-continuous extensions. 
Then, we investigate    optimization and eigenvalue problems (see Section \ref{sec:main}), Morse theory (see \cite{JZ-prepare21}), min-max theory, critical point theory,  and   spectral theory (see \cite{JostZhang}) for the Lipschitz  functions obtained by these discrete-to-continuous extensions. 
Thus, this series  provides new perspectives for understanding certain relations and interactions between discrete and continuous worlds via Lov\'asz-type  extensions. 

The present paper focuses on the aspect of eigenvalue problems and optimizations regarding Lov\'asz extension, while in \cite{JZ-prepare21} we concentrate on the Morse  and Lusternik-Schnirelman  theoretical aspect involving Lov\'asz extension. More generally, in  \cite{JostZhang}, we further explore min-max relations, saddle point problems, spectral theory and  critical point theory involving a more general class of discrete-to-continuous  extensions (namely, the piecewise multilinear extensions). The mixed IP-SD  algorithm  proposed in Section \ref{sec:algo} 
can also  be applied in 
\cite{JostZhang} for approximating the second eigenvalue.  
}
\begin{notification}
Since this paper contains  many interacting parts and relevant results, some notions and concepts may have slightly distinct
meanings in different sections, but this will be stated at the beginning of each section.
\end{notification}

\section{A preliminary:  Lov\'asz extension and submodular functions}
\label{sec:Lovasz}

While most of the results on submodularity are known in the field of discrete
convex analysis, we present some details in a simple manner, which should be helpful to understand our main results in Section \ref{sec:main}.

We first formalize some important results about the original Lov\'asz extension.

\begin{defn}
Two vectors $\vec x$ and $\vec y$ are {\sl comonotonic} if $(x_i-x_j)(y_i-y_j)\ge0$, $\forall i,j\in \{1,2,\cdots,n\}$.

A function $F:\R^n\to\R$ is  {\sl comonotonic additive} if $F(\vec x+\vec y)=F(\vec x)+F(\vec y)$ for any comonotonic pair $\vec x$ and $\vec y$.
\end{defn}
The following proposition shows that a function is comonotonic additive if and only if it can be expressed as the Lov\'asz extension of some function.
\begin{pro}
\label{pro:comonotonic-additivity}
$F:\R^n\to\R$ is the Lov\'asz extension $F=f^L$ of some function $f:\power(V)\to\R$ if and only if $F$ is comonotonic additive. 
\end{pro}

Recall the following known results:

\begin{theorem}[Lov\'asz]\label{thm:Lovasz}
The following conditions are equivalent: (1) $f$ is submodular; (2) $f^L$ is convex; (3) $f^L$ is submodular.
\end{theorem}

\begin{remark}
The fact that $f$ is submodular if and only if  $f^L$ is submodular is provided by Propositions 7.38 and 7.39 in \cite{Murota03book}. We shall give a detailed proof for a generalized version of Theorem \ref{thm:Lovasz} (see Theorem \ref{thm:submodular-L-equivalent}).   
\end{remark}

\begin{theorem}[
Murota \cite{Murota03book}]\label{thm:Chateauneuf-Cornet}
$F:\R^n\to\R$ is the Lov\'asz extension $F=f^L$ of some submodular $f:\power(V)\to\R$ if and only if $F$ is positively one-homogeneous, submodular and $F(\vec x+t\vec 1)=F(\vec x)+tF(\vec 1)$.
\end{theorem}

\begin{remark}
Theorem \ref{thm:Chateauneuf-Cornet} was   originally  proved by establishing a  one-to-one correspondence between positively  
homogeneous L-convex functions and submodular  functions  (see 
  Theorem 7.40 in Murota's book  \cite{Murota03book}). An alternative proof is
  given in   \cite{CC18}.
\end{remark}

We shall  establish these results for  the disjoint-pair version and the $k$-way version of the Lov\'asz extension.

\subsection{Disjoint-pair and $k$-way Lov\'asz extensions}

Under the natural additional assumption that $f(\varnothing,\varnothing)=0$, one can  write \eqref{eq:disjoint-pair-Lovasz-def-integral} as
\begin{equation}\label{eq:disjoint-pair-Lovasz-def-integral2}
f^{L}(\vec x)=\int_0^{\infty} f(V_+^t(\vec x),V_-^t(\vec x))dt,
\end{equation}
where $V_\pm^t(\vec x)=\{i\in V:\pm x_i>t\}$, $\forall t\ge0$. 
Another formulation of \eqref{eq:disjoint-pair-Lovasz-def-integral2} (or \eqref{eq:disjoint-pair-Lovasz-def-integral}) is 
\begin{equation}\label{eq:disjoint-pair-Lovasz-def}
f^{L}(\vec x)=\sum_{i=0}^{n-1}(|x_{\sigma(i+1)}|-|x_{\sigma(i)}|)f(V^{\sigma(i)}_+(\vec x),V^{\sigma(i)}_-(\vec x)),
\end{equation}
where $\sigma:V\cup\{0\}\to V\cup\{0\}$ is a bijection such that $|x_{\sigma(1)}|\le |x_{\sigma(2)}| \le \cdots\le |x_{\sigma(n)}|$ and $\sigma(0)=0$, where $x_0:=0$, and
$$V^{\sigma(i)}_\pm(\vec x):=\{j\in V:\pm x_j> |x_{\sigma(i)}|\},\;\;\;\; i=0,1,\cdots,n-1.$$
 In fact, by $f(\varnothing,\varnothing)=0$, $\|\vec x\|_\infty=|x_{\sigma(n)}|$,  and $f(V_+^t(\vec x),V_-^t(\vec x))=f(V^{\sigma(i)}_+(\vec x),V^{\sigma(i)}_-(\vec x))$ whenever  $|x_{\sigma(i)}|\le t<|x_{\sigma(i+1)}|$, we have 
$$\int_0^{\infty} f(V_+^t(\vec x),V_-^t(\vec x))dt=\int_0^{\|\vec x\|_\infty} f(V_+^t(\vec x),V_-^t(\vec x))dt=\sum_{i=0}^{n-1}\int_{|x_{\sigma(i)}|}^{|x_{\sigma(i+1)}|}f(V_+^t(\vec x),V_-^t(\vec x))dt$$
which deduces that \eqref{eq:disjoint-pair-Lovasz-def-integral2}, \eqref{eq:disjoint-pair-Lovasz-def-integral} and \eqref{eq:disjoint-pair-Lovasz-def} are equivalent.   
We regard $\power_2(V)=3^V$ as $\{-1,0,1\}^n$ by identifying the disjoint pair $(A,B)$ with the ternary (indicator) vector $\vec 1_A-\vec1_B$.

One may compare the original and the disjoint-pair Lov\'asz extensions by writing \eqref{eq:disjoint-pair-Lovasz-def-integral} as
\begin{equation}\label{eq:disjoint-pair-form}
\int_{\min_i |x_i|}^{\max_i |x_i|} f(V_+^t(\vec x),V_-^t(\vec x))dt+\min_i |x_i| f(V_+^0(\vec x)),V_-^0(\vec x))),
\end{equation}
Note that \eqref{eq:disjoint-pair-form} is very similar to \eqref{eq:Lovaintegral}.  We say that  $(A,B)\in\power_2(V)$ is an {\sl associate set-tuple} of a given $\vec x\in\R^n$ if $(A,B)=(V_+^t(\vec x),V_-^t(\vec x))$ for some $t\ge0$. Of course, a vector $\vec x$ may have many associate  set-tuples.  

\begin{defn}\label{def:k-way-Lovasz}
Given $V_i=\{1,\cdots,n_i\}$, $i=1,\cdots,k$, and a function $f:\mathcal{P}(V_1)\times \cdots\times \mathcal{P}(V_k)\to \R$, the  $k$-way Lov\'asz extension $f^L: \R^{V_1}\times\cdots\times \R^{V_k}\to \R$ can be written as
\begin{align*}
f^L(\vec x^1,\cdots,\vec x^k)&=\int_{\min \vec x}^{\max \vec x}f(V^t_1(\vec x^1),\cdots,V^t_k(\vec x^k))dt+ f(V_1,\cdots,V_k)\min\vec x\\&
=\int_{-\infty}^0(f(V^t_1(\vec x^1),\cdots,V^t_k(\vec x^k))-f(V_1,\cdots,V_k))dt + \int_0^{+\infty}f(V^t_1(\vec x^1),\cdots,V^t_k(\vec x^k)) d t
\end{align*}
where $V^t_i(\vec x^i)=\{j\in V_i:x^i_j>t\}$, $\min\vec x=\min\limits_{i,j} x^i_j$ and $\max\vec x=\max\limits_{i,j} x^i_j$. 

 We say that  $(A_1,\cdots,A_k)\in\mathcal{P}(V_1)\times \cdots\times \mathcal{P}(V_k)$ is an {\sl associated set-tuple} of a given vector $\vec x:=(\vec x^1,\cdots,\vec x^k)\in\R^{V_1}\times\cdots\times \R^{V_k}$ if $(A_1,\cdots,A_k)=(V^t_1(\vec x^1),\cdots,V^t_k(\vec x^k))$ for some $t\in\R$. 
\end{defn}

\begin{defn}[$k$-way analog for disjoint-pair Lov\'asz extension]
\label{def:k-way-pair-Lovasz}
   Given  $V_i=\{1,\cdots,n_i\}$, $i=1,\cdots,k$, and a function $f:\mathcal{P}_2(V_1)\times \cdots\times \mathcal{P}_2(V_k)\to \R$, define $f^L: \R^{V_1}\times\cdots\times \R^{V_k}\to \R$ by
   $$f^L(\vec x^1,\cdots,\vec x^k)=
   \int_0^{\|\vec x\|_\infty} f(V_{1,+}^t(\vec x^1),V_{1,-}^t(\vec x^1),\cdots,V_{k,+}^t(\vec x^k),V_{k,-}^t(\vec x^k))dt
   $$
   where $V_{i,\pm}^t(\vec x^i)=\{j\in V_i:\pm x^i_j>t\}$, $\|\vec x\|_\infty=\max\limits_{i=1,\cdots,k} \|\vec x^i\|_\infty$. We can replace $\|\vec x\|_\infty$ by $+\infty$ if we set $f(\varnothing,\cdots,\varnothing)=0$.  A set-tuple  $(A^1_+,A^1_-\cdots,A^k_+,A^k_-)\in\mathcal{P}_2(V_1)\times \cdots\times \mathcal{P}_2(V_k)$ is called an {\sl associated set-tuple} of a given vector $\vec x:=(\vec x^1,\cdots,\vec x^k)\in\R^{V_1}\times\cdots\times \R^{V_k}$ if $(A^1_+,A^1_-\cdots,A^k_+,A^k_-)=(V_{1,+}^t(\vec x^1),V_{1,-}^t(\vec x^1),\cdots,V_{k,+}^t(\vec x^k),V_{k,-}^t(\vec x^k))$ for some $t\ge0$. 
\end{defn}
\vspace{0.19cm}

 For convenience, we always use $f^L$ to express different variants of Lov\'asz extensions of $f$. The reader can identify the  version we are referring to by the domain of $f$.  

Some basic properties of the multi-way Lov\'asz extension are shown below. 
\begin{pro}\label{pro:multi-way-property}
For the multi-way Lov\'asz extension $f^L(\vec x)$, we have
\begin{enumerate}[(a)]
\item  $f^L(\cdot)$ is positively one-homogeneous, piecewise linear, and Lipschitz continuous.
\item $(\lambda f)^L=\lambda f^L$, $\forall\lambda\in\R$.

\end{enumerate}
\end{pro}

\begin{pro}\label{pro:setpair-property}
For the disjoint-pair Lov\'asz extension $f^L(\vec x)$, we have
\begin{enumerate}[(a)]
\item $f^L$ is Lipschitz continuous, and $|f^L(x)-f^L(y)|\le 2\max\limits_{(A,B)\in \power_2(V)}f(A,B) \|x-y\|_1$, $\forall x,y\in \mathbb{R}^n$. Also,  $|f^L(x)-f^L(y)|\le 2\sum\limits_{(A,B)\in \power_2(V)}f(A,B)  \|x-y\|_\infty$, $\forall x,y\in \mathbb{R}^n$.
\item $f^L(-\vec x)=\pm f^L(\vec x)$, $\forall \vec x\in\R^V$ if and only if $f(A,B)=\pm f(B,A)$, $\forall (A,B)\in \power_2(V)$.
\item\label{pro:pro-c} $f^L(\vec x+\vec y)=f^L(\vec x)+f^L(\vec y)$ whenever $V_\pm^0(\vec y)\subset V_\pm^0(\widetilde{\vec x})$, where $\widetilde{\vec x}$ has components $\widetilde{ x}_i=\begin{cases}x_i,&\text{ if }|x_i|=\|\vec x\|_\infty,\\ 0,&\text{ otherwise}.\end{cases}$
\end{enumerate}
\end{pro}
\begin{proof} (a) and (b) are actually known results and their proofs are elementary.  We refer to Theorem 2.2 and its proof in \cite{CSZ18}  for (a). While, for (b), see Proposition 2.5  in \cite{CSZ18}.  
(c) can be derived from the definition \eqref{eq:disjoint-pair-Lovasz-def}.
\end{proof}

 Here we omit  the proofs of Propositions \ref{pro:multi-way-property} and \ref{pro:setpair-property} (c) because they are easy and similar to the case of the original Lov\'asz extension.

\begin{defn}
\label{def:associate-piece}
Two vectors $\vec x$ and $\vec y$ are said to be absolutely comonotonic  
if $x_iy_i\ge0$, $\forall i$, and $(|x_i|-|x_j|)(|y_i|-|y_j|)\ge0$, $\forall i,j$.
\end{defn}

\begin{pro}\label{pro:setpair-character}
A continuous function $F$ is a disjoint-pair Lov\'asz extension of some function $f:\power_2(V)\to\R$, if and only if 
$F(\vec x)+F(\vec y)=F(\vec x+\vec  y)$ whenever $\vec x$ and $\vec y$ are absolutely comonotonic. 
\end{pro}

\begin{proof} By the definition of the disjoint-pair Lov\'asz extension (see \eqref{eq:disjoint-pair-Lovasz-def}), 
we know that $F$ is a disjoint-pair Lov\'asz extension of some function $f:\power_2(V)\to\R$ if and only if
$\lambda F(\vec x)+(1-\lambda)F(\vec y)=F(\lambda\vec x+(1-\lambda)\vec  y)$ for all absolutely comonotonic 
vectors $\vec x$ and $\vec y$, $\forall \lambda\in[0,1]$. Therefore, we only need to prove the sufficiency part.

For $\vec x\in\R^V$, since $s\vec x$ and $t\vec x$ with $s,t\ge 0$ are absolutely comonotonic, 
 $F(s\vec x)+F(t\vec x)=F((s+t)\vec x)$, which yields a Cauchy equation on the half-line. Thus the continuity assumption implies the linearity of $F$ on the ray $\R^+\vec x$, which implies the property $F(t\vec x)=tF(\vec x)$, $\forall t\ge 0$, and hence $\lambda F(\vec x)+(1-\lambda)F(\vec y)=F(\lambda\vec x+(1-\lambda)\vec  y)$ for any absolutely comonotonic 
 vectors $\vec x$ and $\vec y$, $\forall \lambda\in[0,1]$. This completes the proof.
\end{proof}

For relations between the original and the disjoint-pair Lov\'asz extensions, we further have 
\begin{pro}\label{pro:setpair-original}
 For $h:\power(V)\to [0,+\infty)$ with $h(\varnothing)=0$, and $f:\power_2(V)\to [0,+\infty)$ with $f(\varnothing,\varnothing)=0$ \footnote{In fact, if $h(\varnothing)\ne 0$ or $f(\varnothing,\varnothing)\ne0$, one may change the value and it does not affect the related  Lov\'asz extension.}, we have:
\begin{enumerate}[(a)]
\item If $f(A,B)=h(A)+h(V\setminus B)-h(V)$, $\forall (A,B)\in \power_2(V)$, then   $f^L=h^L$.
\item  If $f(A,B)=h(A)+h(B)$ and $h(A)=h(V\setminus A)$, $\forall (A,B)\in \power_2(V)$, then $f^L=h^L$.
\item  If $f(A,B)=h(A)$, $\forall (A,B)\in \power_2(V)$, then $f^L(\vec x)=h^L(\vec x)$, $\forall \vec x\in[0,\infty)^V$.
\item If $f(A,B)=h(A\cup B)$, $\forall (A,B)\in \power_2(V)$, then  $f^L(\vec x)=h^L(\vec x^++\vec x^-)$.
\item  If $f(A,B)=h(A)\pm h(B)$, $\forall (A,B)\in \power_2(V)$, then $f^L(\vec x)=h^L(\vec x^+)\pm h^L(\vec x^-)$.
\end{enumerate}
Here $\vec x^\pm:=(\pm \vec x)\vee \vec0$.
\end{pro}

\begin{proof}
\begin{enumerate}[(a)]
\item Note that
\begin{align*}
f^L(\vec x)&=\int_0^{\|\vec x\|_\infty}f(V_+^t(\vec x),V_-^t(\vec x))dt
=\int_0^{\|\vec x\|_\infty}(h(V^t(\vec x))+h(V^{-t}(\vec x))-h(V))dt\\
&=\int_{-\|\vec x\|_\infty}^{\|\vec x\|_\infty}h(V^t(\vec x))dt- \|\vec x\|_\infty h(V)
=\int_{x_{\sigma(1)}}^{x_{\sigma(n)}}h(V^t(\vec x))dt +x_{\sigma(1)} h(V) = h^L(\vec x),
\end{align*}
where we use $\|x\|_\infty=\max\{-x_{\sigma(1)},x_{\sigma(n)}\}$ and $h(\varnothing)=0$.
\item This is a direct consequence of (a) since  $h(V)=h(\varnothing)=0$ and $h(B)=h(V\setminus B)$.
\item For any $\vec x\in \R^V$ with $x_i\ge0$, we note that $f^L(\vec x)=\int_0^{\|\vec x\|_\infty}h(V_+^t(\vec x))dt=\int_0^{\max x_i}h(V^t(\vec x))dt=\int_{\min x_i}^{\max x_i}h(V^t(\vec x))dt+\min x_i h(V)=h^L(\vec x)$.
\item 
Similar to (c), one can check that $f^L(\vec x)=h^L(\vec x^++\vec x^-)$.
\item It is straightforward.
\end{enumerate}
\end{proof}

In the sequel, we will not distinguish the original and the disjoint-pair Lov\'asz extensions, since the reader can infer it
from the domains ($\power(V)$ or $\power_2(V)$). Sometime we work on $\power(V)$ only, and in this situation, the disjoint-pair Lov\'asz extension acts on the redefined  $f(A,B)=h(A\cup B)$ as Proposition \ref{pro:setpair-original} states.

The next result is useful for the application on graph coloring.

\begin{pro}\label{pro:separable-summation}
For the  simple $k$-way Lov\'asz extension of $f:\power(V_1)\times\cdots\times\power(V_k)\to \R$ with the separable summation form $f(A_1,\cdots,A_k):=\sum_{i=1}^kf_i(A_i)$, $\forall (A_1,\cdots,A_k)\in \power(V)^k$, we have $f^L(\vec x^1,\cdots,\vec x^k)=\sum_{i=1}^kf_i^L(\vec x^i)$, $\forall (\vec x^1,\cdots,\vec x^k)$.

For $f:\power_2(V_1)\times\cdots \times\power_2(V_k)\to \R$ with the form $f(A_1,B_1\cdots,A_k,B_k):=\sum_{i=1}^kf_i(A_i,B_i)$, $\forall (A_1,B_1,\cdots,A_k,B_k)\in \power_2(V_1)\times\cdots \times \power_2(V_k)$, there similarly holds $f^L(\vec x^1,\cdots,\vec x^k)=\sum_{i=1}^kf_i^L(\vec x^i)$.
\end{pro}

\subsection{Submodularity and Convexity}\label{sec:SubmodularityConvexity}
In this subsection, we give new analogs of Theorems \ref{thm:Lovasz} and \ref{thm:Chateauneuf-Cornet} for the disjoint-pair Lov\'asz extension and the $k$-way Lov\'asz extension. The major difference to existing results in the literature is that we work with the restricted or the enlarged domain of a function.

Let's first recall the standard concepts of submodularity:
\begin{enumerate}[({S}1)]
\item A discrete function $f:\A\to \R$   is submodular if $f(A)+f(B)\ge f(A\cup B)+f(A\cap B)$, $\forall A,B\in\A$, where $\A\subset\mathcal{P}(V)$ is an algebra  (i.e., $\A$ is  closed under  union and intersection).
\item A continuous function $F:\R^n\to \R$ is submodular if
 $F(\vec x)+F(\vec y)\ge F(\vec x\vee \vec y)+F(\vec x\wedge \vec y)$, where $(\vec x\vee \vec y)_i=\max\{x_i,y_i\}$ and $(\vec x\wedge \vec y)_i=\min\{x_i,y_i\}$, $i=1,\cdots,n$. For a sublattice $\D\subset\R^n$ that  is closed under $\vee$ and $\wedge$, one can define  submodularity in the same way.
\end{enumerate}

\begin{notification} The discussion about algebras of sets can be reduced to lattices.
Classical submodular functions on a sublattice of the Boolean lattice
$\{0,1\}^n$ and their continuous versions on $\R^n$ are presented in (S1) and
(S2), respectively. Bisubmodular functions on a graded sub-poset (partially
ordered set)  of   $\{-1,0,1\}^n$ are defined in \eqref{eq:2-submodular} below.
\end{notification}

    Now, we 
    recall the concept of bisubmodularity and introduce its continuous version.
    \begin{enumerate}
\item[(BS1)] A discrete function $f:\mathcal{P}_2(V)\to \R$   is bisubmodular if $\forall\,(A,B),(C,D)\in \power_2(V)$
\begin{equation}\label{eq:2-submodular}
f(A,B)+f(C,D)\ge f((A\cup C)\setminus (B\cup D),(B\cup D)\setminus(A\cup C))+f(A\cap C,B\cap D).
\end{equation}
One can denote $A\vee B=((A_1\cup B_1)\setminus (A_2\cup B_2),(A_2\cup B_2)\setminus (A_1\cup B_1))$ and $A\wedge B=(A_1\cap B_1,A_2\cap B_2)$, where $A=(A_1,A_2)$, $B=(B_1,B_2)$. For a subset  $\A\subset \mathcal{P}_2(V)$ that is closed under $\vee$ and $\wedge$, the bisubmodularity of $f:\A\to\R$ can be expressed as $f(A)+f(B)\ge f(A\vee B)+f(A\wedge B)$, $\forall A,B\in\A$.
\end{enumerate}

If we were to continue the definition of submodularity stated in (S2), we would obtain nothing new.  Hence, the proof of Theorem \ref{thm:Chateauneuf-Cornet} cannot directly apply to our situation. To overcome this issue, we need to provide a matched definition of bisubmodularity for functions on $\R^n$, and an appropriate and careful modification of the translation linearity
condition.

\begin{enumerate}
 \item[(BS2)] A continuous function $F:\R^n\to \R$ is bisubmodular if
 $F(x)+F(y)\ge F(x\vee y)+F(x\wedge y)$, where
 $$(x \vee y)_i=\begin{cases}\max\{x_i,y_i\},&\text{ if } x_i,y_i\ge0,\\
  \min\{x_i,y_i\},&\text{ if } x_i,y_i\le0, \\
  0,&\text{ if } x_iy_i<0,\end{cases}\;\;\;\;\;\;\;(x \wedge y)_i=\begin{cases}\min\{x_i,y_i\},&\text{ if } x_i,y_i\ge0,\\
  \max\{x_i,y_i\},&\text{ if } x_i,y_i\le0, \\
  0,&\text{ if } x_iy_i<0.\end{cases}$$
\end{enumerate}

Henceforth, we simply use $\vec 1_{A,B}$ to denote the vector  $\vec 1_A-\vec 1_B$, where $A,B\subset V$.

\begin{pro}\label{pro:bisubmodular-continuous}
A function $F:\R^V\to \R$ is a disjoint-pair Lov\'asz extension of a bisubmodular function if and only if $F$ is (continuously) bisubmodular (in the sense of (BS2)) and for any $\vec x\in\R^V ,\,t\ge0$, 
{ \linespread{0.95} \begin{enumerate}
\item[$\mathrm{(I)}$] $F(t\vec x)=tF(\vec x)$ (positive homogeneity);
\item[$\mathrm{(II)}$] $F(\vec x+t\vec 1_{V_+,V_-})\ge F(\vec x)+F(t\vec 1_{V_+,V_-})$ for some\footnote{This is some kind of `translation linearity' if we adopt the assumption $F(\vec x+t\vec 1_{V_+,V_-})= F(\vec x)+F(t\vec 1_{V_+,V_-})$. } $V_\pm\supset V_\pm^0(\vec x)$ with $V_+\cup V_-=V$.
\end{enumerate} }
\end{pro}

The proof is a modification of the previous  version on the original Lov\'asz extension for  submodular functions. 

\begin{proof}

We  focus on the ``if'' part. Take the discrete function $f$ defined as $f(A_1,A_2)=F(\vec 1_{A_1,A_2})$. One can check the bisubmodularity of $f$ directly. Fix an $\vec x\in\R^n$ and let $\sigma:V\cup\{0\}\to V\cup\{0\}$ be a bijection such that $|x_{\sigma(1)}|\le |x_{\sigma(2)}| \le \cdots\le |x_{\sigma(n)}|$ and $\sigma(0)=0$, where $x_0:=0$, and
$$V^{\sigma(i)}_\pm=V^{\sigma(i)}_\pm(\vec x):=\{j\in V:\pm x_j> |x_{\sigma(i)}|\},\;\;\;\; i=0,1,\cdots,n-1.$$
Also, we denote $\vec x_{V^{\sigma(i)}_+,V^{\sigma(i)}_-}=\vec x * \vec 1_{V^{\sigma(i)}_+\cup V^{\sigma(i)}_-}$ (i.e., the restriction of $\vec x$ onto $V^{\sigma(i)}_+\cup V^{\sigma(i)}_-$,  with other components $0$), where $\vec x*\vec y:=(x_1y_1,\cdots,x_ny_n)$.

For simplicity, in the following formulas, we identify $\sigma(i)$ with $i$ for all $i=0,\cdots,n$.

It follows from $|x_{i+1}|\vec 1_{V^{i}_+,V^{i}_-}\bigvee \vec x_{V^{i+1}_+,V^{i+1}_-}=\vec x_{V^{i}_+,V^{i}_-}$ and
$$|x_{i+1}|\vec 1_{V^{i}_+,V^{i}_-}\bigwedge \vec x_{V^{i+1}_+,V^{i+1}_-}= |x_{i+1}|\vec 1_{V^{i+1}_+,V^{i+1}_-}$$ that
\begin{align*}
f^{L}(\vec x)&=\sum_{i=0}^{n-1}(|x_{i+1}|-|x_{i}|)f(V^{i}_+,V^{i}_-)
\\&=\sum_{i=0}^{n-1}|x_{i+1}|\left(f(V^{i}_+,V^{i}_-)-f(V^{i+1}_+,V^{i+1}_-)\right)
\\&=\sum_{i=0}^{n-1}\left\{F\left(|x_{i+1}|\vec 1_{V^{i}_+,V^{i}_-}\right)-F\left(|x_{i+1}|\vec 1_{V^{i+1}_+,V^{i+1}_-}\right)\right\}
\\&\ge\sum_{i=0}^{n-1}\left\{F\left(\vec x_{V^{i}_+,V^{i}_-}\right)-F\left(\vec x_{V^{i+1}_+,V^{i+1}_-}\right)\right\}=F(\vec x).
\end{align*}
On the other hand,
\begin{align*}
f^{L}(\vec x)&=\sum_{i=0}^{n-1}(|x_{i+1}|-|x_{i}|)f(V^{i}_+,V^{i}_-)
=\sum_{i=0}^{n-1}F\left((|x_{i+1}|-|x_{i}|)\vec 1_{V^{i}_+,V^{i}_-}\right)
\\&=\sum_{i=0}^{n-2}\left\{F((|x_{i+1}|-|x_{i}|)\vec 1_{V^{i}_+,V^{i}_-})-F((|x_{i+1}|-|x_{i}|)\vec 1_{V_+^0,V_-^0})\right\}
\\&\;\;\;\;\;+\left\{\sum_{i=0}^{n-2}F((|x_{i+1}|-|x_{i}|)\vec 1_{V_+^0,V_-^0})\right\}+F\left((|x_n|-|x_{n-1}|)\vec 1_{V_+^{n-1},V_-^{n-1}}\right)
\\ \text{by (BS2)}&\le \sum_{i=0}^{n-2}\left\{F\left(\vec x_{V^{i}_+,V^{i}_-}-|x_{i}|\vec 1_{V_+^0,V_-^0}\right)-F\left(\vec x_{V^{i+1}_+,V^{i+1}_-}-|x_{i+1}|\vec 1_{V_+^0,V_-^0}+(|x_{i+1}|-|x_{i}|)\vec 1_{V_+^0,V_-^0}\right)\right\}
\\&\;\;\;\;\;+\left\{\sum_{i=0}^{n-2}(|x_{i+1}|-|x_{i}|)F(\vec 1_{V_+^0,V_-^0})\right\}+F\left((|x_n|-|x_{n-1}|)\vec 1_{V_+^{n-1},V_-^{n-1}}\right)
\\ \text{by (II)}&\le \sum_{i=0}^{n-2}\left(F(\vec x_{V^{i}_+,V^{i}_-}-|x_{i}|\vec 1_{V_+^0,V_-^0})-F(\vec x_{V^{i+1}_+,V^{i+1}_-}-|x_{i+1}|\vec 1_{V_+^0,V_-^0})\right)+F\left((|x_n|-|x_{n-1}|)\vec 1_{V_+^{n-1},V_-^{n-1}}\right)
\\&= F(\vec x_{V_+^0,V_-^0})=F(\vec x)
\end{align*}
according to $(|x_{i+1}|-|x_{i}|)\vec 1_{V_+^0,V_-^0}\bigwedge (\vec x_{V^{i}_+,V^{i}_-}-|x_{i}|\vec 1_{V_+^0,V_-^0})= (|x_{i+1}|-|x_{i}|)\vec 1_{V^{i}_+,V^{i}_-}$ and $$(|x_{i+1}|-|x_{i}|)\vec 1_{V_+^0,V_-^0}\bigvee (\vec x_{V^{i}_+,V^{i}_-}-|x_{i}|\vec 1_{V_+^0,V_-^0})=\vec x_{V^{i+1}_+,V^{i+1}_-}-|x_{i+1}|\vec 1_{V_+^0,V_-^0}+(|x_{i+1}|-|x_{i}|)\vec 1_{V_+^0,V_-^0}$$
for $i=0,\cdots,n-2$, as well as $\vec x_{V^{n-1}_+,V^{n-1}_-}-|x_{n-1}|\vec 1_{V_+^0,V_-^0}= (|x_n|-|x_{n-1}|)\vec 1_{V_+^{n-1},V_-^{n-1}}$.  Therefore, we have  $F(\vec x)=f^L(\vec x)$.

 The ``only if'' part is easy. We only need to prove that for a bisubmodular function $f:\power_2(V)\to\R$, $f^L$ satisfies (BS2), (I) and (II). For convenience,  the proof is provided below. 
\begin{itemize}   \item By the definition of $f^L$,  it is positively homogeneous. Thus, (I)  holds.
\item Again, by the definition of $f^L$, it is easy to check that $f^L(\vec x+t\vec 1_{V_+,V_-})=f^L(\vec x)+tf(V_+,V_-)$ for any $V_\pm\supset V_\pm^0(\vec x)$, and any $t\ge0$. So,  (II) is proved.
\item We use the formulation \eqref{eq:disjoint-pair-Lovasz-def-integral2} of  $f^L$. It is easy to check that $$(V^t_+(\vec x),V^t_-(\vec x))\vee (V^t_+(\vec y),V^t_-(\vec y))=(V^t_+(\vec x\vee \vec y),V^t_-(\vec x\vee \vec y)),$$ 
$$(V^t_+(\vec x),V^t_-(\vec x))\wedge (V^t_+(\vec y),V^t_-(\vec y))=(V^t_+(\vec x\wedge \vec y),V^t_-(\vec x\wedge \vec y)).$$
By  the bisubmodularity  of $f$, 
and the above equalities, we have
\begin{align*}
f^{L}(\vec x)+f^{L}(\vec y)&=\int_0^{\infty} \left(f(V_+^t(\vec x),V_-^t(\vec x))+f(V_+^t(\vec y),V_-^t(\vec y))\right)dt
\\&\ge \int_0^{\infty} \left(f(V^t_+(\vec x\vee \vec y),V^t_-(\vec x\vee \vec y) )+f(V^t_+(\vec x\wedge \vec y),V^t_-(\vec x\wedge \vec y))\right)dt
\\&=f^{L}(\vec x\vee \vec y)+f^{L}(\vec x\wedge\vec y).
\end{align*}
\end{itemize}
 
The proof is completed.
\end{proof}

\begin{pro}\label{pro:setpair-character2}
A continuous function $F$ is a disjoint-pair Lov\'asz extension of some function $f:\power_2(V)\to\R$ if and only if  for any  $\vec x\in\R^n$, there exists  $(V_+,V_-)\in\power_2(V)$   with  
$V_\pm\supset V^0_\pm(\vec x)$, such that $F(\vec x\wedge c\vec 1_{V_+,V_-})+F(\vec x-\vec x\wedge c\vec 1_{V_+,V_-})=F(\vec x)$, $\forall c\ge0$. 
\end{pro}

\begin{proof}
Let $F$ be a continuous function such that  for any  $\vec x\in\R^n$, there exists $(V_+,V_-)\in\power_2(V)$   with 
$V_\pm\supset V^0_\pm(\vec x)$ satisfying $F(\vec x\wedge c\vec 1_{V_+,V_-})+F(\vec x-\vec x\wedge c\vec 1_{V_+,V_-})=F(\vec x)$, $\forall c\ge0$. 
Define the function $f:\power_2(V)\to\R$ by $f(A,B)=F(\vec1_{A,B})$. 
By induction, the property $F(\vec x\wedge c\vec 1_{V_+,V_-})+F(\vec x-\vec x\wedge c\vec 1_{V_+,V_-})=F(\vec x)$ implies a   summation form   of $F$, i.e., 
 \begin{equation}
\label{eq:F-summation}F(\vec x)=\sum_{i=0}^{n-1}F\left((|x_{\sigma(i+1)}|-|x_{\sigma(i)}|)\vec1_{V^{\sigma(i)}_+(\vec x),V^{\sigma(i)}_-(\vec x)}\right).
\end{equation}
Also, for any $(A,B)\in\power_2(V)$, $t>0$  and $c\ge0$, taking $\vec x=(t+c)\vec1_{A,B}$ and $(V_+,V_-)=(A,B)$, 
we obtain   
$F(c\vec1_{A,B})+F(t\vec1_{A,B})=F((t+c)\vec1_{A,B})$. By Cauchy's functional  equation,  this implies that for any $ t\ge0$ and  $(A,B)\in\power_2(V)$,  $F(t\vec1_{A,B})=tF(\vec1_{A,B})=tf(A,B)$, and together with 
\eqref{eq:F-summation} and the   summation  form (see \eqref{eq:disjoint-pair-Lovasz-def}) of the disjoint-pair Lov\'asz extension,  we further derive $$F(\vec x)
=\sum_{i=0}^{n-1}(|x_{\sigma(i+1)}|-|x_{\sigma(i)}|)f(V^{\sigma(i)}_+(\vec x),V^{\sigma(i)}_-(\vec x))=f^L(\vec x).$$  

On the other hand, based on \eqref{eq:disjoint-pair-Lovasz-def}, it is easy to check that for any  $\vec x\in\R^n$, for any  $(V_+,V_-)\in\power_2(V)$   with 
$V_\pm\supset V^0_\pm(\vec x)$,   $f^L(\vec x\wedge c\vec 1_{V_+,V_-})+f^L(\vec x-\vec x\wedge c\vec 1_{V_+,V_-})=f^L(\vec x)$,  $\forall c\ge0$. 
The proof is then completed.
\end{proof}

The $k$-way submodularity can be naturally defined as (S1) and (S2):
\begin{enumerate}
\item[({KS})] Given a tuple $V=(V_1,\cdots,V_k)$ of finite sets and $\A\subset\{(A_1,\cdots,A_k):A_i\subset V_i,\,i=1,\cdots,k\}$,
a discrete function $f:\A\to \R$   is $k$-way submodular if $f(A)+f(B)\ge f(A\vee B)+f(A\wedge B)$, $\forall A,B\in\A$, where $\A$ is a lattice under the corresponding  lattice operations join $\vee$ and meet $\wedge$ defined by $A\vee B=(A_1\cup B_1,\cdots, A_k\cup B_k)$ and $A\wedge B=(A_1\cap B_1,\cdots, A_k\cap B_k)$.
\end{enumerate}

\begin{theorem}\label{thm:submodular-L-equivalent}
Under the assumptions and notations in (KS) above, $\D_\A$ is also closed under $\wedge$ and $\vee$, with $\wedge$ and $\vee$ as in (S2).
 Moreover, the following statements are equivalent:
 \begin{enumerate}[a)]
 \item $f$ is $k$-way submodular on $\A$;
\item the $k$-way Lov\'asz extension $f^L$  is convex on each convex subset of $\D_\A$;
\item the $k$-way Lov\'asz extension $f^L$  is submodular on $\D_\A$.
\end{enumerate}

If one replaces (KS) and (S2) by (BS1) and (BS2) respectively for the bisubmodular setting, then all the above results hold analogously.
\end{theorem}

The proof is a slight  variation of the  original version by  Lov\'asz,  and
is provided for convenience. 

\begin{proof}
Note that $V^t(\vec x)\vee V^t(\vec y)=V^t(\vec x\vee \vec y)$ and $V^t(\vec x)\wedge V^t(\vec y)=V^t(\vec x\wedge \vec y)$, where $V^t(\vec x):=(V^t(\vec x^1),\cdots,V^t(\vec x^k))$, $\forall t\in\R$.
Since $\vec x\in\D_\A$ if and only if $V^t(\vec x)\in \A$, $\forall t\in\R$, and $\A$ is a lattice, $\D_\A$ must be 
a lattice that is closed under the operations $\wedge$ and $\vee$.
According to the $k$-way Lov\'asz extension \eqref{eq:Lovasz-Form-1}, 
we may write
$$f^L(\vec x)=\int_{-N}^Nf(V^t(\vec x))dt-Nf(V)
$$
where $N>\|\vec x\|_\infty$ is a sufficiently large number\footnote{Here we set $f(\varnothing,\cdots,\varnothing)=0$}.
Note that $\vec 1_A\vee \vec 1_B=\vec 1_{A\vee B}$ and $\vec 1_A\wedge \vec 1_B=\vec 1_{A\wedge B}$.
Combining the above results, we immediately get
$$f(A)+f(B)\ge f(A\vee B)+ f(A\wedge B) \; \Leftrightarrow \; f^L(\vec x)+f^L(\vec y)\ge f^L(\vec x\vee \vec y)+f^L(\vec x\wedge \vec y),$$
which proves (a) $\Leftrightarrow$ (c). Note that for  $\vec x\in\D_\A$, $f^L(\vec x)=\sum\limits_{A\in \C(\vec x)}\lambda_Af(A)$ for a unique chain $\C(\vec x)\subset\A$ that is determined by $\vec x$ only, and
the extension $f^{\mathrm{convex}}(\vec x):=\inf\limits_{\{\lambda_A\}_{A\in\A}\in\Lambda(\vec x)}\sum\limits_{A\in \A}\lambda_Af(A)$ is convex on each convex subset of $\D_\A$, where $\Lambda(\vec x):=\{\{\lambda_A\}_{A\in\A}\in\R^\A:\sum\limits_{A\in\A}\lambda_A\vec 1_A=\vec x,\,\lambda_A\ge 0\text{ whenever }A\ne V\}$. We only need to prove $f^L(\vec x)=f^{\mathrm{convex}}(\vec x)$ if and only if $f$ is submodular. In fact, along a standard idea proposed in Lov\'asz's original paper \cite{Lovasz}, one could prove that for a (strictly) submodular function, the set $\{A:\lambda_A^*\ne0\}$ must be a chain, where $\sum\limits_{A\in \A}\lambda_A^*f(A)=f^{\mathrm{convex}}(\vec x)$ achieves the minimum over $\Lambda(\vec x)$, and  one can then easily check that it agrees with $f^L$. The converse can be proved 
in a standard way: $f(A)+f(B)=f^L(\vec 1_A)+f^L(\vec 1_B)\ge 2f^L(\frac{1}{2}(\vec 1_A+\vec 1_B))=f(\vec 1_A+\vec 1_B)=f(\vec 1_{A\vee B}+\vec 1_{A\wedge B})=f(\vec 1_{A\vee B})+f(\vec 1_{A\wedge B})=f(A\vee B)+f(A\wedge B)$. Now, the proof is completed.

For the  bisubmodular case, the above reasoning can be repeated  with minor differences. 
\end{proof}
\label{pro:how-to-be-k-way-Lovasz-submodular}


\begin{remark}
We show some examples about how both convexity and continuous
submodularity can be satisfied. In fact, it is easy to see that the $l^p$-norm $\|\vec x\|_p$ is both convex and continuously
submodular on $\R_+^n$, while the $l^1$-norm $\|\vec x\|_1$ is convex and continuously
submodular on the whole  $\R^n$. Besides, an elementary proof shows that a one-homogeneous continuously
submodular function on $\R_+^2$ must be convex.
\end{remark}

\section{Main  results on  optimization and eigenvalue problems}
\label{sec:main}

We uncover the links between  combinatorial optimization and continuous  programming as well as eigenvalue
problems  in a general setting.
\subsection{Combinatorial and continuous optimization}
\label{sec:CC-transfer}

As we have told in the introduction, the application of the Lov\'asz extension to non-submodular optimization meets with several difficulties, and in this section, we start attacking those. First, we set up some useful results.

\begin{notification}\label{notification:fL}
In this section, $\R_{\ge0}:=[0,\infty)$ is the set of all non-negative numbers. We use $f^L$ to denote the multi-way Lov\'asz extension which can be either the original or the disjoint-pair or the $k$-way Lov\'asz extension.  Moreover, the families $\A$ and $\D_\A$ we consider for optimization problems are  restricted as follows.
\begin{itemize}
\item For the original Lov\'asz extension, we require  $\{\varnothing,V\}\subset\A\subset \power(V)$, and let  $\D_\A=\{\vec x\in[0,+\infty)^n:V^t(\vec x)\in \A,\forall t\in\R\}$.
\item For the  disjoint-pair Lov\'asz extension, we require  $(\varnothing,\varnothing)\in\A\subset \power_2(V)$, and take  $\D_\A=\{\vec x\in\R^n:(V^t_+(\vec x),V^t_-(\vec x))\in \A,\forall t\ge0\}$.
\item For the $k$-way Lov\'asz extension introduced in Definition \ref{def:k-way-Lovasz}, we require   $\{(\varnothing,\cdots,\varnothing),(V_1,\cdots,V_k)\}\subset\A\subset\mathcal{P}(V_1)\times \cdots\times \mathcal{P}(V_k)$, and let  $\D_\A=\{(\vec x^1,\cdots,\vec x^k)\in[0,+\infty)^{n_1+\cdots+n_k}:(V^t(\vec x^1),\cdots,V^t(\vec x^k))\in \A,\forall t\in\R\}$.
\item For the $k$-way disjoint-pair Lov\'asz extension introduced in Definition \ref{def:k-way-pair-Lovasz}, we let   
$(\varnothing,\cdots,\varnothing)\in\A\subset\mathcal{P}_2(V_1)\times \cdots\times \mathcal{P}_2(V_k)$, and  $\D_\A=\{(\vec x^1,\cdots,\vec x^k)\in\R^{n_1+\cdots+n_k}:(V^t_+(\vec x^1),V^t_-(\vec x^1),\cdots,V^t_+(\vec x^k),V^t_-(\vec x^k))\in \A,\forall t\ge0\}$.
\end{itemize} 
\end{notification}

\begin{theorem}\label{thm:tilde-H-f}
 Given set functions $f_1,\cdots,f_n:\A\to \R_{\ge0}$, and a zero-homogeneous function $H:\R^n_{\ge0}\setminus\{\vec 0\}\to\R\cup\{+\infty\}$ 
 with $H(\vec a+\vec b)\ge\min\{H(\vec a),H(\vec b)\}$, $
 \forall \vec a,\vec b\in\R^n_{\ge0}\setminus\{\vec 0\}$, we have
 \begin{equation}\label{eq:H-minimum}
 \min\limits_{A\in \A'}H(f_1(A),\cdots,f_n(A))=\inf\limits_{\vec x\in \D'} H(f^L_1(\vec x),\cdots,f^L_n(\vec x)),\end{equation}
where $\A'=\{A\in\A: (f_1(A),\cdots,f_n(A))\in\mathrm{Dom}(H)\}$, $\D'=\{ \vec x\in\D_\A 
:\,(f^L_1(\vec x),\cdots,f^L_n(\vec x))\in\mathrm{Dom}(H)\}$ and $\mathrm{Dom}(H)=\{\vec a\in \R^n _{\ge0}\setminus\{\vec 0\}: H(\vec a)\in\R\}$.
\end{theorem}

\begin{proof}
By the property of $H$, $\;\forall t_i\ge0\,,n\in\mathbb{N}^+,\, a_{i,j}\ge0,i=1,\cdots,m,\,j=1,\cdots,n$,
\begin{align*}H\left(\sum_{i=1}^m t_i a_{i,1},\cdots,\sum_{i=1}^m t_i a_{i,n}\right)&=H\left(\sum_{i=1}^m t_i \vec a^i\right)\ge \min_{i=1,\cdots,m} H(t_i\vec a^i) \\ &= \min_{i=1,\cdots,m} H(\vec a^i)
= \min_{i=1,\cdots,m} H(a_{i,1},\cdots,a_{i,n}).
\end{align*}
Therefore, in  the case of the original Lov\'asz extension, for any $\vec x\in\D'$,
\begin{align}
&H\left(f^L_1(\vec x),\cdots,f^L_n(\vec x)\right)  \label{eq:psi-H}
\\ =\, & H\left(\int_{\min \vec x}^{\max  \vec x}f_1(V^t(\vec x))dt+ f_1(V(\vec x))\min \vec x,\cdots,\int_{\min \vec x}^{\max \vec x}f_n(V^t(\vec x))dt+ f_n(V(\vec x))\min \vec x\right)\notag
\\ =\, & H\left(\sum_{i=1}^m (t_i-t_{i-1}) f_1(V^{t_{i-1}}(\vec x)),\cdots,\sum_{i=1}^m (t_i-t_{i-1}) f_n(V^{t_{i-1}}(\vec x)) \right)\notag
\\ \ge\, & \min_{i=1,\cdots,m} H\left(f_1(V^{t_{i-1}}(\vec x)),\cdots,f_n(V^{t_{i-1}}(\vec x)) \right)\notag
\\ \ge\, &  \min\limits_{A\in \A'}H(f_1(A),\cdots,f_n(A))\label{eq:min-H}
\\ =\, &  \min\limits_{A\in \A'}H(f_1^L(\vec1_A),\cdots,f_n^L(\vec1_A))\notag
\\ \ge\, &  \inf\limits_{\vec x’\in \D'} H(f^L_1(\vec x’),\cdots,f^L_n(\vec x’)).\label{eq:inf-H}
\end{align}
Combining \eqref{eq:psi-H} with \eqref{eq:min-H}, we have $\inf\limits_{\vec x\in \D'} H(f^L_1(\vec x),\cdots,f^L_n(\vec x))\ge \min\limits_{A\in \A'}H(f_1(A),\cdots,f_n(A))$, and then together with \eqref{eq:min-H} and \eqref{eq:inf-H}, we get the reverse inequality. Hence, \eqref{eq:H-minimum} is proved for the original Lov\'asz extension $f^L$. For the  multi-way settings, the proof is similar and thus we omit them.
%
\end{proof}

\begin{remark}\label{remark:thm-H} Duality: If one replaces $H(\vec a+\vec b)\ge\min\{H(\vec a),H(\vec b)\}$ by $H(\vec a+\vec b)\le\max\{H(\vec a),H(\vec b)\}$, then
   \begin{equation}\label{eq:H-maximum}
   \max\limits_{A\in \A'}H(f_1(A),\cdots,f_n(A))=\sup\limits_{\vec x\in \D'} H(f^L_1(\vec x),\cdots,f^L_n(\vec x)).\end{equation}
   The proof of the  identity \eqref{eq:H-maximum} is similar to that of \eqref{eq:H-minimum}, and thus we omit it.
   \end{remark}
   \begin{remark}
A function $H:[0,+\infty)^n\to \overline{\R}$ has the (MIN) property
if
$$H\left(\sum_{i=1}^m t_i \vec a^i\right)\ge \min_{i=1,\cdots,m} H(\vec a^i),\;\forall t_i>0\,,m\in\mathbb{N}^+,\,\vec a^i\in [0,+\infty)^n.$$
The (MAX) property is formulated analogously.

We can verify that the (MIN) property is equivalent to the
zero-homogeneity and $H(\vec x+\vec y)\ge\min\{H(\vec x),H(\vec y)\}$. A similar correspondence holds for the (MAX) property.
\end{remark}

\begin{remark}Theorem \ref{thm:tilde-H-f} shows that if $H$ has the (MIN)  or (MAX) property, then the corresponding combinatorial optimization is equivalent to a continuous optimization by means of   
 the multi-way Lov\'asz extension. 
 Here are some examples:
 
 Given $c,c_i\ge 0$ with $\sum_i c_i>0$, let $H(f_1,\cdots,f_n)=\frac{c_1f_1+\cdots+c_nf_n-c\sqrt{f_1^2+\cdots+f_n^2}}{f_1+\cdots+f_n}$. Then $H$ satisfies the (MIN)  property, and by Theorem \ref{thm:tilde-H-f}, we have
 $$\min\limits_{A\in \A'}\frac{\sum_i c_if_i(A)-c\sqrt{\sum_if_i^2(A)}}{\sum_i f_i(A)}=\inf\limits_{\psi\in \D'}\frac{\sum_i c_if_i^L(\psi)-c\sqrt{\sum_i (f_i^L(\psi))^2}}{\sum_i f_i^L(\psi)}.$$

Taking $H(f_1,\cdots,f_n)=\frac{(c_1f_1^p+\cdots+c_nf_n^p)^{\frac1p}}{f_1+\cdots+f_n}$ for some $p> 1$, then $H$ satisfies the (MAX)  property, and by Theorem \ref{thm:tilde-H-f}, there holds $$\max\limits_{A\in \A'}\frac{(\sum_i c_if_i(A)^p)^{\frac1p}}{\sum_i f_i(A)}=\sup\limits_{\psi\in \D'}\frac{(\sum_i c_if_i^L(\psi)^p)^{\frac1p}}{\sum_i f_i^L(\psi)}.$$

\begin{proof}[Proof of Theorem \ref{thm:tilde-fg-equal-PQ}]
Without loss of generality, we may assume that $P(f_1,\cdots,f_n)$ is one-homogeneous and subaddtive, while $Q(f_1,\cdots,f_n)$ is one-homogeneous and superadditive on $(f_1,\cdots,f_n)\in\R_{\ge0}^n$. 

Then $H(f_1,\cdots,f_n)=\frac{P(f_1,\cdots,f_n)}{Q(f_1,\cdots,f_n)}$ is zero-homogeneous on $[0,+\infty)^n$,   and 
$$H(\vec f+\vec g)=\frac{P(\vec f+\vec g)}{Q(\vec f+\vec g)}\le \frac{P(\vec f)+P(\vec g)}{Q(\vec f)+Q(\vec g)}\le \max\{\frac{P(\vec f)}{Q(\vec f)},\frac{P(\vec g)}{Q(\vec g)}\}=\max\{H(\vec f),H(\vec g)\}$$
where $\vec f=(f_1,\cdots,f_n)$ and $\vec g=(g_1,\cdots,g_n)$. 



Then the proof is completed by Theorem \ref{thm:tilde-H-f} 
 (and Remark \ref{remark:thm-H}).
\end{proof}

\begin{example}\label{exam:maxcut1}
Given a finite graph $(V,E)$, for $\{i,j\}\in E$, let $f_{\{i,j\}}(A)=1$ if $A\cap\{i,j\}=\{i\}$ or $\{j\}$, and $f_{\{i,j\}}(A)=0$ otherwise. Let $g(A)=|A|$ for $A\subset V$. It is clear that $\frac{\left(\sum_{\{i,j\}\in E}f_{\{i,j\}}^p\right)^{\frac1p}}{g}$ satisfies the condition of  Theorem \ref{thm:tilde-fg-equal-PQ}. Thus, we  derive that
$$\max\limits_{A\ne\varnothing}\frac{|\partial A|^{\frac1p}}{|A|}=\max\limits_{A\ne\varnothing}\frac{(\sum_{\{i,j\}\in E}f_{\{i,j\}}^p(A))^{\frac1p}}{g(A)} =\max\limits_{x\in\R_+^V}\frac{(\sum_{\{i,j\}\in E}|x_i-x_j|^p)^{\frac1p}}{\sum_{i\in V}x_i}=\max\limits_{x\ne 0}\frac{(\sum_{\{i,j\}\in E}|x_i-x_j|^p)^{\frac1p}}{\sum_{i\in V}|x_i|}.$$
Similarly, we have
$$\max\limits_{A\ne\varnothing}|\partial
A|^{\frac1p}=\max\limits_{x\in\R_+^V}\frac{(\sum_{\{i,j\}\in
    E}|x_i-x_j|^p)^{\frac1p}}{\max\limits_{i\in V}x_i}=\max\limits_{x\ne
  0}\frac{(\sum_{\{i,j\}\in E}|x_i-x_j|^p)^{\frac1p}}{2\|\vec
  x\|_\infty},$$
which gives a  continuous representation of the Max-Cut problem. 
The last equality holds due to the following reason: letting $F(\vec x)=(\sum_{\{i,j\}\in
    E}|x_i-x_j|^p)^{\frac1p}$, we can check that  
$\max\limits_{x\in\R_+^V}\frac{F(\vec x)}{\max_{i\in V}x_i}$ achieves its maximum at some characteristic vector $\vec 1_A$, and then $\vec 1_A-\vec 1_{V\setminus A}$ is a maximizer of $\frac{F(\vec x)}{2\|\vec
  x\|_\infty}$ on $\R^V\setminus\{\vec0\}$.

  Similarly,    $\max\limits_{x\ne
  0}\frac{F(\vec x)}{2\|\vec
  x\|_\infty}$ achieves its maximum at $\vec 1_A-\vec 1_{V\setminus A}$ for some $A$, and then $\vec1_A$ indicates a maximizer of $\frac{F(\vec x)}{\max_{i\in V}x_i}$ on the first orthant $\R^V_+$. We need the factor 2  because $F(\vec 1_A-\vec 1_{V\setminus A})=2F(\vec 1_A)$. 
  
It should be noted that the two equivalent continuous reformulations are derived by the original and disjoint-pair Lov\'asz extensions in the following two ways: $$\max\limits_{A\ne\varnothing}|\partial
A|^{\frac1p}=\max\limits_{A\in\power(V)\setminus\{\varnothing\}}\frac{(\sum_{\{i,j\}\in E}f_{\{i,j\}}^p(A))^{\frac1p}}{1} = \max\limits_{x\in\R_+^V}\frac{(\sum_{\{i,j\}\in
    E}|x_i-x_j|^p)^{\frac1p}}{\max_{i\in V}x_i}$$
    where we use  $f^L_{\{i,j\}}(\vec x)=|x_i-x_j|$ and $1^L=\max_{i\in V}x_i$; 
      $$\max\limits_{A\ne\varnothing}|\partial
A|^{\frac1p}=\max\limits_{(A,B)\in\power_2(V)\setminus\{ (\varnothing,\varnothing)\}}\frac{\left(\sum_{\{i,j\}\in E}(f_{\{i,j\}}(A)+f_{\{i,j\}}(B))^p\right)^{\frac1p}}{2} =\max\limits_{\vec x\ne
  \bf{0}}\frac{(\sum_{\{i,j\}\in E}|x_i-x_j|^p)^{\frac1p}}{2\|\vec
  x\|_\infty}$$
  where we use the fact that the disjoint-pair Lov\'asz extension of $(A,B)\mapsto f_{\{i,j\}}(A)+f_{\{i,j\}}(B)$ is $|x_i-x_j|$ and the disjoint-pair Lov\'asz extension of $(A,B)\mapsto 1$ is $\|\vec x\|_\infty$.
\end{example}


\begin{example}\label{exam:maxcut2}
There are many other equalities that can be obtained by  Theorem \ref{thm:tilde-fg-equal-PQ}, such as:
$$\min\limits_{A\ne\varnothing,V}\frac{|\partial A|}{|A|^{\frac1p}}=\min\limits_{x\in\R^V:\,\min x=0}\frac{\sum_{\{i,j\}\in E}|x_i-x_j|}{(\sum_{i\in V}x_i^p)^{\frac1p}}$$
and 
$$\max\limits_{(A,B)\in\power_2(V)\setminus\{(\varnothing,\varnothing)\}}\frac{2|E(A,B)|^{\frac1p}}{\vol(A\cup B)}= \max\limits_{x\ne 0}\frac{(\sum_{\{i,j\}\in E}(|x_i|+|x_j|-|x_i+x_j|)^p)^{\frac1p}}{\sum_{i\in V}\deg_i|x_i|}$$
whenever $p\ge 1$. Here, $\vol A =\sum_{i\in A} \deg_i$.\\ 
The last equality shows a variant of the dual Cheeger constant. A slight modification gives 
$$\max\limits_{A\in\power(V)}2|\partial A|^{\frac1p}=\max\limits_{(A,B)\in\power_2(V)}2|E(A,B)|^{\frac1p}=\max\limits_{x\ne 0}\frac{\left(\sum_{\{i,j\}\in E}(|x_i|+|x_j|-|x_i+x_j|)^p\right)^{\frac1p}}{\|\vec x\|_\infty}$$
showing a new continuous formulation of the Maxcut problem.
\end{example}
\end{remark}
Taking $n=2$ and $H(f_1,f_2)=\frac{f_1}{f_2}$ in Theorem \ref{thm:tilde-H-f}, then such an $H$ satisfies both (MIN) and (MAX) properties. So, we get
$$\min\limits_{A\in \A'}\frac{f_1(A)}{f_2(A)}=\inf\limits_{\psi\in \D'}\frac{f_1^L(\psi)}{f_2^L(\psi)},\;\;\;\text{ and } \;\; \max\limits_{A\in \A'}\frac{f_1(A)}{f_2(A)}=\sup\limits_{\psi\in \D'}\frac{f_1^L(\psi)}{f_2^L(\psi)}.$$
In fact, we can get more:

\begin{pro}\label{pro:fraction-f/g}
 Given two  functions $f,g:\A\to [0,+\infty)$, let $\tilde{f},\tilde{g}:\D_\A\to \R$ satisfy $\tilde{f}\ge f^L$, $\tilde{g}\le g^L$, $\tilde{f}(\vec1_{A})=f(A)$ and $\tilde{g}(\vec1_{A})=g(A)$ for any $A\in \mathcal{A}$. Then
$$\min\limits_{A\in \A\cap\supp(g)}\frac{f(A)}{g(A)}=\inf\limits_{\psi\in \D_\A\cap\supp(\tilde{g})}\frac{\widetilde{f}(\psi)}{\widetilde{g}(\psi)}.$$

If we replace the condition $\tilde{f}\ge f^L$ and $\tilde{g}\le g^L$  by $\tilde{f}\le f^L$ and $\tilde{g}\ge g^L$, then $$\max\limits_{A\in \A\cap\supp(g)}\frac{f(A)}{g(A)}=\sup\limits_{\psi\in \D_\A\cap\supp(\tilde{g})}\frac{\widetilde{f}(\psi)}{\widetilde{g}(\psi)}.$$

For any $\alpha\ne0$,  then $\tilde{f}=((f^\alpha)^L)^{\frac1\alpha}$ and
$\tilde{g}=((g^\alpha)^L)^{\frac1\alpha}$ satisfy the above two
identities. 
\end{pro}

\begin{proof}
It is obvious that
 $$\inf\limits_{\psi\in \D_\A\cap\supp(\tilde{g})}\frac{\widetilde{f}(\psi)}{\widetilde{g}(\psi)}\le\min\limits_{A\in \A\cap\supp(g)} \frac{\widetilde{f}(\vec1_{A})}{\widetilde{g}(\vec1_{A})} = \min\limits_{A\in \A\cap\supp(g)}\frac{f(A)}{g(A)}.$$
 On the other hand, for any $\psi\in\D_\A\cap\supp(\tilde{g})$, $g^L(\psi)\ge \tilde{g}(\psi)>0$. Hence, there exists $t\in (\min \widetilde{\beta}\psi-1,\max \widetilde{\beta}\psi+1)$ satisfying $g(V^t(\psi))>0$. Here $\widetilde{\beta}\psi=\psi$ (resp., $|\psi|$), if $f^L$ represents either the original or the $k$-way Lov\'asz extension of $f$ (resp., either the disjoint-pair or the $k$-way disjoint-pair Lov\'asz  extension). So, the set $W(\psi):=\{t\in\R: g(V^t(\psi))>0\}$ is nonempty. Since $\{V^t(\psi):t\in W(\psi)\}$ is finite, there exists $t_0\in W(\psi)$ such that $\frac{f(V^{t_0} (\psi))}{g(V^{t_0} (\psi))}=\min\limits_{t\in W(\psi)}\frac{f(V^{t} (\psi))}{g(V^{t} (\psi))}$. Accordingly, $f(V^{t} (\psi))\ge \frac{f(V^{t_0} (\psi))}{g(V^{t_0} (\psi))}g(V^{t} (\psi))$ for any $t\in W(\psi)$, and thus
  $$f(V^{t} (\psi))\ge Cg(V^{t} (\psi)),\;\;\;\text{ with }\;\;C=\min\limits_{t\in W(\psi)}\frac{f(V^{t} (\psi))}{g(V^{t} (\psi))}\ge0,$$
  holds for any $t\in\R$ (because $g(V^{t} (\psi))=0$ for $t\in\R\setminus W(\psi)$ which means that the above inequality automatically holds).
  Consequently,
 \begin{align*}&\tilde{f}(\psi)\ge f^L(\psi)
\\ =\,& \int_{\min \widetilde{\beta}\psi}^{\max  \widetilde{\beta}\psi}f(V^t(\psi))dt+ f(V(\psi))\min \widetilde{\beta}\psi
\\ \ge\,& C\left( \int_{\min \widetilde{\beta}\psi}^{\max  \widetilde{\beta}\psi}g(V^t(\psi))dt+ g(V(\psi))\min \widetilde{\beta}\psi\right).
\\ =\,&C g^L(\psi)\ge C\tilde{g}(\psi).
 \end{align*}
 where we used  $\min \widetilde{\beta}\psi\ge0$. The proof of $\min \widetilde{\beta}\psi\ge0$ is straightforward: in fact, by   $\psi\in\D_\A$,  if we use the original or the $k$-way Lov\'asz extension introduced in Definition \ref{def:k-way-Lovasz}, then  $\widetilde{\beta}\psi=\psi$, and   $\D_\A$ lies in the closure of the first orthant of the Euclidean space, meaning that $\widetilde{\beta}\psi=\psi\ge0$; and if we use  the disjoint-pair or the $k$-way disjoint-pair Lov\'asz  extension in  Definition \ref{def:k-way-pair-Lovasz}, then $\widetilde{\beta}\psi=|\psi|\ge0$. 

 It follows that
 $$\frac{\widetilde{f}(\psi)}{\widetilde{g}(\psi)}\ge C\ge \min\limits_{A\in \A\cap\supp(g)}\frac{f(A)}{g(A)}$$
 and thus the proof is completed. The dual case is similar.

For $\alpha>0$, we can simply suppose $\supp(g)=\A$. Then 
 \begin{align*}
 \min\limits_{A\in \A}\frac{f(A)}{g(A)}=\min\limits_{A\in \A}\frac{ (f^\alpha)^{\frac1\alpha}(A)}{(g^\alpha)^{\frac1\alpha}(A)}=\left(\min\limits_{A\in \A}\frac{f^\alpha(A)}{ g^\alpha(A)}\right)^{\frac1\alpha}
 = \left(\inf\limits_{\psi\in \D_\A}\frac{( f^\alpha)^L(\psi)}{(g^\alpha)^L(\psi)}\right)^{\frac1\alpha}=\inf\limits_{\psi\in \D_\A}\frac{(( f^\alpha)^L)^{\frac1\alpha}(\psi)}{(( g^\alpha)^L)^{\frac1\alpha}(\psi)}.
 \end{align*}
 
For $\alpha<0$, we may suppose without loss of generality that $g(A)>0$ and $f(A)>0$ for any $A\in\A$. Then, in this case,  
 \begin{align*}
 \min\limits_{A\in \A}\frac{f(A)}{g(A)}=\min\limits_{A\in \A}\frac{ (f^\alpha)^{\frac1\alpha}(A)}{(g^\alpha)^{\frac1\alpha}(A)}=\left(\max\limits_{A\in \A}\frac{f^\alpha(A)}{ g^\alpha(A)}\right)^{\frac1\alpha}
= \left(\sup\limits_{\psi\in \D_\A}\frac{( f^\alpha)^L(\psi)}{(g^\alpha)^L(\psi)}\right)^{\frac1\alpha}=\inf\limits_{\psi\in \D_\A}\frac{(( f^\alpha)^L)^{\frac1\alpha}(\psi)}{(( g^\alpha)^L)^{\frac1\alpha}(\psi)}.
 \end{align*}
 
 This completes the proof.
\end{proof}

It is worth noting that in Proposition  \ref{pro:fraction-f/g}, $\A$ can be a family of some set-tuples, and $f^L$ is the  multi-way Lov\'asz extension of the corresponding  $f$. We point out that we can replace the Lov\'asz extension $f^L$ by any other extension $f^E$ with the property that $f^E/g^E$ achieves its minimum and maximum at some $0$-$1$ vector $\vec 1_A$ for some $A\in\A$. 
 Similarly, we have:

\begin{pro}\label{pro:maxconvex}
Let $f,g:\A\to [0,+\infty)$ be two set functions and $f:=f_1-f_2$ and $g:=g_1-g_2$ be decompositions of differences of submodular functions.

Let $\widetilde{f}_2,\widetilde{g}_1$ be the restriction of  positively one-homogeneous convex functions onto $\D_\A$, with $f_1(A)= \widetilde{f}_1(\vec1_{A})$ and $g_2(A)= \widetilde{g}_2(\vec1_{A})$ for any $A\in \mathcal{A}$. 
Define $\widetilde{f}=f_1^L-\widetilde{f}_2$ and $\widetilde{g}=\widetilde{g}_1-g_2^L$. Then,
$$\min\limits_{A\in \A\cap\supp(g)}\frac{f(A)}{g(A)}=\min\limits_{\vec x\in \D_\A\cap\supp(\widetilde{g})}\frac{\widetilde{f}(\vec x)}{\widetilde{g}(\vec x)}.$$
\end{pro}

\begin{remark}
Hirai et al introduce the generalized Lov\'asz extension of $f:\mathcal{L}\to\overline{\R}$ on a graded poset $\mathcal{L}$ (see \cite{Hirai18,HH17-arxiv}). Since $f^L(\vec x)=\sum_i\lambda_i f(\vec p_i)$ for $\vec x=\sum_i\lambda_i \vec p_i$ lying in the orthoscheme complex $K(\mathcal{L})$, the same results as stated in Theorem \ref{thm:tilde-H-f} and Proposition \ref{pro:fraction-f/g} hold for such a generalized Lov\'asz extension $f^L$.   Propositions \ref{pro:fraction-f/g} and \ref{pro:maxconvex} are also generalizations of Theorem 3.1 in \cite{HS11} and Theorem 1 (b) in \cite{BRSH13}.
\end{remark}

Although the continuous representations translate the original problems into  equivalently difficult
optimization problems, 
we should point out that the continuous reformulations ensure that many fast algorithms in continuous programming  can be applied directly to certain  combinatorial optimization problems.  For example, the fractional form of the equivalent continuous optimizations shown in  Theorem \ref{thm:tilde-fg-equal} as well as Propositions \ref{pro:fraction-f/g} and \ref{pro:maxconvex}  implies that we can directly adopt the Dinkelbach iteration in Fractional Programming \cite{SI83} to solve them. In addition, since the equivalent continuous formulation is Lipschitz, we can also adopt the 
stochastic subgradient method  \cite{Davis19-FoCM} to solve certain discrete optimization problems directly.

Tables~\ref{tab:L-one} and \ref{tab:L-two} and Propositions \ref{pro:discrete:one-to-two}, \ref{pro:discrete:one-to-k} and \ref{pro:discrete:two-to-k} present a general correspondence between set or set-pair functions and their Lov\'asz extensions. We shall make use of several of those in Section \ref{sec:examples-Applications}. Note that the first four lines in Table \ref{tab:L-one} for the original Lov\'asz extension, and the first five lines in Table \ref{tab:L-two} for the disjoint-pair Lov\'asz extension are known (see \cite{HS11,CSZ18}).

\begin{table}
\centering
\caption{\small Original Lov\'asz extension of some objective functions.}
\begin{tabular}{|l|l|}
               \hline
               Set function $f(A)=$ & Lov\'asz extension $f^L(\vec x)=$  \\
               \hline
                 $\#E(A,V\setminus A)$ & $\sum\limits_{\{i,j\}\in E}|x_i-x_j|$ (see \cite{HS11})\\
                 \hline
               $C$&$C\max_i x_i$ (see \cite{HS11})\\
               \hline
               $\vol(A)$ & $\sum_i \deg_i x_i$ (see \cite{HS11})\\
               \hline
               $\min\{\vol(A),\vol(V\setminus A)\}$& $\min\limits_{t\in \mathbb{R}}\|\vec x-t \vec 1\|_1$ (see \cite{HS11})\\
               \hline
               $\#A\cdot\#(V\setminus A)$ & $\sum\limits_{
               \{i,j\}\subset V}|x_i-x_j|$ \\
               \hline
              $(\#A)^k$ & 
$\sum_{i_1,\cdots,i_k\in V}\min\{x_{i_1},\cdots,x_{i_k}\}$              
               \\
                  \hline
              $\vol(A)^k$ & 
$\sum_{i_1,\cdots,i_k\in V}\deg_{i_1}\cdots\deg_{i_k}\min\{x_{i_1},\cdots,x_{i_k}\}$              
               \\
               \hline $\#V(E(A,V\setminus A))$ & $\sum\limits_{i=1}^n(\max\limits_{j\in N(i)}x_j-\min\limits_{j\in N(i)}x_j)$
               \\
               \hline
            \end{tabular}
             \label{tab:L-one}
\end{table}

\begin{table}
\centering
\caption{\small Disjoint-pair Lov\'asz extension of several objective functions.}
\begin{tabular}{|l|l|}
               \hline
               Objective function $f(A,B)=$ & Disjoint-pair Lov\'asz extension $f^L(\vec x)=$  \\
               \hline
                 $\#E(A,V\setminus A)+\#E(B,V\setminus B)$& $\sum\limits_{\{i,j\}\in E}|x_i-x_j|$ (see \cite{CSZ18})\\
                 \hline
               $\#E(A,B)$&$\frac12\left(\sum\limits_{i\in V}\deg_i|x_i|- \sum\limits_{\{i,j\}\in E}|x_i+x_j|\right)$ (see \cite{CSZ18})\\
               \hline
               $C$&$C\|\vec x\|_\infty$ (see \cite{CSZ18})\\
               \hline
               $\vol(A)+\vol(B)$&$\sum\limits_{i\in V}\deg_i|x_i|$ (see \cite{CSZ18})\\
               \hline
               $\min\{\vol(A),\vol(V\setminus A)\}+\min\{\vol(B),\vol(V\setminus B)\}$&$\min\limits_{\alpha\in \mathbb{R}}\|(x_1,\cdots,x_n)-\alpha \vec 1\|$ (see \cite{CSZ18})\\
               \hline
              $ \# E(A\cup B,A\cup B)$ & $\sum_{i\sim j}\min\{|x_i|,|x_j|\}$ \\
              \hline
               $\# (A\cup B)\cdot\# E(A\cup B,A\cup B)$ & $\sum_{k\in V,i\sim j}\min\{|x_k|,|x_i|,|x_j|\}$\\ \hline
$\#(A\cup B)\cdot\#(V\setminus (A\cup B))$ & $\sum_{i>j}||x_i|-|x_j||$ \\ \hline
            \end{tabular}
             \label{tab:L-two}
\end{table}

\begin{pro}\label{pro:discrete:one-to-two}
Suppose $f,g:\power(V)\to [0,+\infty)$ are two set functions with $g(A)>0$ for any $A\in \power(V)\setminus\{\varnothing\}$. 
Then $$\min\limits_{A\in \power(V)\setminus\{\varnothing\}}\frac{f(A)}{g(A)}=\min\limits_{(A,B)\in \power(V)^2\setminus\{(\varnothing,\varnothing)\}}\frac{f(A)+f(B)}{g(A)+g(B)}=\min\limits_{(A,B)\in \power_2(V)\setminus\{(\varnothing,\varnothing)\}}\frac{f(A)+f(B)}{g(A)+g(B)},$$
where the right identity needs additional assumptions like
$f(\varnothing)=g(\varnothing)=0$\footnote{This setting is natural, as the
  Lov\'asz extension doesn't use the datum on $\varnothing$.} or  that $f$
  and $g$ are symmetric.\footnote{A function $f:\power(V)\to\R$ is {\sl symmetric} if $f(A)=f(V\setminus A)$, $\forall A\subset V$.} 
Replacing $f(B)$ and $g(B)$ by $f(V\setminus B)$ and $g(V\setminus B)$, all the above identities hold without any additional assumption. Clearly, replacing `min' by `max', all statements still hold.
\end{pro}

\begin{pro}\label{pro:discrete:one-to-k}
Suppose $f,g:\power(V)\to [0,+\infty)$ are two set functions with $g(A)>0$ for any $A\in \power(V)\setminus\{\varnothing\}$. 
Then $$\min\limits_{A\in \power(V)}\frac{f(A)}{g(A)}=\min\limits_{(A_1,\cdots,A_k)\in \power(V)^k}\frac{\sum_{i=1}^kf(A_i)}{\sum_{i=1}^kg(A_i)}=\min\limits_{(A_1,\cdots,A_k)\in \power(V)^k}\sqrt[k]{\frac{\prod_{i=1}^kf(A_i)}{\prod_{i=1}^kg(A_i)}}=\min\limits_{(A_1,\cdots,A_k)\in \power_k(V)}\frac{\sum_{i=1}^kf(A_i)}{\sum_{i=1}^kg(A_i)},$$
where the last identity needs additional assumptions like $f(\varnothing)=g(\varnothing)=0$. 
\end{pro}

\begin{pro}\label{pro:discrete:two-to-k}
Suppose $f,g:\power_2(V)\to [0,+\infty)$ are two set functions with $g(A,B)>0$ for any $(A,B)\in \power_2(V)\setminus\{(\varnothing,\varnothing)\}$. 
Then $$\min\limits_{A\in \power_2(V)}\frac{f(A,B)}{g(A,B)}=\min\limits_{(A_1,B_1,\cdots,A_k,B_k)\in \power_2(V)^k}\frac{\sum_{i=1}^kf(A_i,B_i)}{\sum_{i=1}^kg(A_i,B_i)}=\min\limits_{(A_1,B_1,\cdots,A_k,B_k)\in \power_{2k}(V)}\frac{\sum_{i=1}^kf(A_i,B_i)}{\sum_{i=1}^kg(A_i,B_i)},$$
where the last identity needs additional assumptions like $f(\varnothing,\varnothing)=g(\varnothing,\varnothing)=0$\footnote{This setting is natural, as the  disjoint-pair Lov\'asz extension doesn't use the information on $(\varnothing,\varnothing)$.}. 
\end{pro}

Together with Propositions \ref{pro:setpair-original} and \ref{pro:discrete:one-to-two}, one may directly transfer the data from  Table \ref{tab:L-one} to Table \ref{tab:L-two}. Similarly, by employing Propositions \ref{pro:separable-summation}, \ref{pro:discrete:one-to-k} and \ref{pro:discrete:two-to-k}, the $k$-way Lov\'asz extension of some special functions can be transformed to the  original and the disjoint-pair versions.

\begin{pro}\label{pro:Lovasz-f-pre}
For any $a<b$, and for any $f$, 
$$ \min\limits_{x\in [a,b]^n}f^L(\vec x)= \min\limits_{x\in \{a,b\}^n}f^L(\vec x)=af(V)+(b-a)\min\limits_{A\subset V}f(A).$$
Clearly, we can replace all `min' by `max'.
\end{pro}

\begin{proof}
Since $f^L$ is linear on each piece $\triangle_\sigma\cap [a,b]^n:=\{\vec x\in
\R^n:a\le x_i\le b,x_{\sigma(1)}\le \cdots\le x_{\sigma(n)}\}$, where
$\sigma:\{1,\cdots,n\}\to \{1,\cdots,n\}$ is a permutation, the maximum and
minimum of $f^L$ on $\triangle_\sigma\cap [a,b]^n$ can be reached at some
vertices of the simplex  $\triangle_\sigma\cap [a,b]^n$. Note that the
vertices of $\triangle_\sigma\cap [a,b]^n$ are included in $\{a,b\}^n$ (i.e.,
the vertices of the hypercube $ [a,b]^n$). Thus, the maximum and minimum of
$f^L$ on $ [a,b]^n$ can be attained at some points in  $\{a,b\}^n$. By the definition of Lov\'asz extension, it is easy to check that $f^L(b\vec1_B+a\vec1_{V\setminus B})=af(V)+(b-a)f(B)$.  The proof is completed.
\end{proof}

  \noindent\textbf{Discrete vs.  continuous optimization  and our approach}: 
The general framework  
based on multi-way Lov\'asz extensions  is universal  
and fundamental  with potential to design some  simple iterative algorithms using  equivalent continuous optimization to approach and solve discrete optimization problems. 
 To illustrate this point, we present   in the last paragraph of this section some 
 advantages of our equivalent continuous formulations for certain  combinatorial optimization problems,  
 and we present  in Section  \ref{sec:algo} an algorithm for solving such combinatorial optimization problems.  

%
It is a fundamental aspect of our scheme that the critical data (including min-max data, saddle points, and optimal values)  
of the continuous representations that we develop   cover all the  
 key information 
of the original  combinatorial problems (see Section  \ref{sec:eigenvalue})%


From the viewpoint of applied and computational  mathematics, 
a significant  advantage of the formulations obtained by  multi-way Lov\'asz extensions in this paper is that the approach can be  generally applied to  many combinatorial optimization problems, and compared to other formulations, our  formulation is in  quotient  form, with both numerator and  denominator expressed as the difference of convex functions, which allows us to directly use   techniques from DC (Difference of Convex functions) programming and fractional programming.  Perhaps even  more importantly, our approach doesn't need any additional rounding techniques.  We refer to Section  \ref{sec:algo} for a detailed explanation for such  advantages. 
As can be seen in the next two subsections,  these formulations provide new insight into  structure and properties of certain combinatorial  problems, and allow one to develop more efficient algorithms for computing optimal and approximate solutions. 

\subsection{Eigenvalue problems for Lov\'asz extension}
\label{sec:eigenvalue}

For convenience, we shall work in a normed space $X$, and we will take $X$  as the usual Euclidean space $\R^n$ in this subsection, and a general normed space $X$ will be used in Section \ref{sec:algo}. 

For a convex function $F:X\to \mathbb{R}$,  its sub-gradient (or sub-derivative) $\nabla F(\vec x)$ is defined as the collection of $\vec u\in X^*$ satisfying  $F(\vec y)-F(\vec x)\ge \langle \vec u,\vec y-\vec x\rangle,\;\forall \vec y\in X$, where $X^*$ is the dual of $X$ and $\langle \vec u,\vec y-\vec x\rangle$ is the action of $\vec u$ on $\vec y-\vec x$.
The concept of a sub-gradient has been extended to Lipschitz functions. This is called the  Clarke derivative \cite{Clarke}:
 $$\nabla F(\vec x)=\left\{\vec u\in X^*\left|\limsup_{\vec y\to \vec x, t\to 0^+}\frac{F(\vec y+t\vec h)-F(\vec y)}{t}\ge \langle \vec u,\vec h\rangle,\forall \vec h\in X\right.\right\}.$$
 And it can even be  generalized to the class of lower semi-continuous functions \cite{DM94,D10}.

In this section, we give the proof of Theorem \ref{introthm:eigenvalue} by
establishing some  properties on the nonlinear eigenvalue problem of the
function pair $(f^L,g^L)$.
\begin{defn}
 $\lambda \in \R$ is called an eigenvalue, and $\vec x \in
\R^n\setminus\{\vec0\}$ a corresponding eigenvector of the
function pair $(f^L,g^L)$ if
$$\vec0\in \nabla f^L (\vec x)- \lambda\nabla  g^L (\vec x).$$
We then also call $(\lambda, \vec x)$ an eigenpair. 
    \end{defn}

\begin{pro}\label{pro:Lovasz-eigen-pre}
In the setting of the original Lov\'asz extension, for any eigenvalue $\lambda$ of $(f^L,g^L)$, there exists $A\in\power(V)\setminus\{\varnothing\}$  such that $\lambda=f(A)/g(A)$ and $\vec 1_A$ is a corresponding eigenvector.  Indeed, every eigenvalue has an  eigenvector in   $\{a,b\}^n$, for given distinct real numbers $a$ and $b$. 
Moreover, we have:
\begin{itemize}
    \item if $g(V)\ne0$, then $\lambda=f(V)/g(V)$ is the only possible eigenvalue of $(f^L,g^L)$; 
    \item if $g(V)=0$ and $(f^L,g^L)$ has at least one eigenvalue, then $f(V)=0$ and in this case, $(f^L,g^L)$ may have many distinct eigenvalues.
\end{itemize}
\end{pro}
\begin{proof}
 We need the following  basic statement.

\textbf{Argument}. For a piecewise linear function $F:\R^n\to\R$ with finite pieces, if $F$ is linear on a convex subset  $\Omega$, then $\nabla F(\vec x)\subset \nabla F(\vec y)$ for any relative interior point $\vec x$ of $\Omega$ and any relative boundary point $\vec y$ in $ \overline{\Omega}$. 

\vspace{0.16cm}

Suppose that $(\lambda,\vec x)$ is an eigenpair of $(f^L,g^L)$. Since $f^L$ is one-homogeneous and  linear along the direction $\vec 1$, we can assume without loss of generality that  $\vec x$ lies in the interior of the simplex $\triangle$ with vertices $\vec 1_{A_1},\cdots,\vec1_{A_k}$, where $A_1,\cdots,A_k$ are upper level sets of $\vec x$.  Applying  the above argument to the piecewise linear functions $f^L$ and $g^L$, we immediately get $\vec0\in\nabla f^L(\vec x)-\lambda \nabla g^L(\vec x)\subset \nabla f^L(\vec 1_{A_i})-\lambda \nabla g^L(\vec 1_{A_i})$ for any $i=1,\cdots,k$, meaning that each $\vec 1_{A_i}$ is an eigenvector of  $(f^L,g^L)$. 

Moreover, by the definition of Lov\'asz extension, we can check that $\forall \vec v\in\nabla f^L(\vec x)$, $\sum_{i\in V} v_i=f(V)$. Therefore, $\nabla f^L(\vec x)\cdot\vec1=f(V)$, $\forall \vec x$. Hence, for any eigenpair $(\lambda,\vec x)$, $f(V)-\lambda g(V)=(\nabla f^L(\vec x)-\lambda\nabla g^L(\vec x) )\cdot\vec1=0$, implying that $\lambda=f(V)/g(V)$ if $g(V)\ne 0$; and $f(V)=0$ if $g(V)= 0$.

For the case of $f(V)=g(V)= 0$, we may take a look at the example
$f(A)=|\partial A|$ and $g(A)=\min\{\vol(A),\vol(V\setminus A)\}$ defined on a
simple graph $G=(V,E)$. Then the eigenvalue problem of $(f^L,g^L)$ reduces to
the problem for the  graph 1-Laplacian. And it is known that the 1-Laplacian may have many different  eigenvalues \cite{CSZ17}.
\end{proof}

\begin{pro}\label{pro:Lovasz-eigen}
In the setting of the  disjoint-pair  Lov\'asz extension, for any eigenvalue $\lambda$ of $(f^L,g^L)$, there exists $(A,B)\in\power_2(V)\setminus\{(\varnothing,\varnothing)\}$  such that $\lambda=f(A,B)/g(A,B)$ and $\vec 1_A-\vec1_B$ is a corresponding eigenvector.  If we further assume that $2f(A,B)=f(A,V\setminus A)+f(V\setminus B,B)$ and $2g(A,B)=g(A,V\setminus A)+g(V\setminus B,B)$ for any $(A,B)\in \power_2(V)\setminus\{(\varnothing,\varnothing)\}$, then for any eigenvalue $\lambda$ of $(f^L,g^L)$, there exists $A\subset V$  such that $\lambda=f(A,V\setminus A)/g(A,V\setminus A)$ and $\vec 1_A-\vec1_{V\setminus A}$ is a corresponding eigenvector.
\end{pro}

\begin{proof}
 The general result on  the disjoint-pair  Lov\'asz extension is similar to  
 Proposition \ref{pro:Lovasz-eigen-pre}, and thus we omit the proof. 
 
Let's focus on the special case that $2f(A,B)=f(A,V\setminus A)+f(V\setminus B,B)$ and $2g(A,B)=g(A,V\setminus A)+g(V\setminus B,B)$  for any $(A,B)\in \power_2(V)\setminus\{(\varnothing,\varnothing)\}$. We shall prove that in this case, 
typical eigenvectors can be taken from  $\{-1,1\}^n$. 

\vspace{0.16cm}

\textbf{Claim}. For any $A_1\subset \cdots\subset A_k$ with $(A_1,A_k)\ne (\varnothing,V)$, $f^L$ is linear on  $\mathrm{conv}\{\vec1_{A_i}-\vec1_{V\setminus A_i}:i=1,\cdots,k\}$. 

\textbf{Proof}. The absolute  comonotonicity of the disjoint-pair Lov\'asz extension implies that $f^L$ is linear on $\mathrm{conv}\{\vec1_{A_i}-\vec1_{B_i}:i=1,\cdots,k\}$ whenever $A_1\subset \cdots\subset A_k$ and $B_1\subset \cdots\subset B_k$ and $A_k\cap B_k=\varnothing$. Thus, for any $\sigma,\tau:\{1,\cdots,k\}\to \{1,\cdots,k\}$  with $\sigma(1)\le \cdots\le \sigma(k)\le\tau(k)\le\cdots\le\tau(1)$, $f^L$ is linear on $\mathrm{conv}\{\vec1_{A_{\sigma(i)}}-\vec1_{V\setminus A_{\tau(i)}}:i=1,\cdots,k\}$. Note that
\begin{align*}
\mathrm{conv}\{\vec1_{A_i}-\vec1_{V\setminus A_i}:i=1,\cdots,k\}&=\bigcup_{\sigma,\tau} \mathrm{conv}\{\vec1_{A_{\sigma(i)}}-\vec1_{V\setminus A_{\tau(i)}}:i=1,\cdots,k\} 
\\&=\bigcup_{l=1}^k\bigcup_{\sigma(1)=1,\sigma(k)=\tau(k)=l,\tau(1)=k}\mathrm{conv}\{\vec1_{A_{\sigma(i)}}-\vec1_{V\setminus A_{\tau(i)}}:i=1,\cdots,k\}
\end{align*}
where in the second line we can  further assume that each $\mathrm{conv}\{\vec1_{A_{\sigma(i)}}-\vec1_{V\setminus A_{\tau(i)}}:i=1,\cdots,k\}$ is a $(k-1)$-dim simplex, and there are exactly $\sum_{l=1}^k{k-1\choose l-1}=2^{k-1}$ simplexes of dimension $(k-1)$.  
Since
\begin{align*}
2f^L(\vec1_{A_{\sigma(i)}}-\vec1_{V\setminus A_{\tau(i)}})
&=2f(A_{\sigma(i)},V\setminus A_{\tau(i)})
\\&=f(A_{\sigma(i)},V\setminus A_{\sigma(i)})+f(A_{\tau(i)},V\setminus A_{\tau(i)})
\\&=
f^L(\vec1_{A_{\sigma(i)}}-\vec1_{V\setminus A_{\sigma(i)}})+f^L(\vec1_{A_{\tau(i)}}-\vec1_{V\setminus A_{\tau(i)}})
\end{align*}
and $\vec1_{A_{\sigma(i)}}-\vec1_{V\setminus A_{\sigma(i)}}+\vec1_{A_{\tau(i)}}-\vec1_{V\setminus A_{\tau(i)}}=2(\vec1_{A_{\sigma(i)}}-\vec1_{V\setminus A_{\tau(i)}})$, $f^L$ must be linear on the segment $ \mathrm{conv}\{\vec1_{A_{\sigma(i)}}-\vec1_{V\setminus A_{\sigma(i)}},\vec1_{A_{\tau(i)}}-\vec1_{V\setminus A_{\tau(i)}}\}$. Therefore, one can check that $f^L$ is linear on the simplex  $\mathrm{conv}\{\vec1_{A_i}-\vec1_{V\setminus A_i}:i=1,\cdots,k\}$. 

\vspace{0.16cm}

Given an eigenvalue $\lambda$, let $\vec 1_A-\vec 1_B$ be a corresponding eigenvector for some $(A,B)\in\power_2(V)\setminus\{(\varnothing,\varnothing)\}$. By the above claim and argument, it can be verified that  both $\vec1_A-\vec 1_{V\setminus A}$ and  $\vec 1_{V\setminus B}-\vec1_B$ are eigenvectors w.r.t. $\lambda$.  
\end{proof}

\begin{remark}
In Proposition \ref{pro:Lovasz-eigen}, the  condition $2f(A,B)=f(A,V\setminus A)+f(V\setminus B,B)$ for any $(A,B)\in \power_2(V)$ is natural and easy to satisfy. Below, we provide two examples satisfying the condition. 

Example 1. $f(A,B)=\hat{f}(A)+\hat{f}(B)$ for some $\hat{f}:\power(V)\to\R$ with $\hat{f}(A)=\hat{f}(V\setminus A)$ for any $A\subset V$. 
 In this case,  $f^L(\vec x)=\hat{f}^L(\vec x)$. 

Example 2. $f(A,B)=1$ whenever $A\cup B\ne\varnothing$. In this case,  $f^L(\vec x)=\|\vec x\|_\infty$.  

One may observe  that in the above examples, $f$ is symmetric, i.e.,
$f(A,B)=f(B,A)$, but it is not a sufficient condition for Proposition
\ref{pro:Lovasz-eigen}. In fact, for symmetric functions  $f$ and $g$ on
$\power_2(V)$,  not every  eigenvalue has an eigenvector possessing the form
of $\vec 1_A-\vec1_{V\setminus A}$. In fact, taking $f(A,B)=\#E(A,V\setminus
A)+\#E(B,V\setminus B)$ and $g(A,B)=\vol(A\cup B)$, we have $f^L(\vec
x)=\sum_{\{i,j\}\in E}|x_i-x_j|$  and $g^L(\vec x)=\sum_{i\in
  V}\deg(i)|x_i|$. Letting $V=\{1,2,3\}$ and $E=\{\{1,2\},\{2,3\},\{1,3\}\}$,
it is known in \cite{Chang16} that $\vec 1_A-\vec1_{V\setminus A}$ cannot be
an eigenvector w.r.t. the largest eigenvalue of the  1-Laplacian, $\forall A\subset V$.
\end{remark}

 We conclude the following  result, which asserts that every 
vector 
in $\{-1,1\}^n$ is an eigenvector of $(f^L,\|\cdot\|_\infty)$ if $f^L$ is nonnegative,  where $f^L$ indicates the disjoint-pair Lov\'asz extension of $f$. 

\begin{pro}\label{pro:set-pair-infty-norm}
Let $f:\power_2(V)\to[0,+\infty)$ be nonnegative. Then, for any $A\subset V$,  $\vec 1_A-\vec 1_{V\setminus A}$ is an eigenvector of  $(f^L,\|\cdot\|_\infty)$.
\end{pro}

\begin{proof}
By the definition of the eigenvalue problem of $(f^L,\|\cdot\|_\infty)$, we only need to prove that $\nabla f^L(\vec x)\cap ((-\R^n_{\mathrm{sign}(x)})\cup \R^n_{\mathrm{sign}(x)}) \ne \varnothing$ for any $\vec x\in\{-1,1\}^n$, where $\R^n_{\mathrm{sign}(x)}:=\{\vec y\in\R^n:y_ix_i\ge 0,\forall i\}=\mathrm{cone}(\nabla \|\vec x\|_\infty)$. We first assume that $f$ is positive-definite, i.e., $f(A,B)>0$ whenever $(A,B)\ne(\varnothing,\varnothing)$, and we shall apply the following argument about polar cones to this case. 

\vspace{0.16cm}

\textbf{Argument}. Let $C$ and $\Omega$ be two convex cones in $\R^n$ such that $\Omega\cap ((-C)\cup C)=\{\vec0\}$. Then $\Omega^*\cap C^*\ne\{\vec 0\}$ and $\Omega^*\cap (-C^*)\ne\{\vec 0\}$, where $C^*$ indicates the polar cone of $C$. 

Proof: Indeed, $\Omega^*\cap C^*=(\Omega\cup C)^*\supset (\Omega+ C)^*$, where $\Omega+ C$ is the Minkowski summation of $C$ and $\Omega$. If $\Omega+ C=\R^n$, then for any $-\vec c\in (-C)\setminus\{\vec0\}$, there exist $\vec a\in\Omega\setminus\{\vec0\}$ and $\vec c'\in C\setminus\{\vec0\}$ such that $\vec a+\vec c'=-\vec c$. This implies $\vec a=-\vec c'-\vec c\in -C\setminus\{\vec0\}$, which contradicts  the condition that $\Omega\cap(-C)=\{\vec0\}$. Therefore, the convex cone $\Omega+ C$ is not the whole space $\R^n$, which implies that $(\Omega+ C)^*\ne\{\vec0\}$. Consequently, $\Omega^*\cap C^*\ne\{\vec 0\}$ and similarly, $\Omega^*\cap (-C^*)\ne\{\vec 0\}$. The proof is completed.

\vspace{0.16cm}

Suppose on the contrary, that $\nabla f^L(\vec x)\cap ((-\R^n_{\mathrm{sign}(x)})\cup \R^n_{\mathrm{sign}(x)}) =\varnothing$ for some $\vec x\in\{-1,1\}^n$. 
Fixing such an $\vec x$, then $\mathrm{cone}(\nabla f^L(\vec x))\cap ((-\R^n_{\mathrm{sign}(x)})\cup \R^n_{\mathrm{sign}(x)}) =\{\vec0\}$, and by the above  argument, we have  $\mathrm{cone}^*(\nabla f^L(\vec x))\cap\R^n_{\mathrm{sign}(x)}=\mathrm{cone}^*(\nabla f^L(\vec x))\cap(-\R^n_{\mathrm{sign}(x)})^* \ne\{\vec0\}$. 

However, since $f$ is positive-definite, it is known  
that  $\mathrm{cone}^*(\nabla f^L(\vec x))\subset T_x(\{\vec y:f^L(\vec y)\le f^L(\vec x)\})$, meaning that $ T_x(\{\vec y:f^L(\vec y)\le f^L(\vec x)\})\cap \R^n_{\mathrm{sign}(x)}\ne\{\vec0\}$, where $T_x$ represents the tangent cone at $\vec x$. Now, suppose $\vec x=\vec 1_{A_n}-\vec 1_{B_n}$ with $A_n\sqcup B_n=V$. Every  permutation  $\sigma:\{1,\cdots,n\}\to \{1,\cdots,n\}$ determines a  sequence $\{(A_i,B_i):i=1,\cdots,n\}\subset \power_2(V)\setminus\{(\varnothing,\varnothing)\}$ by the iterative construction:   $A_1\cup B_1=\{\sigma(1)\}$ and  $A_{i+1}\cup B_{i+1}=A_i\cup B_i\cup \{\sigma(i+1)\}$, $i=1,\cdots,n-1$. 

Since $f(A_i,B_i)>0$, $f^L(\vec x^i)=1$ where $\vec x^i:=(\vec1_{A_i}-\vec1_{B_i})/f(A_i,B_i)$, $i=1,\cdots,n$. Also, $T_{x^n}\{\vec y:f^L(\vec y)\le 1\}=T_x(\{\vec y:f^L(\vec y)\le f^L(\vec x)\})$.
Without loss of generality, we may assume that $\vec x=\vec x^n$. 

The definition of $f^L$ yields that $\mathrm{conv}(\vec0,\vec x^1,\cdots,\vec x^n)\subset \{\vec y:f^L(\vec y)\le 1\}$.  We denote by $\triangle_\sigma=\mathrm{conv}(\vec0,\vec x^1,\cdots,\vec x^n)$ since the  construction of $\vec x^1,\cdots,\vec x^n$ depends on the permutation  $\sigma$. 
For any $\vec y=\sum_{i=1}^nt_i\vec x^i\in\mathrm{conv}(\vec0,\vec x^1,\cdots,\vec x^n)\setminus \{\vec x\}$,  
$(\vec y-\vec x)_{\sigma(n)}x_{\sigma(n)}=-(1-t_n)x_{\sigma(n)}^2<0$, and thus  $\vec y-\vec x\not\in \R^n_{\mathrm{sign}(x)}$. 
Hence, $T_x(\triangle_\sigma)\cap \R^n_{\mathrm{sign}(x)}=\{\vec0\}$. It follows from the fact  $T_x(\{\vec y:f^L(\vec y)\le 1\})=\bigcup_{\sigma}T_x(\triangle_\sigma)$ that $T_x(\{\vec y:f^L(\vec y)\le 1\})\cap \R^n_{\mathrm{sign}(x)}=\{\vec0\}$. This is a contradiction. 

\vspace{0.16cm}

Now we turn to the general case that $f\ge 0$. Take a sequence $\{f_n\}_{n\ge 1}$ of  positive-definite functions on $\power_2(V)$ such that $f_n\to f$ as $n$ tends to $+\infty$. Then it can be verified that for any $\vec v_n\in \nabla f_n^L(\vec x)$, all limit  points of $\{\vec v_n\}_{n\ge 1}$ belong to $\nabla f^L(\vec x)$. Now, there exist $\vec u_n\in \nabla \|\vec x\|_{\infty}$ and $\lambda_n=f^L_n(\vec x)/\|\vec x\|_{\infty}>0$ such that $\lambda_n\vec u_n\in \nabla f^L_n(\vec x)$. Then for any limit point $\vec u$ of $\{\vec u_n\}_{n\ge 1}$, $\vec u\in  \nabla \|\vec x\|_{\infty}$ and $\lambda\vec u\in \nabla f^L(\vec x)$ where $\lambda=\lim\limits_{n\to+\infty}\lambda_n$. Therefore, $(\lambda,\vec x)$ is an eigenpair of $(f^L,\|\cdot\|_\infty)$. 

The proof is completed. 
\end{proof}

By Propositions \ref{pro:Lovasz-eigen} and  \ref{pro:set-pair-infty-norm}, we have 
\begin{cor}
If $2f(A,B)=f(A,V\setminus A)+f(V\setminus B,B)$ for any $(A,B)\in \power_2(V)$, then the set of  eigenvalues of $(f^L,\|\cdot\|_\infty)$ coincides with $\{f(A,V\setminus A):A\subset V\}$, and every vector in $\{-1,1\}^n$ is an eigenvector.
\end{cor}

\begin{remark}
The proof of  Proposition  \ref{pro:set-pair-infty-norm} heavily depends on
the property  that $N_v(X)=-T_v(X)$ for any vertex $v$ of the  hypercube  $X:=\{\vec x:\|\vec x\|_\infty\le1\}$.  Characterizing  the class of  polytopes satisfying   $N_v=-T_v$ for any vertex $v$  remains an open problem, where $N_v$ is the normal cone at $v$ and $T_v$ is the tangent  cone at $v$. 
\end{remark}

Motivated by Propositions \ref{pro:Lovasz-eigen} and  \ref{pro:set-pair-infty-norm}, we suggest a combinatorial eigenvalue problem for $(f,g)$ as follows:

Given $A\subset V$ and a permutation  $\sigma:\{1,\cdots,n\}\to \{1,\cdots,n\}$, there exists a unique sequence $\{(A_i^\sigma,B_i^\sigma):i=1,\cdots,n\}\subset \power_2(V)\setminus\{(\varnothing,\varnothing)\}$   satisfying $A_1^\sigma\subset \cdots \subset A_n^\sigma=A$, $B_1^\sigma\subset \cdots \subset B_n^\sigma=V\setminus A$, $A_1^\sigma\cup B_1^\sigma=\{\sigma(1)\}$ and  $A_{i+1}^\sigma\cup B_{i+1}^\sigma=A_i^\sigma\cup B_i^\sigma\cup \{\sigma(i+1)\}$, $i=1,\cdots,n-1$. 
Let $\vec u^{A,\sigma}\in\R^n$ be defined by
$$u^{A,\sigma}_i=\begin{cases} f(A_{\sigma^{-1}(i)}^\sigma,B_{\sigma^{-1}(i)}^\sigma)-f(A_{\sigma^{-1}(i)-1}^\sigma,B_{\sigma^{-1}(i)-1}^\sigma)
,&\text{ if }i\in A,\\
f(A_{\sigma^{-1}(i)-1}^\sigma,B_{\sigma^{-1}(i)-1}^\sigma)-f(A_{\sigma^{-1}(i)}^\sigma,B_{\sigma^{-1}(i)}^\sigma)
,&\text{ if }i\not\in A.\end{cases}
$$
Denote by $S(f,A)=\{\vec u^{A,\sigma}:\sigma\in S_n\}$ and
$$\nabla f(A,B):=\mathrm{conv}\left(\bigcup\limits_{\tilde{A}:\,A\subset \tilde{A}\subset V\setminus B}S(f,\tilde{A})\right),\;\;\forall (A,B)\in\power_2(V)$$
where $S_n$ is the permutation group over $\{1,\cdots,n\}$.

\begin{defn}[Combinatorial eigenvalue problem]\label{def:comb-eigen}
Given $f,g:\power_2(V)\to \R$, the combinatorial eigenvalue problem of $(f,g)$ is to find $\lambda\in\R$ and $(A,B)\in \power_2(V)\setminus\{(\varnothing,\varnothing)\}$ such that $\nabla f(A,B)\cap \lambda \nabla g(A,B)\ne\varnothing$,   in which $\lambda$ is called an eigenvalue, and $(A,B)$ is called an eigenset.
\end{defn}

Since it can be verified that $\nabla f^L(\vec1_A-\vec1_B)=\nabla f(A,B)$, Proposition \ref{pro:Lovasz-eigen} (or Theorem \ref{introthm:eigenvalue}) implies that the combinatorial eigenvalue problem for $(f,g)$ is equivalent to the nonlinear eigenvalue problem of $(f^L,g^L)$.

By Propositions \ref{pro:Lovasz-eigen-pre} and \ref{pro:Lovasz-eigen}, for a pair of functions $f$ and $g$ on $\power(V)$ (resp.,  $\power_2(V)$), every eigenvalue of the function pair   $(f^L,g^L)$ generated by Lov\'asz extension has an eigenvector of the form $\vec1_A$ (resp., $\vec1_A-\vec1_B$) for some $A\in\power(V)\setminus\{\varnothing\}$ (resp., $(A,B)\in\power_2(V)\setminus\{(\varnothing,\varnothing)\}$). We  call such a set $A$ (resp., $(A,B)$) an  eigen-set of $(f,g)$. And, we are interested in the eigen-sets and the corresponding eigenvalues, which  encode the 
key information about the data structure generated by the function pair $(f,g)$.     The spectrum of $(f^L,g^L)$  provides a way to understand the interaction between  data on $f$ and data on $g$.  

Next, we study the second eigenvalue of the function pair $(f^L,g^L)$, which
is closely related to a combinatorial Cheeger-type constant  of the form 
$$\Ch(f,g):=\min\limits_{A\in\mathcal{P}(V)\setminus\{\varnothing,V\}}\frac{f(A)}{\min\{g(A),g(V\setminus A)\}}$$
where $f:\power(V)\to\R$ is symmetric, i.e., $f(A)=f(V\setminus A)$, $\forall A$, and $g:\power(V)\to\R_+$ is  submodular and non-decreasing. 

\begin{pro}\label{eq:Cheeger-identity}
Let $f_s,g_s:\power_2(V)\to\R$ be defined by $f_s(A,B)=f(A)+f(B)$ and $g_s(A,B)=g(A)+g(B)$. Then 
$$\Ch(f,g)=\text{the second eigenvalue of the function pair }(f_s^L,g_s^L)\;(\text{or equivalently }(f_s,g_s)).$$
\end{pro}

We need the following  auxiliary proposition.
\begin{pro}\label{pro:original-vs-disjoint-pair}Suppose that $g:\power(V)\to\R_+$ is non-decreasing, i.e., $g(A)\le g(B)$ whenever $A\subset B$. 
Let $G:\R^n\to\R$ be the disjoint-pair Lov\'asz extension of the function $(A,B)\mapsto g(A)+g(B)$.   Then the  Lov\'asz extension of the function $A\mapsto\min\{g(A),g(V\setminus A)\}$ is $\min\limits_{t\in\R}G(\vec x-t\vec1)$. 
\end{pro}
\begin{proof}
We put $g_{m}(A)=\min\{g(A),g(V\setminus A)\}$ and $g_s(A,B)=g(A)+g(B)$, where $g_{m}^L$ is the original Lov\'asz extension of  $g_{m}$, and $g_s^{L}$ is the disjoint-pair Lov\'asz extension of $g_s$. 
Since $g$ is non-decreasing, $g(V^t(\vec x))$ must be non-increasing on $t\in \R$, i.e., $g(V^{t_1}(\vec x))\ge g(V^{t_2}(\vec x))$ whenever $t_1\le t_2$. Hence, there exists $t_0\in\R$ such that $g(V^t(\vec x))\ge g(V\setminus V^t(\vec x))$, $\forall t\le t_0$; and   $g(V^t(\vec x))\le g(V\setminus V^t(\vec x))$, $\forall t\ge t_0$.
Then \begin{align*}
g_{m}^L(\vec x)&= \int_{\min \vec x}^{\max \vec x}g_{m}(V^t(\vec x))dt+\min\vec x g_{m}(V)
\\&=\int_{\min \vec x}^{t_0}g(V\setminus V^t(\vec x))dt+\int_{t_0}^{\max\vec x}g(V^t(\vec x))dt
\\&=\int_{\min (\vec x-t_0\vec1)}^0g(V\setminus V^t(\vec x-t_0\vec1))dt+\int_{0}^{\max(\vec x-t_0\vec1)}g(V^t(\vec x-t_0\vec1))dt
\\&=\int_{-\|\vec x-t_0\vec1\|_\infty}^0g(V\setminus V^t(\vec x-t_0\vec1))dt+\int_{0}^{\|\vec x-t_0\vec1\|_\infty}g(V^t(\vec x-t_0\vec1))dt
\\&=\int_{0}^{\|\vec x-t_0\vec1\|_\infty}g(V^t(\vec x-t_0\vec1))+g(V\setminus V^{-t}(\vec x-t_0\vec1))dt
\\&=g_{s}^{L}(\vec x-t_0\vec1)=\min\limits_{t\in\R}g_s^{L}(\vec x-t\vec1).
\end{align*}
The proof is completed. 
\end{proof}

\begin{proof}[Proof of Proposition \ref{eq:Cheeger-identity}]
Since $f$ is symmetric, by Proposition \ref{pro:setpair-original}, $f_s^L(\vec x)=f^L(\vec x)=f_m^L(\vec x)$, where $f_m$ is defined by  $f_{m}(A):=\min\{f(A),f(V\setminus A)\}$, and $f_m^L$ is the original Lov\'asz extension of $f_m$. 

Since $g$ is positive,  submodular and non-decreasing, it is not difficult to check that 
$g_s$ is bisubmodular.  Thus, by the equivalence of submodularity and convexity, $g_s^L$ is a convex function. 
Therefore, we have
\begin{align*}
\min\limits_{\vec x\bot\vec 1}\frac{f_s^L(\vec x)}{\min\limits_{t\in\R}g_s^L(\vec x-t\vec 1)}&=\min\limits_{\text{nonconstant }\vec x\in\R_+^n}\frac{f_s^L(\vec x)}{\min\limits_{t\in\R}g_s^L(\vec x-t\vec 1)}\\&=\min\limits_{\vec x\in\R_+^n:\min\vec x=0}\frac{f_m^L(\vec x)}{g_m^L(\vec x)}=\min\limits_{A\ne\varnothing,V}\frac{f_m(A)}{g_m(A)}=\Ch(f,g),
\end{align*}
where the first equality is based on the fact that $\vec x\mapsto f_s^L(\vec x)=f^L(\vec x)$ and $\vec x\mapsto\min\limits_{t\in\R}g_s^L(\vec x-t\vec 1)$ are translation invariant along $\vec1$, the second equality is derived by Proposition \ref{pro:original-vs-disjoint-pair}, and the third one follows from Theorem \ref{thm:tilde-fg-equal}. 
 It follows from the nonlinear eigenvalue  characterization (Theorem 2.1 and Proposition 2.4 in \cite{JostZhang})  that $$\min\limits_{\vec x\bot\vec 1}\frac{f_s^L(\vec x)}{\min\limits_{t\in\R}g_s^L(\vec x-t\vec 1)}=\min\limits_{\vec x\text{ nonconstant}}\frac{f_s^L(\vec x)}{\min\limits_{t\in\R}g_s^L(\vec x-t\vec 1)}$$ is actually the second eigenvalue of the function pair $(f_s^L,g_s^L)$.
\end{proof}
Finally, we prove that for any $A,B\ne \varnothing$ with $A\cap B=\varnothing$,  $$\max\{\frac{f(A)}{g(A)},\frac{f(B)}{g(B)}\}\ge\min\{\frac{f(A)}{\min\{g(A),g(V\setminus A)\}},\frac{f(B)}{\min\{g(B),g(V\setminus B)\}}\}.$$
Suppose the contrary, and keep $f(A)=f(V\setminus A)$ in mind. Then, we have $g(A)>g(V\setminus A)$ and $g(B)>g(V\setminus B)$, implying $g(A)+g(B)>g(V\setminus A)+g(V\setminus B)$. 
Since $A\subset V\setminus B$ and $g$ is non-decreasing, one has $g(A)\le g(V\setminus B)$. Similarly, $g(B)\le g(V\setminus A)$, which leads to a contradiction. 

Combining all the results and discussions in this section, we complete the proof of Theorem \ref{introthm:eigenvalue}. 

In a general form, given  $\A\in\{\power(V),\power_2(V),\power(V_1)\times\cdots\times\power(V_k),\power_2(V_1)\times\cdots\times\power_2(V_k)\}$,   
for $f=\sum_i f_i$ and $g=\sum_j g_j$, where $f_i,g_j:\A\to\R$, one can also define the subgradient for functions on $\A$  via $\nabla f(A)=\nabla f^L(\vec 1_A)$  and 
define the combinatorial eigenvalue problem 
\begin{equation}\label{eq:relax-eigen-combina}
\vec 0\in\sum_i\nabla f_i(A)-\lambda\sum_j\nabla g_j(A)    
\end{equation}
which is  a slightly extended 
version of the combinatorial  eigenvalue problem for  $(f,g)$ (see Definition \ref{def:comb-eigen}). We shall note that \eqref{eq:relax-eigen-combina} is  equivalent in some sense to the nonlinear eigenvalue problem:
\begin{equation}\label{eq:relax-eigen-lovasz}
\vec 0\in\sum_i\nabla f_i^L(\vec x)-\lambda\sum_j\nabla g_j^L(\vec x).    
\end{equation}
In fact, similar to  Propositions \ref{pro:Lovasz-eigen-pre} and \ref{pro:Lovasz-eigen}, we have:
\begin{pro}\label{pro:general-eigen-equiv}
Any  associate set-tuple of any eigenvector $\vec x$ of the nonlinear eigenvalue problem \eqref{eq:relax-eigen-lovasz} is an eigenset of  the combinatorial  eigenvalue problem \eqref{eq:relax-eigen-combina}.  Conversely, for any eigenset $A$  satisfying   \eqref{eq:relax-eigen-combina}, its  indicator vector $\vec1_A$ is an eigenvector $\vec x$ satisfying  \eqref{eq:relax-eigen-lovasz}.

Any eigenpair  $(\lambda,\vec x)$ of $(f^L,g^L)$ satisfies 
\eqref{eq:relax-eigen-lovasz}. Moreover, the minimum and maximum eigenvalues of \eqref{eq:relax-eigen-lovasz} and their corresponding eigenvectors are also that  of $(f^L,g^L)$. 
\end{pro}

\subsection{
Dinkelbach-type schemes and mixed 
IP-SD algorithms}
 \label{sec:algo}

We would like to establish an iteration framework for finding minimal  and maximal eigenvalues.  These extremal eigenvalues play significant roles in optimization theory. They can be found via the  so-called Dinkelbach iterative scheme \cite{D67}. This will provide  a good starting point for an appropriate 
iterative algorithm for the resulting fractional programming. Actually, the equivalent continuous optimization has a fractional form, but such kind of fractions have been hardly touched in the field of fractional programming \cite{SI83}, where optimizing the ratio of a concave function to a convex one is usually considered. 

\begin{theorem}[Global convergence of a Dinkelbach-type  scheme \cite{D67}] \label{thm:global convergence}
Let $S$ be a compact set and let $F,G:S\to\mathbb{R}$ be two continuous functions with $G(\vec x)>0$, $\forall \vec x\in S$. Then the sequence $\{r^k\}$ generated by the two-step iterative scheme
\begin{numcases}{}
\vec x^{k+1}=\argopti\limits_{\vec  x\in S} \{F(\vec x)-r^k G(\vec x)\}, \label{iter0-1}
\\
r^{k+1}=\frac{F( \vec x^{k+1})}{G(\vec x^{k+1})},
\label{iter0}
\end{numcases}
 from any initial point $\vec  x^0\in S$, converges monotonically to a global optimum of $F(\cdot)/G(\cdot)$,
where `opti' is `min' or `max'.
\end{theorem}

\begin{cor}
If $F/G$ is a zero-homogeneous continuous function, then the iterative scheme \eqref{iter0-1}\eqref{iter0} from any initial point $\vec  x^0$ converges monotonically to a global optimum on the cone spanned by $S$ (i.e., $\{t\vec x: t>0, \vec x\in S\}$).
\end{cor}

We note that Theorem \ref{thm:global convergence} generalizes Theorem 3.1 in \cite{CSZ15} and Theorem 2 in \cite{CSZ18}.  Since it is a Dinkelbach-type iterative algorithm in the field of fractional programming, we omit the proof.

Many minimization problems in the field of fractional programming possess the form
$$
\min\,\frac{\text{convex }F}{\text{concave }G},
$$
which is not necessary for a  convex programming problem. The original Dinkelbach iterative scheme turns the ratio form to the inner problem \eqref{iter0-1} with the form like
$$
\min \; (\text{convex }F-\text{concave }G),
$$
which is indeed a convex programming problem. However, most of our examples are in the form
$$
\min\frac{\text{convex }F}{\text{convex }G},
$$
i.e., both the numerator and the denominator of the fractional object function are convex.
Since the difference of two convex functions may not be convex, the inner
problem \eqref{iter0-1} is no longer a convex optimization problem and hence might be very difficult to solve. 

In other practical applications, we may encounter optimization problems of the form
\begin{equation}\label{eq:RatioDC-form}
\min\frac{\text{convex }F_1 - \text{convex }F_2}{\text{convex }G_1- \text{convex }G_2}.    
\end{equation}
This is NP-hard in general. Fortunately, we can construct an effective relaxation of \eqref{iter0-1}.

The starting point of the relaxation step is the following  classical fact:
\begin{pro}\label{pro:difference-two-submodular}
For any function $f:\A\to \R$, there are two submodular functions $f_1$ and $f_2$ on $\A$ such that $f=f_1-f_2$.
\end{pro}

Although this is  an old result,  for readers' convenience, we present a short
proof below. 

\begin{proof} We put  $$\delta(g):=\min\limits_{A\ne A'\in \A}\left(g(A)+g(A')-g(A\vee A')-g(A\wedge A')\right).$$ 
Recall that a function $g:\A\to\R$ is  {\sl strictly submodular} if $g(A)+g(A')>g(A\vee A')+g(A\wedge A')$ whenever $A\ne A'\in\A$. Since $\A$ has finitely many elements, it is known that there always exists a strict submodular function on $\A$. Clearly, $g$ is  submodular if and only if $\delta(g)\ge0$, while $g$ is strictly submodular if and only if $\delta(g)>0$. 
Let 
 $g:\A\to\R$  be  strictly  submodular, and pick  $C>\max\{\frac{\delta(f)}{\delta(g)},0\}$. 
Take $f_2=Cg$ and $f_1=f+f_2$. It is clear that $\delta(f_2)=C\delta(g)>0$ and $\delta(f_1)\ge \delta(f)+ \delta(f_2)= \delta(f)+C\delta(g)\ge \delta(f)+\frac{\delta(f)}{\delta(g)}\delta(g)=0$. Therefore, we have the decomposition  $f=f_1-f_2$, where $f_2$ is strictly submodular and $f_1$ is submodular.
\end{proof}

Thanks to Proposition \ref{pro:difference-two-submodular}, any discrete function can be expressed as the difference of two submodular functions. Since the Lov\'asz extension of a submodular function is convex, every Lov\'asz extension function is the difference of two convex functions. 

Then, for the fractional programming derived by Theorem \ref{thm:tilde-fg-equal} (or Propositions  \ref{pro:fraction-f/g} and \ref{pro:maxconvex}),  both the numerator
and denominator can be rewritten as the differences of two convex functions. This  implies  that a simple iterative  
algorithm can be obtained via further relaxing the Dinkelbach iteration by
techniques in DC Programming \cite{HT99}. It should be noted that the following recent works (especially the papers by Hein  et al \cite{HeinBuhler2010,HS11,TVhyper-13,TMH18}) motivated us to investigate more on this direction: 
\begin{enumerate}
\item  The efficient generalization of the inverse power method proposed by Hein et al  \cite{HeinBuhler2010} and the extended steepest descent 
method by Bresson et al  \cite{BLUB12}  deal with fractional programming  in the same spirit. For more relevant papers,  
we refer to \cite{HS11} for the RatioDCA method, and  \cite{TMH18} for the generalized RatioDCA technique.
\item In \cite{MaeharaMurota15,MMM18}, the authors address difference convex programming (DC programming)
for discrete convex functions, in which  an algorithm and a convergence result similar to Theorem \ref{th:gsd} are presented. 

\item A simple iterative algorithm based on the continuous reformulation by the disjoint-pair Lov\'asz extension provides the best cut values for maxcut on a G-set among all existing continuous algorithms \cite{SZZmaxcut}. 
\end{enumerate}
 In view of these recent  developments, and in order to enlarge the scope of fractional programming and RatioDCA method, it is helpful to study this
  aspect by general formulations (see also Remark \ref{remark:very-general-RatioDCA} for the most general form).   
  Thus,  
we begin to establish a method based on convex programming for solving $\min
\frac{F(\vec x)}{G(\vec x)}$ with $F=F_1-F_2$ and $G=G_1-G_2$ being two nonnegative functions, where $F_1,F_2,G_1,G_2$ are four  nonnegative convex functions on $X$. For any $\vec y\in X$, let $H_{\vec y}:X\to\R$ be a convex differentiable function such that $\vec y$ is a minimizer of $H_{\vec y}$. For example, we may simply take $H_{\vec y}(\vec x)=\|\vec x-\vec y\|_2^2$.  
Consider the following three-step iterative scheme
\begin{subequations}
\label{iter1}
\begin{numcases}{}
\vec  x^{k+1}\in \argmin\limits_{\vec x\in \mathbb{B}} \{F_1(\vec x)+r^k G_2(\vec x) -(\langle \vec u^k,\vec x\rangle+r^k \langle \vec v^k,\vec x\rangle) + H_{\vec x^k}(\vec x)\}, \label{eq:twostep_x2}
\\
r^{k+1}=F( \vec x^{k+1})/G( \vec x^{k+1}),
\label{eq:twostep_r2}
\\
 \vec u^{k+1}\in\nabla F_2( \vec x^{k+1}),\;
 \vec v^{k+1}\in\nabla G_1( \vec x^{k+1}),
\label{eq:twostep_s2}
\end{numcases}
\end{subequations}
where $\mathbb{B}$ is a convex body containing $\vec 0$ as its inner point.  The following slight modification  
\begin{subequations}
\label{iter2}
\begin{numcases}{}
\vec  y^{k+1}\in \argmin\limits_{\vec x\in X} \{F_1(\vec x)+r^k G_2(\vec x) -(\langle \vec u^k,\vec x\rangle+r^k \langle \vec v^k,\vec x\rangle) + H_{\vec x^k}(\vec x)\}, \label{eq:2twostep_x2}
\\
r^{k+1}=F( \vec y^{k+1})/G( \vec y^{k+1}),~~ \vec x^{k+1}=\partial \mathbb{B}\cap\{t\vec y^{k+1}:t\ge 0\} 
\label{eq:2twostep_r2}
\\
 \vec u^{k+1}\in\nabla F_2( \vec x^{k+1}),\;
 \vec v^{k+1}\in\nabla G_1( \vec x^{k+1}),
\label{eq:2twostep_s2}
\end{numcases}
\end{subequations}
 is available when $F/G$ is zero-homogeneous and \eqref{eq:2twostep_x2} has a
 solution. In  \eqref{eq:2twostep_r2}, $\vec x^{k+1}$ indicates  the
 normalization of $\vec y^{k+1}$ w.r.t. the convex body $\mathbb{B}$; in  particular, $\vec x^{k+1}:=\vec y^{k+1}/\|\vec y^{k+1}\|_2$ if we let  $\mathbb{B}$ be the unit ball.     
 These schemes  mixing the inverse power (IP) method and steepest descent (SD) method can be well used in computing special eigenpairs of $(F,G)$. Note that the inner problem \eqref{eq:twostep_x2} (resp.  \eqref{eq:2twostep_x2}) is a convex optimization and thus many algorithms in convex programming are applicable. We should note that the above schemes provide a generalization of the RatioDCA  technique in \cite{HS11},    and we  establish our proof by revising
the technique in \cite{HS11,HeinBuhler2010}.

\begin{theorem}[Local convergence for the  mixed IP-SD scheme]\label{th:gsd}
The sequence $\{r^k\}$ generated by the iterative scheme \eqref{iter1} (resp. \eqref{iter2}) from any initial point $\vec x^0\in \mathrm{supp}(G)\cap \mathbb{B}$ (resp. $\vec x^0\in \mathrm{supp}(G)$) converges monotonically, where $\mathrm{supp}(G)$ is the support of $G$.

Next we further assume that $X$ is of finite dimension. If one of the following  additional  conditions  holds, then  $\lim_{k\to+\infty} r^k=r^*$ is an eigenvalue of the function pair $(F,G)$ 
in the sense that it fulfills $\vec0\in \nabla F_1(\vec x^*)-\nabla F_2(\vec x^*)-r^*\left(\nabla G_1(\vec x^*)-\nabla G_2(\vec x^*)\right)$,  where $\vec x^*$ is a  cluster point of $\{\vec x^k\}$.

\begin{itemize}
\item[Case 1.] For the  scheme \eqref{iter1}, $F_2$ and $G_1$ are one-homogeneous, and  $F_1$ and $G_2$ are $p$-homogeneous with $p\ge 1$, and $H_{\vec x}=\text{const}$, $\forall \vec x\in\mathbb{B}$. 
\item[Case 2.1.]  For the  scheme \eqref{iter2}, $F_1$, $F_2$, $G_1$ and $G_2$ are $p$-homogeneous with  $p>1$.
\item[Case 2.2.]  For the  scheme \eqref{iter2}, $F_1$, $F_2$, $G_1$ and $G_2$
  are one-homogeneous, and $H_{\vec x}(\vec x)$ is a continuous
  function of $\vec x\in \mathbb{B}$ and  $\forall M>0$,  $\exists C>0$ such that   $H_{\vec x}(\vec y)>M \|\vec y\|_2$  whenever $\vec x\in\mathbb{B}$ and  $\|\vec y\|_2\ge C$. 
\end{itemize}

\end{theorem}

Theorem \ref{th:gsd} partially generalizes  Theorem 3.4 in  \cite{CSZ15}, 
and it is indeed an extension of both the IP and the SD method \cite{BLUB12,CP11,M16,HeinBuhler2010}.

\begin{proof}[Proof of Theorem \ref{th:gsd}]

It will be helpful to divide this proof into several parts and steps:

\begin{enumerate}
\item[Step 1.] We may assume $G(\vec x^k)>0$ for any $k$. 
In fact, the initial point $\vec x^0$ satisfies $G(\vec x^0)> 0$. We will show $F(\vec x^1)=0$ if $G(\vec x^1)=0$ and thus the iteration should be terminated at $\vec x^1$. This tells us that we may assume $G(\vec x^k)>0$ for all $k$ before the termination of the iteration.

Note that
\begin{align*}&F_1(\vec x^1)+r^0 G_2(\vec x^1) -(\langle \vec u^0,\vec x^1\rangle+r^0 \langle \vec v^0,\vec x^1\rangle) + H_{\vec x^0}(\vec x^1)\\ \le~& F_1(\vec x^0)+r^0 G_2(\vec x^0) -(\langle \vec u^0,\vec x^0\rangle+r^0 \langle \vec v^0,\vec x^0\rangle) + H_{\vec x^0}(\vec x^0),
\end{align*}
which implies
\begin{align*}&F_1(\vec x^1)-F_1(\vec x^0)+r^0 (G_2(\vec x^1)-G_2(\vec x^0)) + H_{\vec x^0}(\vec x^1)-H_{\vec x^0}(\vec x^0)\\ \le~& \langle \vec u^0,\vec x^1-\vec x^0\rangle+r^0 \langle \vec v^0,\vec x^1-\vec x^0\rangle\le F_2(\vec x^1)-F_2(\vec x^0) +r^0 (G_1(\vec x^1)-G_1(\vec x^0)),
\end{align*}
i.e.,
\begin{align}F(\vec x^1)-F(\vec x^0)+  H_{\vec x^0}(\vec x^1)-H_{\vec x^0}(\vec x^0)&\le r^0 (G(\vec x^1)-G(\vec x^0))\label{eq:important-inequality}\\&=-r^0G(\vec x^0)=-F(\vec x^0).\notag
\end{align}
Since the equality holds, we have $F(\vec x^1)=0$, $H_{\vec x^0}(\vec x^1)=H_{\vec x^0}(\vec x^0)$, $\langle \vec u^0,\vec x^1-\vec x^0\rangle=F_2(\vec x^1)-F_2(\vec x^0)$ and $\langle \vec v^0,\vec x^1-\vec x^0\rangle=G_1(\vec x^1)-G_1(\vec x^0)$. So this step is finished.

\item[Step 2.] $\{r^k\}_{k=1}^\infty$ is monotonically decreasing and hence convergent.

Similar to \eqref{eq:important-inequality} in Step 1, we can arrive at
$$F(\vec x^{k+1})-F(\vec x^k)+  H_{\vec x^k}(\vec x^{k+1})-H_{\vec x^k}(\vec x^k) \le r^k (G(\vec x^{k+1})-G(\vec x^k)),$$
which leads to
$$F(\vec x^{k+1})\le r^k G(\vec x^{k+1}).$$
Since $G(\vec x^{k+1})$ is assumed to be positive,
$r^{k+1}=F(\vec x^{k+1})/G(\vec x^{k+1})\le r^k$.
 Thus, there exists $r^*\in [r_{\min},r^0]$ such that $\lim\limits_{k\to+\infty}r^k=r^*$, where $r_{\min}:=\min_{x\ne0} F(\vec x)/G(\vec x)$.
\end{enumerate}

In the sequel, we assume that the  dimension of $X$ is finite.

\begin{enumerate}
\item[Step 3.]   $\{\vec x^k\}$, $\{\vec u^k\}$ and $\{\vec v^k\}$ are sequentially compact.

In this setting, $\mathbb{B}$ must be compact. In consequence, there exist $k_i$, $r^*$, $\vec x^*$,  $\vec x^{**}$, $\vec u^*$ and $\vec v^*$ such that $\vec x^{k_i}\to \vec x^*$, $\vec x^{k_i+1}\to \vec x^{**}$, $\vec u^{k_i}\to \vec u^*$ and $\vec v^{k_i}\to \vec v^*$, as $i\to +\infty$.

 Clearly, the statements in Steps 1, 2 and 3 are also available for the  scheme \eqref{iter2}.

\item[Step 4.] For the  scheme \eqref{iter1},   $\vec x^*$ is a minimum of $F_1(\vec x)+r^* G_2(\vec x) -(\langle \vec u^*,\vec x\rangle+r^* \langle \vec v^*,\vec x\rangle) + H_{\vec x^*}(\vec x)$ on $\mathbb{B}$.  For the  scheme  \eqref{iter2}, under the additional assumptions introduced in  Case 2.1 or Case  2.2, $\vec x^*$ is a minimum of $F_1(\vec x)+r^* G_2(\vec x) -(\langle \vec u^*,\vec x\rangle+r^* \langle \vec v^*,\vec x\rangle) + H_{\vec x^*}(\vec x)$ on $X$. 

Let $g(r,\vec y,\vec u,\vec v)=\min\limits_{  \vec x\in \mathbb{B}} \{F_1(\vec x)+r G_2(\vec x) -(\langle \vec u,\vec x\rangle+r \langle \vec v,\vec x\rangle) + H_{\vec y}(\vec x)\}$. It is standard to verify that $g(r,\vec y,\vec u,\vec v)$ is continuous on $\mathbb{R}^{1}\times X\times X^*\times X^*$ according to the compactness of $\mathbb{B}$.

Since $g(r^{k_i},\vec x^{k_i},\vec u^{k_i},\vec v^{k_i})=r^{k_i+1}$, taking $i\to+\infty$, one obtains $g(r^*,\vec x^*,\vec u^*,\vec v^*)=r^*$.

By Step 3, $\vec x^{**}$ attains the minimum of $F_1(\vec x)+r^* G_2(\vec x) -(\langle \vec u^*,\vec x\rangle+r^* \langle \vec v^*,\vec x\rangle) + H_{\vec x^*}(\vec x)$ on $\mathbb{B}$. Suppose the contrary, that $\vec x^*$ is not a minimum of $F_1(\vec x)+r^* G_2(\vec x) -(\langle \vec u^*,\vec x\rangle+r^* \langle \vec v^*,\vec x\rangle) + H_{\vec x^*}(\vec x)$ on $\mathbb{B}$. Then
\begin{align*}&F_1(\vec x^{**})+r^* G_2(\vec x^{**}) -(\langle \vec u^*,\vec x^{**}\rangle+r^* \langle \vec v^*,\vec x^{**}\rangle) + H_{\vec x^*}(\vec x^{**})\\ <~& F_1(\vec x^*)+r^* G_2(\vec x^*) -(\langle \vec u^*,\vec x^*\rangle+r^* \langle \vec v^*,\vec x^*\rangle) + H_{\vec x^*}(\vec x^*),
\end{align*}
and thus $F(\vec x^{**})<r^* G(\vec x^{**})$ (similar to Step 1), which implies $G(\vec x^{**})>0$ and $F(\vec x^{**})/ G(\vec x^{**})<r^*$. This is a contradiction. Consequently, $  \vec x^*$ is a minimizer of $F_1(\vec x)+r^* G_2(\vec x) -(\langle \vec u^*,\vec x\rangle+r^* \langle \vec v^*,\vec x\rangle) + H_{\vec x^*}(\vec x)$ on $\mathbb{B}$.

 On the scheme  \eqref{iter2}, we refer to   Cases 2.1 and 2.2 below for details.

Next, we will 
verify that $(r^*,\vec x^*)$ is an eigenpair  under certain  additional conditions. 

\item[Case 1.]  On  the scheme \eqref{iter1}, $F_2$ and $G_1$ are one-homogeneous, and  $F_1$ and $G_2$ are $p$-homogeneous with $p\ge 1$, and $H_{\vec x}=Const$, $\forall \vec x\in\mathbb{B}$. 

Since $H_{\vec x^k}=Const$,  the above claim shows that $\vec x^*$ is a minimizer of $F_1(\vec x)+r^* G_2(\vec x) -(\langle \vec u^*,\vec x\rangle +r^* \langle \vec v^*,\vec x\rangle)$ on $\mathbb{B}$. Also, since $F_2$ and $G_1$ are one-homogeneous, the Euler identity on homogeneous functions gives $F_1(\vec x^*)+r^* G_2(\vec x^*) -(\langle \vec u^*,\vec x^*\rangle +r^* \langle \vec v^*,\vec x^*\rangle)=F_1(\vec x^*)+r^* G_2(\vec x^*) -(F_2(\vec x^*)+r^* G_1(\vec x^*))=F(\vec x^*)-r^* G(\vec x^*)=0$. Thus, $F_1(\vec x)+r^* G_2(\vec x) -(\langle \vec u^*,\vec x\rangle+r^* \langle \vec v^*,\vec x\rangle)\ge 0$, $\forall \vec x\in \mathbb{B}$, and the equality holds when $\vec x=\vec x^*$.   

Since $\mathbb{B}$ contains $0$ as its inner point, we have $\{\alpha \vec x: \vec x\in \mathbb{B},\alpha\ge 1\}=X$.
 Keeping  $\alpha\ge 1$ and $p\ge 1$ in mind, for any $\alpha\ge 1$ and $\vec x\in \mathbb{B}$,
 \begin{align*}
 &F_1(\alpha \vec x)+r^* G_2(\alpha \vec x) -(\langle \vec u^*,\alpha \vec x\rangle+r^* \langle \vec v^*,\alpha \vec x\rangle)
  \\ =~& \alpha\left(F_1( \vec x)+r^* G_2( \vec x) -(\langle \vec u^*, \vec x\rangle+r^* \langle \vec v^*, \vec x\rangle)\right)+(\alpha^p-\alpha)(F_1( \vec x)+r^* G_2( \vec x))
  \\ (\text{by Step 4})~\ge~& (\alpha^p-\alpha)(F_1( \vec x)+r^* G_2( \vec x)) \ge 0.
 \end{align*}

Consequently, $\vec x^*$ is a minimizer of $F_1(\vec x)+r^* G_2(\vec x) -(\langle \vec u^*,\vec x\rangle+r^* \langle \vec v^*,\vec x\rangle)$ on $X$, and thus
\begin{align*}
\vec 0 &\in \nabla|_{  \vec x=  \vec x^*} \left(F_1(\vec x)+r^* G_2(\vec x) -(\langle \vec u^*,\vec x\rangle+r^* \langle \vec v^*,\vec x\rangle) \right)
\\&=\nabla F_1(\vec x^*)+r^* \nabla G_2(\vec x^*) -\vec u^*-r^*\vec v^*
\\& \subset \nabla F_1(\vec x^*)-\nabla F_2(\vec x^*)+r^*\nabla G_2(\vec x^*)-r^*\nabla G_1(\vec x^*).
\end{align*}


\item[Case 2.1.] On the  scheme \eqref{iter2}, $F_1$, $F_2$, $G_1$ and $G_2$ are $p$-homogeneous with  $p>1$. 

Denote by $B:X\to[0,+\infty)$ the unique convex and one-homogeneous function satisfying $B(\partial\mathbb{B})=1$. Then the normalization of $\vec x$ in  \eqref{eq:2twostep_r2} can be expressed as $\vec x/B(\vec x)$.

The compactness of $\{\vec x:B(\vec x)\le 1\}$ and the upper semi-continuity and compactness of subderivatives imply that $\bigcup_{\vec x:B(\vec x)\le 1}\nabla F_2(\vec x)$ and $\bigcup_{\vec x:B(\vec x)\le 1}\nabla G_1(\vec x)$ are bounded sets. So, we have a uniform constant $C_1>0$ such that $\|\vec u\|_2+r^*\|\vec v\|_2\le C_1$, $\forall \vec u\in \nabla F_2(\vec x)$, $ \vec v\in \nabla G_1(\vec x)$, $\forall \vec x\in\mathbb{B}$. Let $C_2>0$ be such that  $\|\vec x\|_2\le C_2B(\vec x)$, and  $C_3=\min\limits_{B(\vec x)=1} F_1(\vec x)>0$ (here we assume without loss of generality that $F_1(\vec x)>0$ whenever $\vec x\ne \vec 0$). For any $\vec x$ with $B(\vec x)\ge \max\{2,(2C_1C_2/C_3)^{\frac{1}{p-1}}\}$, and for any $\vec x^*\in \mathbb{B}$, $ \vec u^*\in \nabla F_2(\vec x^*)$, $ \vec v^*\in \nabla G_1(\vec x^*)$, 
\begin{align*}
&F_1(\vec x)+r^* G_2(\vec x) -(\langle \vec u^*,\vec x\rangle+r^*\langle \vec v^*,\vec x\rangle) + H_{\vec x^*}(\vec x)
\\ =~&   B(\vec x)^p F_1(\frac{\vec x}{B(\vec x)})+r^*B(\vec x)^p G_2(\frac{\vec x}{B(\vec x)}) -(\|\vec x\|_2\langle \vec u^*,\frac{\vec x}{\|\vec x\|_2}\rangle+r^*\|\vec x\|_2\langle \vec v^*,\frac{\vec x}{\|\vec x\|_2}\rangle) + H_{\vec x^*}(\vec x)
\\ \ge~& B(\vec x)^pF_1(\frac{\vec x}{B(\vec x)})-\|\vec x\|_2(\|\vec u^*\|_2+r^*\|\vec v^*\|_2 ) + H_{\vec x^*}(\vec x^*)
\\ \ge~& B(\vec x)^pC_3-C_2C_1B(\vec x) + H_{\vec x^*}(\vec x^*)=B(\vec x)(B(\vec x)^{p-1}C_3-C_2C_1) + H_{\vec x^*}(\vec x^*)
> H_{\vec x^*}(\vec x^*)\\ >~& -(p-1)(F_2(\vec x^*)+r^* G_1(\vec x^*))+ H_{\vec x^*}(\vec x^*)
\\ =~& F_1(\vec x^*)+r^* G_2(\vec x^*) -(\langle \vec u^*,\vec x^*\rangle+r^*\langle \vec v^*,\vec x^*\rangle) + H_{\vec x^*}(\vec x^*)
\end{align*}
which means that the  minimizers of $F_1(\vec x)+r^* G_2(\vec x) -(\langle \vec u^*,\vec x\rangle+r^*\langle \vec v^*,\vec x\rangle) + H_{\vec x^*}(\vec x)$ exist and they always lie in the bounded set $\{\vec x:B(\vec x)< \max\{2,(2C_1C_2/C_3)^{\frac{1}{p-1}}\}\}$. Since $B(\vec x^k)=1$, $\{\vec y^k\}$ must be a bounded sequence.  There exists  $\{k_i\}\subset \{k\}$ such that $\vec x^{k_i} \to \vec x^*$, $\vec y^{k_i+1}\to \vec y^{**}$, $\vec x^{k_i+1} \to \vec x^{**}$ for some $\vec x^*$, $\vec y^{**}$ and $\vec x^{**}=\vec y^{**}/B(\vec y^{**})$. Similar to Step 4 and  Case 1,  $\vec x^*$ is a minimizer of $F_1(\vec x)+r^* G_2(\vec x) -(\langle \vec u^*,\vec x\rangle+r^* \langle \vec v^*,\vec x\rangle)+H_{\vec x^*}(\vec x)$ on $X$, and thus
\begin{align*}
\vec 0 &\in \nabla|_{  \vec x=  \vec x^*} \left(F_1(\vec x)+r^* G_2(\vec x) -(\langle \vec u^*,\vec x\rangle+r^* \langle \vec v^*,\vec x\rangle)+H_{\vec x^*}(\vec x) \right)
\\&=\nabla F_1(\vec x^*)+r^* \nabla G_2(\vec x^*) -\vec u^*-r^*\vec v^*
\subset \nabla F_1(\vec x^*)-\nabla F_2(\vec x^*)+r^*\nabla G_2(\vec x^*)-r^*\nabla G_1(\vec x^*).
\end{align*}

\item[Case 2.2.] On the scheme \eqref{iter2},  $F_1$, $F_2$, $G_1$ and $G_2$ are one-homogeneous; $H_x(\vec x)$ is continuous of $\vec x\in \mathbb{B}$ and for any $M>0$, there exists $C>0$ such that $H_{\vec x}(\vec y)>M\cdot B(\vec y)$ whenever $\vec x\in\mathbb{B}$ and  $B(\vec y)\ge C$. 

Taking $M=C_1C_2+2$ in which the constants  $C_1$ and $C_2$ are introduced in Case 2.1, there exists $C>\max\{\max\limits_{x\in \mathbb{B}}H_x(\vec x),1\}$ such that $H_{\vec x^*}(\vec x)\ge M\cdot B(\vec x)$ whenever $\vec x^*\in \mathbb{B}$ and $B(\vec x)\ge C$.  

Similar to Case 2.1, for any $\vec x^*\in \mathbb{B}$,  $\vec x\in X$ with $B(\vec x)\ge C$, and $\forall \vec u^*\in \nabla F_2(\vec x^*)$, $ \vec v^*\in \nabla G_1(\vec x^*)$, 
\begin{align*}
&F_1(\vec x)+r^* G_2(\vec x) -(\langle \vec u^*,\vec x\rangle+r^*\langle \vec v^*,\vec x\rangle) + H_{\vec x^*}(\vec x)
\\>~& B(\vec x)(C_3-C_2C_1) + (C_1C_2+2)\cdot B(\vec x)
\ge 2 B(\vec x) >H_{\vec x^*}(\vec x^*)
\\ =~& F_1(\vec x^*)+r^* G_2(\vec x^*) -(\langle \vec u^*,\vec x^*\rangle+r^*\langle \vec v^*,\vec x^*\rangle) + H_{\vec x^*}(\vec x^*).
\end{align*}
The remaining part can refer to  Case 2.1. 
   \end{enumerate}
\end{proof}

\begin{remark}\label{remark:very-general-RatioDCA}
As some direct extensions of the so-called {\sl generalized RatioDCA} in \cite{TMH18}, we have the following modified schemes:
\begin{subequations}
\label{iter1-}
\begin{numcases}{}
\vec  x^{k+1}\in \argmin\limits_{\vec x\in \mathbb{B}} F_1(\vec x)+r^k G_2(\vec x) -(\langle \vec u^k,\vec x\rangle+r^k \langle \vec v^k,\vec x\rangle) + H_{\vec x^k}(\vec x)\text{ if }r^k\ge0, \label{eq:twostep_x2-}
\\
\vec  x^{k+1}\in \argmin\limits_{\vec x\in \mathbb{B}} G_1(\vec x)-\langle \vec w^k,\vec x\rangle-\frac{1}{r^k}( F_1(\vec x) - \langle \vec u^k,\vec x\rangle) + H_{\vec x^k}(\vec x)\text{ if }r^k<0, \label{eq:twostep_x22-}
\\
r^{k+1}=F( \vec x^{k+1})/G( \vec x^{k+1}),
\label{eq:twostep_r2-}
\\
 \vec u^{k+1}\in\nabla F_2( \vec x^{k+1}),\;
 \vec v^{k+1}\in\nabla G_1( \vec x^{k+1}),\;\vec w^{k+1}\in\nabla G_2( \vec x^{k+1})
\label{eq:twostep_s2-}
\end{numcases}
\end{subequations}
and
\begin{subequations}
\label{iter2-}
\begin{numcases}{}
\vec  y^{k+1}\in \argmin\limits_{\vec x\in X} F_1(\vec x)+r^k G_2(\vec x) -(\langle \vec u^k,\vec x\rangle+r^k \langle \vec v^k,\vec x\rangle) + H_{\vec x^k}(\vec x)\text{ if }r^k\ge0, \label{eq:2twostep_x2-}
\\
\vec  y^{k+1}\in \argmin\limits_{\vec x\in X} G_1(\vec x)-\langle \vec w^k,\vec x\rangle-\frac{1}{r^k}( F_1(\vec x) - \langle \vec u^k,\vec x\rangle) + H_{\vec x^k}(\vec x)\text{ if }r^k<0, \label{eq:2twostep_x22-}
\\
r^{k+1}=F( \vec y^{k+1})/G( \vec y^{k+1}),~~ \vec x^{k+1}=\partial \mathbb{B}\cap\{t\vec y^{k+1}:t\ge 0\} 
\label{eq:2twostep_r2-}
\\
 \vec u^{k+1}\in\nabla F_2( \vec x^{k+1}),\;
 \vec v^{k+1}\in\nabla G_1( \vec x^{k+1}),
\label{eq:2twostep_s2-}
\end{numcases}
\end{subequations}
in which the previous assumption $F_1-F_2\ge 0$ in \eqref{iter1} and \eqref{iter2}  has been removed. For these modifications, a convergence property like Theorem \ref{th:gsd} still holds.

\end{remark}

\begin{remark}
Theorem \ref{th:gsd} shows the local convergence of a general relaxation of  Dinkelbach's algorithm  in the spirit of DC programming. 
 The DC programming  consists in  minimizing $F-G$ where $F$ and $G$ are convex functions. 
As described in \cite{MaeharaMurota15,MMM18}, both the original DC algorithm and its discrete version can be written as the simple iteration: $\vec u^k\in \nabla G(\vec x^k)$, $\vec x^{k+1}\in \nabla F^\star(\vec u^k)$, where $F^\star$ is the Fenchel conjugate of $F$. It is known that such an iteration is equivalent to the following scheme  \begin{subequations}
\label{F-Giter1}
\begin{numcases}{}
\vec  x^{k+1}\in \argmin\limits_{\vec x} F(\vec x) -\langle \vec u^k,\vec x\rangle, \label{eq:F-Gtwostep_x2}
\\ \vec u^{k+1}\in\nabla G( \vec x^{k+1}).
\label{eq:F-Gtwostep_s2}
\end{numcases}
\end{subequations}
Moreover, 
a slight variation of the above scheme by adding a normalization step 
\begin{subequations}
\label{FGiter1}
\begin{numcases}{}
  \hat{\vec x}^{k+1}\in \argmin\limits_{\vec x\in\R^n} F(\vec x) -\langle \vec u^k,\vec x\rangle, \label{eq:FGtwostep_x2}
\\ \vec x^{k+1}=  \hat{\vec x}^{k+1}/G(\hat{\vec x}^{k+1})^{\frac1p}
\\ \vec u^{k+1}\in\nabla G( \vec x^{k+1}).
\label{eq:FGtwostep_s2}
\end{numcases}
\end{subequations}
can be used to solve the fractional programming $\min F/G$, where $F$ and $G$ are convex and $p$-homogeneous with $p>1$. This scheme is nothing but 
Algorithm 2 in \cite{HeinBuhler2010}. 
In fact, we can say more about it. 
\end{remark}
\begin{pro}
Let $F$ and $G$ be convex,   $p$-homogeneous and positive-definite functions on $\R^n$, where $p>1$.  Then, for any  initial point $\vec x^0$, the sequence of the pairs $\{(r^k,\vec x^k)\}_{k\ge1}$ produced by the following scheme
\begin{subequations}
\label{iter1FG}
\begin{numcases}{}
\hat{\vec x}^{k+1}\in \argmin\limits_{\vec x\in \R^n} F(\vec x) -a_k\langle \vec u^k,\vec x\rangle, \label{eq:twostep_x2FG}
\\
\vec x^{k+1}= b_{k+1}  
\hat{\vec x}^{k+1}\; (\mathrm{scaling}),\;\; r^{k+1}=F( \vec x^{k+1})/G( \vec x^{k+1}),
\label{eq:twostep_r2FG}
\\
 \vec u^{k+1}\in\nabla G( \vec x^{k+1}),
\label{eq:twostep_s2FG}
\end{numcases}
\end{subequations}
converges to an eigenpair $(r^*,\vec x^*)$ of $(F,G)$ in the sense that $\lim\limits_{k\to+\infty}r^k= r^*$ and $\vec x^*$ is a limit point of $\{\vec x^k\}_{k\ge1}$,  whenever $a_k,b_k>0$ as well as  both $\{a_k\}_{k\ge 1}$ and  $\{b_k\hat{\vec x}^k\}_{k\ge 1}$
 are  bounded away from $0$ and $\infty$.
\end{pro}

The proof is very similar to the original proof of Theorem 3.1 in  \cite{HeinBuhler2010},  with an additional trick  like the proof of Case 2.1 in Theorem \ref{th:gsd}. It can be regarded as a supplement of both Theorem 3.1 in  \cite{HeinBuhler2010} and Theorem \ref{th:gsd}. It is also interesting that the scheme is stable under   perturbations of $a_k$ and $b_k$. Besides, it can be seen that the resulting eigenvalue $r^*$ should be  independent of the choice of $a_k$ and $b_k$. In fact, $r^*$ only depends on the initial data and the choice of subgradient $\vec u^k$. 
The assumption that  $F$ is positive-definite can be removed in some sense. Indeed, if $r^k\le0$ for some $k$, we can modify \eqref{eq:twostep_x2FG} as $\hat{\vec x}^{k+1}\in \argmin\limits_{\vec x\in \R^n} F(\vec x) -r^kG(\vec x)$ or $\hat{\vec x}^{k+1}\in \argmin\limits_{\vec x\in \R^n} G(\vec x) -\frac{1}{r^k}F(\vec x)$ when $r^k<0$. Then $\{r^k\}$  converges to the global minimum of $F/G$. 

Now, we apply the above mixed 
IP-SD scheme to fractional  combinatorial optimization problems. 
By the results 
in Section \ref{sec:CC-transfer}, any combinatorial optimization in  ratio form can be translated to fractional programming of the form  \eqref{eq:RatioDC-form} via multi-way  Lov\'asz  extensions. 
Then, applying the mixed IP-SD  scheme to the resulting  optimization, we  get a solution of the equivalent  continuous optimization. And it is surprising that such a continuous solution  can produce a combinatorial solution of the original    problem directly, as precisely described in the following  proposition.

\begin{pro}\label{pro:Lovasz-algorithm}
Given $f:\A\to\R$ and $g:\A\to\R_+$, where $\A=\power(V)$ or $\power_2(V)$ or $\power^k(V)$ or $\power^k_2(V)$,   let $F=f^L$,   $G=g^L$, and  take  $\vec x^0=\vec1_A$ in the 
iteration scheme  \eqref{iter1} or \eqref{iter2} for some $A\in\A$. Suppose that 
$\vec x^*$ 
is a limit point of the iterative sequence $\{\vec x^k\}$ obtained by \eqref{iter1} or \eqref{iter2}. Then, 
any associated  set-tuple $A^*$ of $\vec x^*$ 
is an eigen-set 
of the corresponding combinatorial eigenvalue problem \eqref{eq:f1f2g1g2eigen},  and there holds 
$f(A^*)/g(A^*)\le f(A)/g(A)$.
\end{pro}

\begin{proof}
It suffices to consider the case that $F=f_1^L-f_2^L$ and $G=g_1^L-g_2^L$ are one-homogeneous in Theorem \ref{th:gsd}, where $f=f_1-f_2$ and $g=g_1-g_2$ are submodular decompositions.  Then, from any initial point $\vec x^0:=\vec1_A$,  either  \eqref{iter1} or \eqref{iter2} provides a solution $\vec x^*$ which must be an eigenvector of the nonlinear eigenvalue problem
$$\vec0\in \nabla f_1^L(\vec x^*)-\nabla f_2^L(\vec x^*)-r^*\left(\nabla g_1^L(\vec x^*)-\nabla g_2^L(\vec x^*)\right).$$
Similar to  the proof of  Proposition \ref{pro:Lovasz-eigen-pre}, or simply using  Proposition \ref{pro:general-eigen-equiv}, 
for any associate set-tuple $A^*$ of $\vec x^*$, the indicator vector $\vec 1_{A^*}$ also satisfies 
\begin{align*}
\vec0&\in \nabla f_1^L(\vec 1_{A^*})-\nabla f_2^L(\vec 1_{A^*})-r^*\left(\nabla g_1^L(\vec 1_{A^*})-\nabla g_2^L(\vec 1_{A^*})\right)    
\end{align*}
which can be rewritten in the form of the combinatorial
eigenvalue problem 
\begin{equation}\label{eq:f1f2g1g2eigen}
\vec0\in\nabla f_1(A^*)-\nabla f_2(A^*)-r^*\left(\nabla g_1(A^*)-\nabla g_2(A^*)\right).
\end{equation}
Moreover, $$\frac{f(A^*)}{g(A^*)}=\frac{f^L(\vec 1_{A^*})}{g^L(\vec 1_{A^*})}=\frac{f^L(\vec x^*)}{g^L(\vec x^*)}\le \frac{f^L(\vec x^0)}{g^L(\vec x^0)}=\frac{f^L(\vec 1_{A})}{g^L(\vec 1_{A})}=\frac{f(A)}{g(A)}$$
where the inequality is due to Theorem \ref{th:gsd}.
\end{proof}


\noindent\textbf{Advantages of the  mixed IP-SD  algorithm}. 
By Proposition \ref{pro:Lovasz-algorithm},  
the advantage  of the mixed IP-SD scheme over existing continuous algorithms  for solving combinatorial optimization  in fractional form 
 is that it provides  an iterative solution  {\sl without}  rounding, and can be used  to improve initially given data.  %
In fact, it should be noted that  these two advantages, namely, an iterative solution  without rounding, and usage to improve initially given data, do not apply to  other continuous algorithms, like 
semi-definite relaxations \cite{GW95,Goemans97} and its variants \cite{BMZ01}, spectral cut method \cite{DP93,PR95} and its recursive implementations \cite{Trevisan2012}, as well as polynomial programming \cite{ST92}.

A special version of the previous     mixed IP-SD algorithm 
has been  actually used in some classic graph cut problems \cite{HeinBuhler2010,CSZ15,CSZ16,SZZmaxcut}. 
Although we have not yet systematically investigated   the solution quality and the computational complexity  in general, some good numerical simulations have  been reported  for certain problems,  including the Cheeger cut  \cite{BuhlerHein2009,HeinBuhler2010,CSZ15}, the dual Cheeger problem \cite{CSZ16}, and the maxcut problem \cite{SZZmaxcut}. 
In particular, in Section \ref{sec:maxcut}, we will    discuss the maxcut problem in detail to illustrate the performance, solution quality   and numerical simulations.  
Successful numerical experiments have shown that the mixed IP-SD iterative algorithm is likely to be efficient, and converges in  polynomial time. We propose to investigate  the computation time or  convergence rates required to obtain the  solution in future work. 

Furthermore, the mixed IP-SD scheme proposed  in this section can be generalized slightly  to compute the second eigenvalue of the function pair obtained by the Lov\'asz extension.  We refer the reader to Section 2.1 
in  \cite{JostZhang} for a more general  description.  

Another solver for the continuous optimization $\min\frac{F(\vec x)}{G(\vec x)}$ is 
 the stochastic subgradient method: 
\begin{equation*}
\vec x^{k+1}=\vec x^k-\alpha_k(\vec y^k+\vec\xi^k),\;\;\;\vec y^k\in\nabla\frac{F(\vec x^k)}{G(\vec x^k)},
\end{equation*}
where $\{\alpha_k\}_{k\ge1}$ is a step-size sequence and $\{\vec\xi^k\}_{k\ge1}$ is now a sequence of random variables (the ``noise'') on some probability space. Theorem 4.2 in \cite{Davis19-FoCM} shows that under some natural assumptions, almost surely, every limit point of the stochastic subgradient iterates $\{\vec x^k\}_{k\ge1}$ is critical for $F/G$, and the function values $\{\frac{F}{G}(\vec x^k)\}_{k\ge1}$ converge. 
Of course, many other  continuous optimization algorithms can be applied, and the mixed IP-SD  scheme is just one suitable option. It is expected that better algorithms can be designed based on the obtained equivalent continuous optimization problem via our multi-way Lov\'asz extensions.

\section{Examples and Applications}
\label{sec:examples-Applications}

\subsection{Submodular vertex cover and multiway partition problems}
As a first immediate application of Theorem \ref{thm:tilde-fg-equal}, we obtain an easy way to rediscover the famous identity by Lov\'asz, and the two typical submodular optimizations -- submodular vertex cover and multiway partition problems.
\begin{example}
The identity $\min\limits_{A\in\power(V)}f(A)=\min\limits_{\vec x\in[0,1]^V}f^L(\vec x)$ discovered by Lov\'asz in his original paper \cite{Lovasz} can be obtained by our result. In fact, $$\min\limits_{A\in\power(V)}f(A)=\min\limits_{A\in\power(V)}\frac{f(A)}{1}=\min\limits_{\vec x\in [0,\infty)^V}\frac{f^L(\vec x)}{\max\limits_{i\in V}x_i}=\min\limits_{\vec x\in [0,1]^V}\frac{f^L(\vec x)}{\max\limits_{i\in V}x_i}=\min\limits_{\vec x\in [0,1]^V,\max\limits_i x_i=1}f^L(\vec x).$$
Checking this is easy:
 if $f\ge 0$, then $\min\limits_{\vec x\in [0,1]^V,\max\limits_i x_i=1}f^L(\vec x)=0$; if $f(A)<0$ for some $A\subset V$, then $\min\limits_{\vec x\in [0,1]^V,\max\limits_i x_i=1}f^L(\vec x)=\min\limits_{\vec x\in [0,1]^V}f^L(\vec x)$.
\end{example}

\paragraph{Vertex cover number }
A vertex cover (or node cover) of a graph is a set of vertices such that each edge of the graph is incident to at least one vertex of the set. The vertex cover number is the minimal cardinality of a  vertex cover. Similarly, the independence number of a graph is the maximal number of vertices not connected by edges. The sum of the vertex cover number and the independence number is  the  cardinality of the vertex set. 

By a variation of the Motzkin-Straus theorem and Theorem \ref{thm:graph-numbers}, the vertex cover number thus has at least two equivalent continuous representations similar to the independence number.

\paragraph{Submodular vertex cover problem}
Given a graph $G=(V,E)$, and a submodular function $f:\power(V)\to[0,\infty)$,  find a vertex cover $S\subset V$ minimizing $f(S)$.

By Theorem \ref{thm:tilde-fg-equal},
$$\min\{f(S):S\subset V,\,S\text{ is a vertex cover}\}=\min\limits_{\vec x\in\D}\frac{f^L(\vec x)}{\|\vec x\|_\infty}=\min\limits_{\vec x\in \widetilde{\D}}f^L(\vec x)$$
where $\D=\{\vec x\in[0,\infty)^V:V^t(\vec x)\text{ vertex cover},\,\forall t\ge0\}=\{\vec x\in[0,\infty)^V:x_i+x_j>0,\forall\{i,j\}\in E,\,\{i:x_i=\max_j x_j\}\text{ vertex cover}\}$,  and $\widetilde{\D}=\{\vec x\in\D: \|\vec x\|_\infty=1\}=\{\vec x\ge\vec 0:x_i+x_j\ge 1,\forall\{i,j\}\in E,\,\{i:x_i=\max_j x_j\}\text{ vertex cover}\}$. Note that
$$\mathrm{conv}(\widetilde{\D})=\{\vec x:x_i+x_j\ge 1,\forall\{i,j\}\in E,\,x_i\ge 0,\forall i\in V\}.$$

Therefore, $\min\limits_{\vec x\in \mathrm{conv}(\widetilde{\D})}f^L(\vec x)\le \min\{f(S):\text{ vertex cover }S\subset V\}$, which rediscovers the convex programming relaxation.

\paragraph{Submodular multiway partition problem}
This problem  is about to minimize $\sum_{i=1}^k f(V_i)$ subject to $V=V_1\cup\cdots\cup V_k$, $V_i\cap V_j=\varnothing$, $i\ne j$, $v_i\in V_i$, $i=1,\cdots,k$, where $f:\power(V)\to\R$ is a submodular function.

 Letting $\A=\{\text{ partition }(A_1,\cdots,A_k)\text{ of }V:A_i\ni a_i,\,i=1,\cdots,k\}$, by Theorem \ref{thm:tilde-fg-equal}, 
 $$\min\limits_{(A_1,\cdots,A_k)\in\A}\sum_{i=1}^k f(A_i)=\inf\limits_{\vec x\in \D_\A}\frac{\sum_{i=1}^k f^L(\vec x^i)}{\|\vec x\|_\infty}=\inf\limits_{\vec x\in \D'}\sum_{i=1}^k f^L(\vec x^i),$$
  where $\D_\A=\{\vec x\in [0,\infty)^{kn}: (V^t(\vec x^1),\cdots,V^t(\vec x^k))\text{ is a partition}, V^t(\vec x^i)\ni a_i,\forall t\ge 0\}=\{\vec x\in [0,\infty)^{kn}: \vec x^i=t1_{A_i},A_i\ni a_i,\forall t\ge 0\}$, and $\D'=\{(\vec x^1,\cdots,\vec x^k):\vec x^i\in [0,\infty)^V,\,\vec x^i=\vec 1_{A_i},A_i\ni a_i\}$. Note that $$\mathrm{conv}(\D')=\{(\vec x^1,\cdots,\vec x^k):\sum_{v\in V} x^i_v=1,x^i_{a_i}=1,x^i_v\ge0\}.$$ So one rediscovers the corresponding convex programming relaxation $\min\limits_{\vec x\in \mathrm{conv}(\D')}\sum_{i=1}^k f^L(\vec x^i)$.
  
\subsection{Min-cut and Max-cut}
\label{sec:maxcut}
Given an undirected weighted   graph $(V,E,w)$, the  min-cut problem  
$$\min\limits_{S\ne\varnothing,V}|\partial S|:=\min\limits_{S\ne\varnothing,V}|E(S,V\setminus S)|=\min\limits_{S\ne\varnothing,V}\sum\limits_{i\in S,j\in V\setminus S}w_{ij}$$
 and the 
max-cut problem  $$\max\limits_{S\ne\varnothing,V}|\partial S|:=\max\limits_{S\ne\varnothing,V}|E(S,V\setminus S)|=\max\limits_{S\ne\varnothing,V}\sum\limits_{i\in S,j\in V\setminus S}w_{ij}$$
have been  investigated systematically.

\begin{theorem}\label{thm:mincut-maxcut-eigen}Let $(V,E,w)$ be a weighted  undirected graph. 
Then, we have the equivalent continuous optimization formulations for the min-cut and max-cut problems:
$$\min\limits_{S\ne\varnothing,V}|\partial S|=\min\limits_{\min _ix_i+\max_i x_i=0}\frac{\sum_{ij\in E}w_{ij}|x_i-x_j|}{2\|\vec x\|_\infty}=\tilde{\lambda}_2,$$
$$\max\limits_{S\ne\varnothing,V}|\partial S|= \max\limits_{\vec x\ne \vec0}\frac{\sum_{ij\in E}w_{ij}|x_i-x_j|}{2\|\vec x\|_\infty}=\tilde{\lambda}_{\max},$$
where $\tilde{\lambda}_2$ and $\tilde{\lambda}_{\max}$ are the second 
(i.e., the smallest  nontrivial)  eigenvalue and the largest eigenvalue of the nonlinear eigenvalue problem:
\begin{equation}\label{eq:mincut-maxcut-eigen}
\vec0\in \nabla\sum_{ij\in E}w_{ij}|x_i-x_j|-\lambda\nabla 2\|\vec x\|_\infty.    
\end{equation}\end{theorem}

\begin{proof}
We only prove the min-cut case. It is clear that
$$ \min\limits_{S\ne\varnothing,V}|\partial S|= \min\limits_{A,B\ne\varnothing,A\cap B=\varnothing}\frac{ |\partial A|+|\partial B|}{2}$$
Let $\A=\{(A,B)\in \power_2(V):A,B\ne\varnothing\}$. Then $\D_\A=\{\vec x\in\R^n: \max_i x_i=-\min_i x_i>0\}$, and by Theorem  \ref{thm:tilde-fg-equal}, $$\min\limits_{S\ne\varnothing,V}|\partial S|=\min\limits_{(A,B)\in\A}\frac{ |\partial A|+|\partial B|}{2}=\min\limits_{\vec x\in\D_\A}\frac{\sum_{ij\in E}w_{ij}|x_i-x_j|}{2\|\vec x\|_\infty}. $$ 
 In addition,  according to Theorem \ref{introthm:eigenvalue}, the set of the eigenvalues of $(f^L,g^L)$ coincides with 
 $$\left\{\frac{f^L(\vec 1_A-\vec 1_{V\setminus A})}{g^L(\vec 1_A-\vec 1_{V\setminus A})}:A\subset V\right\}=\left\{\frac{f(A,V\setminus A)}{g(A,V\setminus A)}:A\subset V\right\}=\left\{\frac{|\partial A|+|\partial (V\setminus A)|}{2}:A\subset V\right\}=\{|\partial A|:A\subset V\},$$
 where $f(A,B)=|\partial A|+|\partial B|$ and $g(A,B)=2$. In consequence, $\min\limits_{S\ne\varnothing,V}|\partial S|$ is the second 
eigenvalue  of $(f^L,g^L)$. The proof is completed.
\end{proof}

Eq.~\eqref{eq:mincut-maxcut-eigen} shows the first nonlinear eigenvalue problem which possesses two nontrivial eigenvalues that are equivalent to two important graph optimization problems, respectively. 

In addition, by our results, we  present a lot of equivalent continuous optimizations for the maxcut problem (see Examples \ref{exam:maxcut1} and \ref{exam:maxcut2}):
\begin{align}
\max\limits_{S\subset V}|\partial S|&=\max\limits_{x\ne
  0}\frac{ \sum_{\{i,j\}\in E}w_{ij}(|x_i|+|x_j|-|x_i+x_j|)^p }{(2\|\vec
  x\|_\infty)^p}\notag
\\&=\max\limits_{x\ne
  0}\frac{ \sum_{\{i,j\}\in E}w_{ij}|x_i-x_j|^p }{(2\|\vec
  x\|_\infty)^p}=\max\limits_{\|\vec
  x\|_\infty\le\frac12 }  \sum_{\{i,j\}\in E}w_{ij}|x_i-x_j|^p  \label{eq:maxcut-p}
\end{align}
for any $p\ge1$.  We shall now show three applications of the above formulation. 

\noindent\textbf{CirCut
algorithm (by Burer et al  \cite{BMZ01})}.\; From the  equality \eqref{eq:maxcut-p}, it is easy to see   $$\max\limits_{S\subset V}|\partial S|=\frac{1}{2^p}\max\limits_{\theta\in\R^n}\sum_{\{i,j\}\in E}w_{ij}|\cos\theta_i-\cos\theta_j|^p,$$
and taking $p=2$, we have the equivalent formulation of the maxcut problem
\begin{equation}\label{eq:maxcut-theta}
\max\limits_{S\subset V}|\partial S|=\frac14\max\limits_{\theta\in\R^n}\sum_{\{i,j\}\in E}w_{ij}(\cos\theta_i-\cos\theta_j)^2=\max\limits_{\theta\in\R^n}\sum_{\{i,j\}\in E}w_{ij}\sin^2\frac{\theta_i-\theta_j}{2}\sin^2\frac{\theta_i+\theta_j}{2}.
\end{equation}
If we remove the term $\sin^2\frac{\theta_i+\theta_j}{2}$ on the right-hand-side of \eqref{eq:maxcut-theta}, 
that is, consider instead   the continuous relaxation 
$$\max\limits_{\theta\in\R^n}\sum_{\{i,j\}\in E}w_{ij}\sin^2\frac{\theta_i-\theta_j}{2}=
\frac12\sum_{\{i,j\}\in E}w_{ij}-\frac12\min\limits_{\theta\in\R^n}\sum_{\{i,j\}\in E}w_{ij}\cos(\theta_i-\theta_j)
\Longleftrightarrow \min\limits_{\theta\in\R^n}\sum_{\{i,j\}\in E}w_{ij}\cos(\theta_i-\theta_j)$$
we 
immediately recover the CirCut algorithm proposed by Burer, Monteiro and Zhang  \cite{BMZ01}, which is a smart  relaxation of the  maxcut problem. Until now, 
it is still one of the best  algorithms for solving maxcut in terms of numerical experiments. Burer et al  \cite{BMZ01} consider their method as a rank-two  relaxation of the Goemans-Williamson algorithm, where the latter  is a semi-definite  relaxation of the  maxcut problem.  
Thus, our new formulation \eqref{eq:maxcut-theta} indeed provides an   alternative  perspective to  the  Burer-Monteiro-Zhang's CirCut algorithm. 



\vspace{0.2cm}

\noindent\textbf{A simple iterative 
algorithm}.\; Based on the case of taking  $p=1$ in \eqref{eq:maxcut-p}, there is a previous work on computing the maxcut problem by the second author and his  collaborators  
\cite{SZZmaxcut}, in which the mixed IP-SD algorithm is  essentially  used. 
Specifically, the simple iterative algorithm in \cite{SZZmaxcut} is indeed an  implementation  of the mixed IP-SD scheme in Section \ref{sec:algo}  by taking $H=0$, $F(\vec x)= F_1(\vec x)=\sum_{\{i,j\}\in E}w_{ij}|x_i-x_j|$ and $G(\vec x)=G_1(\vec x)=\|\vec x\|_\infty$ in the iterative scheme \eqref{iter1}.  

We will  briefly report  the performance of the  mixed IP-SD algorithm applied to the maxcut problem,   which is presented in  \cite{SZZmaxcut}.  
As discussed in Section \ref{sec:algo}, our  algorithm is \emph{completely rounding-free}, whereas almost all other  algorithms require  additional   explicit or  implicit rounding operations;  for example,  the Procedure-CUT operation in the CirCut  algorithm of Burer et al \cite{BMZ01} can be seen as an implicit  rounding technique. 
More importantly, the iterative values obtained by our mixed IP-SD algorithm are monotonic to the equivalent continuous objective function  of the maxcut problem and can be used for  \emph{post-processing} to improve the quality of the solution obtained by any other algorithms. 
In particular, we would like to point out that our algorithm \emph{does  improve} the cuts obtained by the CirCut algorithm, while  conversely  the CirCut  algorithm \emph{cannot}  improve the quality of the solutions  produced by our mixed IP-SD  algorithm (see Section 4.4 in \cite{SZZmaxcut} for a detailed  comparison and illustration). 
 These numerical experiments in \cite{SZZmaxcut} show that the mixed IP-SD  algorithm is  efficient, and  converges in  polynomial time, and is one of the best continuous iterative algorithms for the maxcut problem.

\vspace{0.2cm}

\noindent\textbf{A new geometric perspective  for the Goemans-Williamson 
algorithm}.\; In addition, our  equivalent continuous reformulation of the maxcut problem also provides  a new geometric perspective  for Goemans-Williamson's SDP  algorithm, via the following relations:
\begin{align*}
\max\limits_{\|\vec
  x\|_\infty\le 1 }  \sum_{\{i,j\}\in E}w_{ij}|x_i-x_j|^2 &=\frac1n\max\limits_{\substack{\vec
  x^i\in\R^n\\ \|\vec
  x^i\|_\infty\le 1} }  \sum_{\{i,j\}\in E}w_{ij}\|\vec x^i-\vec x^j\|_2^2 = \max\limits_{\|\vec
  x^i\|_\infty\le 1/\sqrt{n} }  \sum_{\{i,j\}\in E}w_{ij}\|\vec x^i-\vec x^j\|_2^2
  \\&\le \max\limits_{\|\vec
  x^i\|_2\le 1 }  \sum_{\{i,j\}\in E}w_{ij}\|\vec x^i-\vec x^j\|_2^2 =\max\limits_{\|\vec
  x^i\|_2=1 }  \sum_{\{i,j\}\in E}w_{ij}\|\vec x^i-\vec x^j\|_2^2
 \\&=2\sum_{\{i,j\}\in E}w_{ij}-2\min\limits_{\|\vec
  x^i\|_2=1 }\langle \vec x^i,\vec x^j\rangle
\end{align*}
  where the inequality is based on the fact that $\|\vec
  x^i\|_\infty\le 1/\sqrt{n}$ implies $\|\vec
  x^i\|_2\le 1$, and the second-to-last  equality is  due to the convexity of the relaxed objective function.  
According to the above inequality, we can see that the famous  Goemans-Williamson  algorithm for the maxcut problem  actually relaxes the constraint domain from the $l^\infty$-ball (i.e., a hypercube) to its circumscribed sphere.

\subsection{Max $k$-cut problem}
\label{sec:max-k-cut}
The max $k$-cut problem is to determine a graph $k$-cut by solving
\begin{equation}\label{eq:maxk}
\mathrm{MaxC}_k(G)=\max_{\text{partition }(A_1,A_2,\ldots,A_k)\text{ of }V}\sum_{i\ne j}|E(A_i,A_j)|=\max_{(A_1,A_2,\ldots,A_k)\in \mathcal{C}_{k}(V)}\sum_{i=1}^k|\partial A_i|,
\end{equation}
where $\mathcal{C}_{k}(V)=\{(A_1,\ldots,A_k)\big|A_i\cap A_j = \varnothing, \bigcup_{i=1}^{k} A_i= V \}$,  and $\partial A_i:=E(A_i,V\setminus A_i)$. We may write \eqref{eq:maxk} as
$$ \mathrm{MaxC}_k(G)=\max_{(A_1,A_2,\ldots,A_{k-1})\in \mathcal{P}_{k-1}(V)}\sum_{i=1}^{k-1}|\partial A_i|+|\partial (A_1\cup\cdots\cup A_{k-1})|.$$
Taking $f_k(A_1,\cdots,A_k)=\sum_{i=1}^{k}|\partial A_i|+|\partial (A_1\cup\cdots\cup A_{k})|$, the $k$-way Lov\'asz extension is $$f^L_k(\vec x^1,\cdots,\vec x^k)=\sum_{l=1}^k\sum_{i\sim j}|x^l_i-x^l_j|+\sum_{j\sim j'}\left|\max\limits_{i=1,\cdots,k} x^i_j-\max\limits_{i=1,\cdots,k} x^i_{j'}\right|.$$
Applying  Theorem \ref{thm:tilde-fg-equal}, we have
$$ \mathrm{MaxC}_{k+1}(G)=\max\limits_{\vec x^i\in\R^n_{\ge0}\setminus\{\vec0\},\,\supp(\vec x^i)\cap \supp(\vec x^j)=\varnothing}\frac{\sum_{l=1}^k\sum_{i\sim j}|x^l_i-x^l_j|+\sum_{j\sim j'}\left|\max\limits_{i=1,\cdots,k} x^i_j-\max\limits_{i=1,\cdots,k} x^i_{j'}\right|}{\max\limits_{i,j}x^i_j}$$
Also, it is clear that 
$$\mathrm{MaxC}_k(G)=\max_{(A_1,A_2,\ldots,A_k)\in\power_k(V)}\sum_{i\ne j}|E(A_i,A_j)|=\max_{(A_1,A_2,\ldots,A_k)\in \power_{k}(V)}\sum_{i=1}^k|\partial A_i|,
$$
and by employing Theorem \ref{thm:tilde-fg-equal}, 
the max $k$-cut constant  $\mathrm{MaxC}_{k}(G)$ has the following equivalent continuous reformulations:
$$ \max\limits_{\vec x^i\in\R^n_{\ge0}\setminus\{\vec0\},\,\supp(\vec x^i)\cap \supp(\vec x^j)=\varnothing}\frac{\sum_{l=1}^k\sum_{i\sim j}|x^l_i-x^l_j|}{\max\limits_{l,j}x^l_j}=\max\limits_{\vec x^i\in\R^n_{\ge0}\setminus\{\vec0\},\,\supp(\vec x^i)\cap \supp(\vec x^j)=\varnothing}\sum_{l=1}^k\frac{\sum_{i\sim j}|x^l_i-x^l_j|}{\max\limits_{j}x^l_j}
$$
$$ 
=\max\limits_{\vec x^i\in\R^n\setminus\{\vec0\},\supp(\vec x^i)\cap \supp(\vec x^j)=\varnothing}\frac{\sum_{l=1}^k\sum_{i\sim j}|x^l_i-x^l_j|}{2\max_l\|\vec x^l\|_\infty}=\max\limits_{\vec x^i\in\R^n\setminus\{\vec0\},\supp(\vec x^i)\cap \supp(\vec x^j)=\varnothing}\sum_{l=1}^k\frac{\sum_{i\sim j}|x^l_i-x^l_j|}{2\|\vec x^l\|_\infty}
$$
\subsection{Relative isoperimetric constants on a subgraph with boundary}
\label{sec:boundary-graph-1-lap}
Given a finite graph $G=(V,E)$ and a subgraph, we 
consider the Dirichlet and Neumann eigenvalue problems for the corresponding 1-Laplacian. For $A\subset V$, put $\overline{A}=A\cup \delta A$, where $\delta A$ is the set of points in  $A^c$ that are adjacent to some points in $A$ (see Fig.~\ref{fig:AdeltaA}).

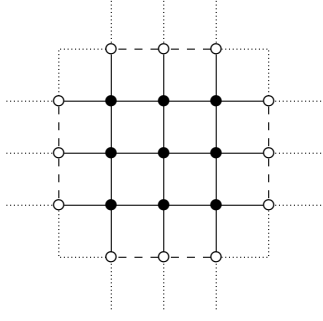
\begin{figure}[!h]\centering
\begin{tikzpicture}[scale=0.69]
\draw (0,0) to (1,0);
\draw (0,0) to (0,1);
\draw (1,0) to (2,0);
\draw (1,0) to (1,1);
\draw (1,1) to (0,1);
\draw (0,2) to (0,1);
\draw (0,2) to (1,2);
\draw (2,0) to (2,1);
\draw (2,2) to (1,2);
\draw (2,2) to (2,1);
\draw (1,1) to (1,2);
\draw (1,1) to (2,1);
\draw (0,0) to (-0.9,0);
\draw (0,0) to (0,-0.9);
\draw (1,0) to (1,-0.9);
\draw (2,0) to (2,-0.9);
\draw (0,2) to (-0.9,2);
\draw (0,1) to (-0.9,1);
\draw[densely dotted] (-2,2) to (-1.1,2);
\draw[densely dotted] (-2,1) to (-1.1,1);
\draw[densely dotted] (-2,0) to (-1.1,0);
\draw[densely dotted] (4,2) to (3.1,2);
\draw[densely dotted] (4,1) to (3.1,1);
\draw[densely dotted] (4,0) to (3.1,0);
\draw[densely dotted] (2,-2) to (2,-1.1);
\draw[densely dotted] (1,-2) to (1,-1.1);
\draw[densely dotted] (0,-2) to (0,-1.1);
\draw[densely dotted] (2,4) to (2,3.1);
\draw[densely dotted] (1,4) to (1,3.1);
\draw[densely dotted] (0,4) to (0,3.1);
\draw[densely dotted] (-0.1,3) to (-1,3);
\draw[densely dotted] (-0.1,-1) to (-1,-1);
\draw[densely dotted] (2.1,3) to (3,3);
\draw[densely dotted] (2.1,-1) to (3,-1);
\draw[densely dotted] (3,-0.1) to (3,-1);
\draw[densely dotted] (-1,-0.1) to (-1,-1);
\draw[densely dotted] (3,2.1) to (3,3);
\draw[densely dotted] (-1,2.1) to (-1,3);
\draw[dashed] (-1,1.9) to (-1,1.1);
\draw[dashed] (-1,0.9) to (-1,0.1);
\draw[dashed] (3,1.9) to (3,1.1);
\draw[dashed] (3,0.9) to (3,0.1);
\draw[dashed] (1.9,-1) to (1.1,-1);
\draw[dashed] (0.9,-1) to (0.1,-1);
\draw[dashed] (1.9,3) to (1.1,3);
\draw[dashed] (0.9,3) to (0.1,3);
\node (00) at (0,0) {$\bullet$};
\node (10) at (1,0) {$\bullet$};
\node (11) at (1,1) {$\bullet$};
\node (01) at (0,1) {$\bullet$};
\node (02) at (0,2) {$\bullet$};
\node (20) at (2,0) {$\bullet$};
\node (12) at (1,2) {$\bullet$};
\node (21) at (2,1) {$\bullet$};
\node (22) at (2,2) {$\bullet$};
\node (03) at (0,3) {$\circ$};
\node (30) at (3,0) {$\circ$};
\node (13) at (1,3) {$\circ$};
\node (31) at (3,1) {$\circ$};
\node (23) at (2,3) {$\circ$};
\node (32) at (3,2) {$\circ$};
\node (01') at (0,-1) {$\circ$};
\node (1'0) at (-1,0) {$\circ$};
\node (11') at (1,-1) {$\circ$};
\node (1'1) at (-1,1) {$\circ$};
\node (21') at (2,-1) {$\circ$};
\node (1'2) at (-1,2) {$\circ$};
\draw (2.9,0) to (2,0);
\draw (0,2.9) to (0,2);
\draw (2.9,1) to (2,1);
\draw (1,2.9) to (1,2);
\draw (2.9,2) to (2,2);
\draw (2,2.9) to (2,2);
\end{tikzpicture}
\caption{\label{fig:AdeltaA} In this graph, let  $A$ be the set of solid
  points, $\delta A$ the set of hollow points. We only consider the edges
   for which one vertex is in $A$ and the other in $\overline{A}$ (solid lines). We will ignore the dashed lines in $\delta A$, and the dotted lines outside $\overline{A}$.}
\end{figure}

Given $S\subset \overline{A}$, denote the boundary of $S$ relative to $A$ by 
$$\partial_A S=\{(u,v)\in E:u\in S\cap A,v\in \delta A\setminus S\text{ or }u\in S, v\in A\setminus S\}.$$
If $S\subset A$, then $\partial_A S=\{(u,v)\in E:u\in S, v\in \overline{A}\setminus S\}$.~

The Cheeger (cut) constant of the subgraph $A$ of $G$ is defined as
$$h(A)=\min_{S\subset \overline{A}}\frac{|\partial_A S|}{\min\{\vol(A\cap S),\vol(A\setminus S)\}}.$$
A set pair $(S,\overline{A}\setminus S)$ that achieves the Cheeger constant is called a Cheeger cut.

The Cheeger isoperimetric constant\footnote{Some authors call it the Dirichlet isoperimetric constant.} of $A$ is defined as
$$h_1(A)=\min_{S\subset  A}\frac{|\partial_A S|}{\vol(S)},$$
where a set $S$ achieving the Cheeger isoperimetric constant is called a Cheeger set. In the sequel, we fix  $A\subset V$,  and we write $h(G)$ and $h_1(G)$ instead of $h(A)$ and $h_1(A)$, respectively.

According to our generalized Lov\'asz extension, we have
\begin{equation}\label{eq:Dirichlet-Cheeger-h_1}
h_1(G)
=\inf_{\vec x\in \mathbb{R}^n\setminus\{0\},\,\mathrm{supp}(\vec x)\subset A}\frac{\sum_{i\sim j} |x_i-x_j|+\sum_{i\in A} p_i|x_i|}{\sum_{i\in A}d_i|x_i|}
\end{equation}
and
$$h(G)
=\inf_{\vec x\in \mathbb{R}^n\setminus\{0\}}\frac{\sum_{i\sim j,i,j\in A}|x_i-x_j|+\sum_{i\sim j,i\in A,j\in\delta A}|x_i-x_j|}{\inf_{c\in \mathbb{R}}\sum_{i\in A}d_i|x_i-c|}.$$

Note that the term on the right hand side of \eqref{eq:Dirichlet-Cheeger-h_1} can be written as
$$
\inf\limits_{ \vec x|_{V\setminus A}=0,\, \vec x\ne 0}\mathcal{R}_1(x)
$$
which is called the {\sl Dirichlet $1$-Poincare constant} (see \cite{OSY19}) over $S$,
where $$\mathcal{R}_1(\vec x):=\frac{\sum\limits_{\{i,j\}\in E} |x_i-x_j|}{\sum_i d_i|x_i|}$$
is called the $1$-Rayleigh quotient of $\vec x$.

We can consider the corresponding spectral problems.
\begin{itemize}
\item Dirichlet eigenvalue problem:
$$\begin{cases}
\Delta_1 \vec x\cap \mu D \Sgn \vec x\ne\varnothing,& \text{ in } A\\
\vec x = 0,&\text{ on }  \delta A
\end{cases}$$  where $D$ is the diagonal matrix of the vertex degrees,  
that is,
$$\begin{cases}
(\Delta_1 \vec x)_i-\mu d_i\Sgn x_i\ni 0,&i\in A\\
x_i=0,&i\in\delta A
\end{cases}$$~
whose component form is: $\exists$~$c_i\in \Sgn(x_i)$,~$z_{ij}\in \Sgn(x_i-x_j)$ satisfying $z_{ji}=-z_{ij}$ and
$$\sum_{j\sim i} z_{ij}+ p_ic_i\in \mu d_i\Sgn(x_i),~i\in A,$$
in which $p_i$ is the number of neighbors of $i$ in $\delta A$.
\item Neumann eigenvalue problem: There exists $c_i\in \Sgn(x_i)$,~$z_{ij}\in \Sgn(x_i-x_j)$ with $z_{ji}=-z_{ij}$ such that
$$
\begin{cases}
\sum_{j\sim i,j\in \overline{A}}z_{ij}-\mu d_i c_i=0,&i\in A\\
\sum_{j\sim i,j\in A}z_{ij}=0,&i\in\delta A.
\end{cases}
$$
\end{itemize}
For a graph $G$ with boundary, we use $\Delta_1^D(G)$ and $\Delta_1^N(G)$ to denote the Dirichlet 1-Laplacian and the Neumann 1-Laplacian, respectively. Then
\begin{pro}
$$h_1(G)=\lambda_1(\Delta_1^D(G))\;\text{ and }\;h(G)=\lambda_2(\Delta_1^N(G)).$$
\end{pro}
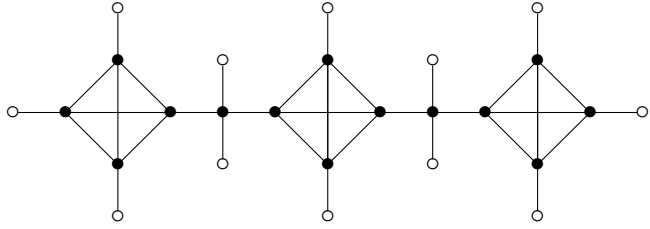
\begin{figure}\centering
\begin{tikzpicture}[scale=0.69]
\draw (0,0) to (2,0);\draw (0,0) to (-0.9,0);
\draw (0,0) to (1,1);
\draw (0,0) to (1,-1);
\draw (1,1) to (1,-1);\draw (1,1) to (1,1.9);
\draw (1,1) to (2,0);
\draw (1,-1) to (2,0);\draw (1,-1) to (1,-1.9);
\draw (3,0) to (2,0);\draw (3,0) to (3,0.9);
\draw (3,0) to (4,0);\draw (3,0) to (3,-0.9);
\draw (6,0) to (4,0);
\draw (5,1) to (4,0);
\draw (5,-1) to (4,0);
\draw (5,1) to (6,0);
\draw (5,-1) to (6,0);
\draw (5,-1) to (5,1);
\draw (6,0) to (7,0);
\draw (8,0) to (7,0);
\draw (8,0) to (10,0);
\draw (8,0) to (9,1);
\draw (8,0) to (9,-1);
\draw (10,0) to (9,1);
\draw (10,0) to (9,-1);
\draw (9,1) to (9,-1);
\draw (10,0) to (10.9,0);
\draw (5,1) to (5,1.9);
\draw (5,1) to (5,-1.9);
\draw (9,1) to (9,1.9);
\draw (9,1) to (9,-1.9);
\draw (7,0) to (7,0.9);
\draw (7,0) to (7,-0.9);
\node (00) at (0,0) {$\bullet$};
\node (11) at (1,1) {$\bullet$};
\node (1-1) at (1,-1) {$\bullet$};
\node (20) at (2,0) {$\bullet$};
\node (30) at (3,0) {$\bullet$};
\node (40) at (4,0) {$\bullet$};
\node (60) at (6,0) {$\bullet$};
\node (51) at (5,1) {$\bullet$};
\node (5-1) at (5,-1) {$\bullet$};
\node (70) at (7,0) {$\bullet$};
\node (80) at (8,0) {$\bullet$};
\node (100) at (10,0) {$\bullet$};
\node (91) at (9,1) {$\bullet$};
\node (9-1) at (9,-1) {$\bullet$};
\node (-10) at (-1,0) {$\circ$};
\node (110) at (11,0) {$\circ$};
\node (12) at (1,2) {$\circ$};
\node (1-2) at (1,-2) {$\circ$};
\node (31) at (3,1) {$\circ$};
\node (31) at (3,-1) {$\circ$};
\node (71) at (7,1) {$\circ$};
\node (7-1) at (7,-1) {$\circ$};
\node (52) at (5,2) {$\circ$};
\node (5-2) at (5,-2) {$\circ$};
\node (92) at (9,2) {$\circ$};
\node (9-2) at (9,-2) {$\circ$};
\end{tikzpicture}
\caption{\label{fig:k-nodal-domain} In this example, there are $3$ nodal domains of an eigenvector corresponding to the first Dirichlet eigenvalue of the graph 1-Laplacian. Each nodal domain is the vertex set of the $4$-order complete subgraph shown in the figure. }
\end{figure}

For a connected graph, the first eigenvector of $\Delta_1^N(G)$ is constant 
and it has only one nodal domain while the first eigenvector of $\Delta_1^D(G)$ 
 may have any number of nodal domains. 
 In fact, we have:
 \begin{pro}
 For any $k\in \mathbb{N}^+$, there exists a connected graph $G$ with boundary such that its Dirichlet 1-Laplacian $\Delta_1^D(G)$ has an  eigenvector  corresponding to $\lambda_1(\Delta_1^D(G))$  with exactly $k$ nodal domains; and there exists a connected graph $G'$ with boundary such that its Neumann 1-Laplacian $\Delta_1^N(G')$ possesses an  eigenvector  corresponding to $\lambda_2(\Delta_1^N(G'))$  with exactly $k$ nodal domains.
 \end{pro}

 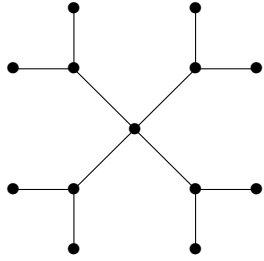
\begin{figure}\centering
 \begin{tikzpicture}[scale=0.8]
\node (1) at (0,1) {$\bullet$};
\node (2) at (1,0) {$\bullet$};
\node (3) at (1,1) {$\bullet$};
\node (4) at (2,2) {$\bullet$};
\node (5) at (3,1) {$\bullet$};
\node (6) at (4,1) {$\bullet$};
\node (7) at (3,0) {$\bullet$};
\node (8) at (1,3) {$\bullet$};
\node (9) at (1,4) {$\bullet$};
\node (10) at (0,3) {$\bullet$};
\node (11) at (3,3) {$\bullet$};
\node (12) at (3,4) {$\bullet$};
\node (13) at (4,3) {$\bullet$};
\draw (0,1) to (1,1);
\draw (1,0) to (1,1);
\draw (1,1) to (2,2);
\draw (2,2) to (3,1);
\draw (3,1) to (4,1);
\draw (3,1) to (3,0);
\draw (2,2) to (1,3);
\draw (1,3) to (1,4);
\draw (1,3) to (0,3);
\draw (2,2) to (3,3);
\draw (3,3) to (3,4);
\draw (3,3) to (4,3);
\end{tikzpicture}
\caption{\label{fig:second-eigen-k-nodal-domain} In this example, there are $4$ nodal domains of an eigenvector corresponding to the second Neumann eigenvalue of the graph 1-Laplacian. Each nodal domain is the vertex set of the $3$-order subgraph after removing the center vertex and its edges.}
\end{figure}

 We provide a final comment on the computational aspects for  Cheeger constants and  1-Laplacians on graphs. The work \cite{BuhlerHein2009} shows  that the  spectral clustering based on the graph $p$-Laplacian for $p\to1$ generally has a superior performance compared to the  standard linear spectral clustering. 
In their subsequent work \cite{HeinBuhler2010}, the authors also developed an improved method based on the eigenvectors of the graph 1-Laplacian, which can be computed using their nonlinear inverse power method.  This method runs faster and produces better cuts, and in fact, this process achieved  state-of-the-art results of its time in terms of solution quality and runtime  \cite{HeinBuhler2010}. 
Their nonlinear inverse power algorithms \cite{BuhlerHein2009,HeinBuhler2010} have  been  subsumed into our mixed IP-SD scheme in Section \ref{sec:algo}.

\subsection{Independence number}\label{sec:independent-number}
The independence number $\alpha(G)$ of an unweighted and undirected simple graph $G$ is  the largest cardinality of a subset of vertices in $G$, no two of which are adjacent. It can be seen as an optimization problem $\max\limits_{S\subset V \text{ s.t. }E(S)=\varnothing}\#S$. However, such a graph optimization is not global, and the feasible domain seems to be very complicated. 
But we may simply  multiply by a truncated term $(1-\#E(S))$. 
The independence number can then be expressed as a global optimization on the power set of vertices:
\begin{equation}\label{eq:independent-multiple}
\alpha(G)=\max\limits_{S\subset V}\#S(1-\#E(S)),
\end{equation}
and thus the Lov\'asz extension can be applied.
\begin{proof}[Proof of Eq.~\eqref{eq:independent-multiple}] Since $G$ is simple, $\#S$ and $\#E(S)$ take values in the natural numbers. Therefore,
$$
\#S(1-\#E(S))\;\; \begin{cases}\le 0,&\text{ if } E(S)\ne\varnothing\text{ or }S=\varnothing,\\
\ge 1,&\text{ if } E(S)=\varnothing\text{ and }S\ne\varnothing.\end{cases}
$$
Thus, $\max\limits_{S\subset V}\#S(1-\#E(S))=\max\limits_{S\subset V \text{ s.t. }E(S)=\varnothing}\#S=\alpha(G)$.
\end{proof}
However, Eq.~\eqref{eq:independent-multiple} is  difficult to calculate. By the disjoint-pair Lov\'asz extension, it equals 
$$
\alpha(G)=\max\limits_{\vec x\ne \vec 0}\frac{\|\vec x\|_1-\sum\limits_{k\in V,i\sim j}\min\{|x_k|,|x_i|,|x_j|\}}{\|\vec x\|_\infty},
$$
but we don't know how to further simplify it.

Fortunately, there is a known  representation of the independence number as follows, and  we present a proof for convenience. 
\begin{pro}The independence number $\alpha(G)$ of a finite simple graph $G=(V,E)$ satisfies
\begin{equation}\label{eq:independent-difference}
\alpha(G)=\max\limits_{S\subset V}\left(\#S-\#E(S)\right).
\end{equation}
\end{pro}

\begin{proof}
Let $A$ be an independent set of $G$, then $\alpha(G)=\#A=\#A-\#E(A)\le \max\limits_{S\subset V}\left(\#S-\#E(S)\right)$ because there is no edge connecting points in $A$.

Let $B\subset V$ satisfy $ \#B-\#E(B) =\max\limits_{S\subset V}\left(\#S-\#E(S)\right)$. Assume the induced subgraph $(B,E(B))$ has $k$ connected components, $(B_i,E(B_i))$, $i=1,\cdots,k$. Then $B=\sqcup_{i=1}^k B_i$ and $E(B)=\sqcup_{i=1}^k E(B_i)$. Since $(B_i,E(B_i))$ is connected, $\# B_i\le \# E(B_i)+1$ and  equality holds if and only if $(B_i,E(B_i))$ is a tree.
Now taking $B'\subset B$ such that $\#(B'\cap B_i)=1$, $i=1,\cdots,k$, then $B'$ is an independent set and thus  \begin{align*}
\alpha(G)&\ge \#B'=k=\sum_{i=1}^k 1\ge \sum_{i=1}^k (\# B_i- \# E(B_i))
= \sum_{i=1}^k\# B_i- \sum_{i=1}^k\# E(B_i)\\&=\#(\cup_{i=1}^kB_i)-\#(\cup_{i=1}^kE(B_i)) = \#B-\#E(B)=\max\limits_{S\subset V}\left(\#S-\#E(S)\right).
\end{align*}
As a result, Eq.~\eqref{eq:independent-difference} is proved.
\end{proof}

According to Lov\'asz extension, we get
\begin{equation}\label{eq:independent-continuous}
\alpha(G)=\max\limits_{\vec x\ne \vec 0}\frac{\|\vec x\|_1-\sum\limits_{i\sim j}\min\{|x_i|,|x_j|\}}{\|\vec x\|_\infty}.
\end{equation}
By the elementary identities: $\sum_{i\sim j}|x_i+x_j|+\sum_{i\sim j}|x_i-x_j|=2\sum_{i\sim j}\max\{|x_i|,|x_j|\}=\sum_{i\sim j}\left||x_i|-|x_j|\right|+\sum_i\mathrm{deg}_i|x_i|$ and $\sum_i\mathrm{deg}_i|x_i|=\sum_{i\sim j}\max\{|x_i|,|x_j|\}+\sum_{i\sim j}\min\{|x_i|,|x_j|\}$,  Eq.~\eqref{eq:independent-continuous} can be reduced to
\begin{equation}\label{eq:independent-continuous-1}
\alpha(G)=\max\limits_{\vec x\ne \vec 0}\frac{2\|\vec x\|_1+I^-(\vec x)+I^+(\vec x)- 2\|\vec x\|_{1,\mathrm{deg}}}{2\|\vec x\|_\infty},
\end{equation}
where $I^\pm(\vec x)=\sum_{i\sim j}|x_i\pm x_j|$ and $\|\vec x\|_{1,\mathrm{deg}}=\sum_i\mathrm{deg}_i|x_i|$. One would like to write Eq.~\eqref{eq:independent-continuous-1} as
\begin{equation}\label{eq:independent-continuous-2}
\alpha(G)=\max\limits_{\vec x\ne \vec 0}\frac{I^-(\vec x)+I^+(\vec x)- 2\|\vec x\|_{1,\mathrm{deg}'}}{2\|\vec x\|_\infty},
\end{equation}
where $\|\vec x\|_{1,\mathrm{deg}'}=\sum\limits_{i\in V}(\mathrm{deg}_i-1)|x_i|$.

\begin{remark}
The maximum clique number can be reformulated in a similar way. 
In addition, we refer to \cite{BBcontinuous05,BBcontinuous06,SBBB20} for some other continuous formulations of the independence number. 
\end{remark}

\paragraph{Chromatic number of a  perfect graph}
Berge's strong perfect graph conjecture has been proved in \cite{Annals06}.  A graph $G$ is perfect if for every induced subgraph $H$ of $G$, the chromatic number of $H$ equals the size of the largest clique of $H$. The complement of every perfect graph is perfect. 


So for a  perfect graph, we have an easy way to calculate the  chromatic number. In a general simple graph, we refer to Section \ref{sec:chromatic-number} for transforming the chromatic number.

\paragraph{Maximum matching }
A matching $M$ in $G$ is a set of pairwise non-adjacent edges, none of which are loops; that is, no two edges share a common vertex.
A maximal matching  is one with the largest possible number of edges.

Consider the line graph $(E,R)$ whose vertex set $E$ is the edge set of $G$, and whose edge set is $R=\{\{e,e'\}:e\cap e'\not=\varnothing,\,e,e'\in E\}$. Then the maximum matching number of $(V,E)$ coincides with the independence number of $(E,R)$. So, we have an equivalent continuous optimization for a maximum matching problem.

Hall's Marriage Theorem  provides a characterization of bipartite graphs which have a perfect matching and the Tutte theorem provides a characterization for arbitrary graphs.

The Tutte-Berge formula says that the size of a maximum matching of a graph is
$$\frac12\min\limits_{U\subset V}\left(\#V+\#U-\#\text{ odd connected components of }G|_{V\setminus U}\right).$$
Can one  transform the above discrete optimization problem into an explicit continuous optimization via some extension?

\paragraph{$k$-independence number }
The independence number admits several generalizations:
 the maximum size of a set of vertices in a graph whose induced subgraph has maximum degree $(k-1)$ \cite{CaroH13}; the size of the largest $k$-colourable subgraph \cite{Spacapan11}; the
size of the largest set of vertices such that any two vertices in the set are at short-path distance larger than $k$ (see \cite{Fiol97}). For the $k$-independence number involving short-path distance, one can easily transform it into the following two continuous representations:
$$\alpha_k= \max\limits_{\vec x\in\R^V\setminus\{\vec 0\}}\frac{\|\vec x\|_1^2}{\|\vec x\|_1^2-2\sum\limits_{\dist(i,j)\ge k+1}x_ix_j} = \max\limits_{\vec x\in \R^n\setminus\{\vec 0\}}\frac{\sum\limits_{\dist(i,j)\le k}(|x_i-x_j|+|x_i+x_j|)- 2\sum\limits_{i\in V}(\deg_{k}(i)-1)|x_i|}{2\|\vec x\|_\infty},$$
where $\deg_{k}(i)=\#\{j\in V:\dist(j,i)\le k\}$, $i=1,\cdots,n$. 
\subsection{Various and variant Cheeger problems}
\label{sec:variantCheeger}
In \cite{HS11}, the equality relating the  Cheeger constant on graphs and the second eigenvalue of the graph 1-Laplacian was reproved via Lov\'asz extension.  
Moreover, an equality relating the  dual Cheeger constant on graphs and the first eigenvalue of the signless  1-Laplacian
has been obtained by the second author and his coauthors via the disjoint-pair Lov\'asz extension \cite{CSZ16,CSZ18}. As the reported results on both  analytical properties and numerical experiments  are very satisfactory, 
we believe that the multi-way Lov\'asz extension in our general framework 
should be useful to obtain more results  on other types of discrete Cheeger constants, from which the mixed IP-SD iterative algorithm is expected to be efficient.  In this section,   several  Cheeger-type  constants on graphs have been proposed that are  different from  the classical one. 
And based on our general multi-way Lov\'asz extension framework, we establish some equivalent continuous representations of these Cheeger-type  constants. 

\paragraph{Multiplicative Cheeger constant}
For instance
$$h=\min\limits_{\varnothing\ne A\subsetneqq V}\frac{ \#E(A,V\setminus A) }{\#A\cdot\#(V\setminus A)}.$$
It is called the normalized cut problem which has many applications in image segmentation and spectral clustering \cite{SM00,Luxburg07,HS11}.  By Proposition \ref{pro:fraction-f/g}, it is equal to
$$\min\limits_{\langle\vec x,\vec1\rangle=0,\vec x\ne \vec0}\frac{\sum_{i\sim j}|x_i-x_j|}{\sum_{i< j}|x_i-x_j|}.$$

\paragraph{(Weighted) sparsest cut problem}
Given non-negative weights $w_{ij}$ and $\mu_{ij}$ for  $i,j\in V$, the weighted sparsest cut problem is to solve  $$\min\limits_{\varnothing\ne A\subsetneqq V}\frac{ \sum_{i\in A,j\in V\setminus A}w_{ij}}{\sum_{i\in A,j\in V\setminus A}\mu_{ij}}$$ which is related to some famous  open problems in 
theoretical computer science such as the Unique Games
Conjecture 
\cite{ARV09,Goemans97,Linial02}. By Proposition \ref{pro:fraction-f/g}, the sparsest cut problem  is equivalent to solve 
$$\min\limits_{\vec x:\text{ denominator nonzero}
}\frac{\sum_{i,j\in V}w_{ij}|x_i-x_j|}{\sum_{i,j\in V}\mu_{ij}|x_i-x_j|}=\min\limits_{\vec y^1,\cdots,\vec y^n:\text{ denominator nonzero}
}\frac{\sum_{i,j\in V}w_{ij}\|\vec y^i-\vec y^j\|_1}{\sum_{i,j\in V}\mu_{ij}\|\vec y^i-\vec y^j\|_1}
,$$
which provides a  direct  way to get the $l^1$-metric tight relaxation, 
and if we replace the $l^1$-norm by the squared $l^2$-norm with the  additional constraint $\|\vec y^i-\vec y^j\|_2^2\le \|\vec y^i-\vec y^k\|_2^2+\|\vec y^k-\vec y^j\|_2^2$ for all $i,j,k$,  we immediately get the  relaxed sparsest cut problem.

\paragraph{Isoperimetric profile}
The isoperimetric profile $IP:\mathbb{N}\to [0,\infty)$ is defined by
$$IP(k)= \inf\limits_{A\subset V,\#A\le k} \frac{\#E(A,V\setminus A)}{\#A}.$$
Then by Lov\'asz extension, it is equal to
$$\inf\limits_{\vec x\in\R^V,\,1\le \# \supp(\vec x)\le k}\frac{\sum_{\{i,j\}\in E}|x_i-x_j|}{\|\vec x\|_1}=\min\limits_{\vec x\in CH_k(\R^V)}\frac{\sum_{\{i,j\}\in E}|x_i-x_j|}{\|\vec x\|_1},$$
where $CH_n:=\{\vec x\in\R^V,\,\# \supp(\vec x)\le k\}$ is the union of all $k$-dimensional coordinate hyperplanes in $\R^V$. 
%
%

\paragraph{Modified Cheeger constant}

On a graph $G=(V,E)$, there are three definitions of the vertex-boundary of a subset $A\subset V$:
\begin{align}
& \partial_{\textrm{ext}} A:=\{j\in V\setminus A\,\left|\,\{j,i\}\in E\text{ for some }i\in A\right.\} \label{eq:ext-vertex-boundary}\\
& \partial_{\textrm{int}} A:=\{i\in A\,\left|\,\{i,j\}\in E\text{ for some }j\in V\setminus A\right.\}\label{eq:int-vertex-boundary}\\
& \partial_{\textrm{ver}} A:=\partial_{\textrm{out}} A\cup \partial_{\textrm{int}} A=V(E(A,V\setminus A))=V(\partial_{\textrm{edge}} A)\label{eq:vertex-boundary}
\end{align}
The  {\sl external vertex boundary} \eqref{eq:ext-vertex-boundary} and the  {\sl internal vertex boundary} \eqref{eq:int-vertex-boundary} are introduced and studied recently in \cite{Vigolo19tams,VigoloPHD}.  Research on metric measure space \cite{HMT19} suggests to consider the {\sl vertex boundary} \eqref{eq:vertex-boundary}. 

Denote by $N(i)=\{i\}\cup\{j\in V:\{i,j\}\in E\}$ the 1-neighborhood of $i$. Then the Lov\'asz extensions of $\#\partial_{\textrm{ext}} A$, $\#\partial_{\textrm{int}} A$ and $\#\partial_{\textrm{ver}} A$ are
$$\sum\limits_{i=1}^n(\max\limits_{j\in N(i)}x_j-x_i),\;\;\;\sum\limits_{i=1}^n(x_i-\min\limits_{j\in N(i)}x_j)\;\;\text{ and }\;\;\sum\limits_{i=1}^n(\max\limits_{j\in N(i)}x_j-\min\limits_{j\in N(i)}x_j),$$ respectively. 
They can be seen as the  `total variation' of $\vec x$ with respect to $V$ in $G$, while the usual {\sl edge boundary} leads to $\sum\limits_{\{i,j\}\in E}|x_i-x_j|$ which is  regarded as the total variation of $\vec x$ with respect to $E$ in $G$.  Their disjoint-pair Lov\'asz extensions are $$\sum_{i=1}^n \max\limits_{j\in N(i)} |x_j|-\|\vec x\|_1,\;\;\;\|\vec x\|_1-\sum_{i=1}^n \min\limits_{j\in N(i)} |x_j|,\;\;\;\sum_{i=1}^n \left(\max\limits_{j\in N(i)} |x_j|-\min\limits_{j\in N(i)} |x_j|\right).$$
Comparing with the graph $1$-Poincare profile (see \cite{Hume17,HMT19,Hume19arxiv}) $$P^1(G):=\inf\limits_{\langle\vec x,\vec 1\rangle=0,\vec x\ne\vec 0}\frac{\sum_{i\in V} \max\limits_{j\sim i}|x_i-x_j|}{\|\vec x\|_1},$$
we easily get the following
\begin{pro}
$$ \frac12\max\{h_{\mathrm{int}}(G),h_{\mathrm{ext}}(G)\}\le P^1(G)\le h_{\mathrm{ver}}(G):=\min\limits_{A\in\power(V)\setminus\{\varnothing,V\}}
\frac{\#\partial_{\mathrm{ver}} A}{\min\{\#(A),\#(V\setminus A)\}}$$
where $h_{\mathrm{int}}(G)$,  $h_{\mathrm{ext}}(G)$ and $h_{\mathrm{ver}}(G)$ are modified Cheeger constants w.r.t. the type of vertex-boundary.
\end{pro}
\begin{proof}
By Theorem \ref{introthm:eigenvalue}, $$h_{\mathrm{ver}}(G)=\min\limits_{A\in\power(V)\setminus\{\varnothing,V\}}
\frac{\#\partial_{\mathrm{ver}} A}{\min\{\#(A),\#(V\setminus A)\}}=\inf\limits_{\langle\vec x,\vec 1\rangle=0,\vec x\ne\vec 0}\frac{\sum_{i\in V} \max\limits_{j\sim i}|x_i-x_j|}{\min\limits_{t\in\R}\|\vec x-t\vec1\|_1}\ge P^1(G).$$
On the other hand, it is easy to check that $\min\limits_{t\in\R}\|\vec x-t\vec1\|_1\ge \frac12 \|\vec x\|_1$ whenever $\langle\vec x,\vec 1\rangle=0$. Thus, $h_{\mathrm{ver}}(G)\le 2P^1(G)$. The proof is then completed by noting that $\max\{h_{\mathrm{int}}(G),h_{\mathrm{ext}}(G)\}\le h_{\mathrm{ver}}(G)$.
\end{proof}

\begin{remark}
We remark here that the numerator term $\sum_{i\in V} \max_{j\sim i}|x_i-x_j|$ in general is neither the Lov\'asz extension of any discrete function $f:\power(V)\to\R$ nor the disjoint-pair  Lov\'asz extension of any discrete function $f:\power_2(V)\to\R$.
\end{remark}

\paragraph{Cheeger-like constant}
Some further recent results  \cite{JM19arxiv} can be also rediscovered via  Lov\'asz extension.

A main equality in \cite{JM19arxiv} can be absorbed into the following identities:
\begin{align}
	      \max_{\text{edges }(v,w)}\biggl(\frac{1}{\deg v}+\frac{1}{\deg w}\biggr)
&=
\max_{\gamma:E\rightarrow\mathbb{R}}\frac{\sum_{v\in V}\frac{1}{\deg v}\cdot \biggl|\sum_{e_{\text{in}}: v\text{ input}}\gamma(e_{\text{in}})-\sum_{e_{\text{out}}: v\text{ output}}\gamma(e_{\text{out}})\biggr|}{\sum_{e\in E}|\gamma(e)|} \notag
\\&=\max_{\hat{\Gamma}\subset\Gamma \text{ bipartite}} \frac{\sum_{v\in V}\frac{\deg_{\hat{\Gamma}}(v)}{\deg_\Gamma (v)}}{|E(\hat{\Gamma})|}, \label{eq:mainJM19}
\end{align}
where the left quantity is called a Cheeger-like constant \cite{JM19arxiv}.

In fact, given $c_i\ge 0$, $i\in V$,
$$\max\limits_{\{i,j\}\in E}(c_i+c_j)=\max\limits_{E'\subset E}\frac{\sum_{\{i,j\}\in E'}(c_i+c_j)}{\# E'},$$
and then via Lov\'asz extension, one immediately gets that the above constant equals to
$$
\max\limits_{\vec x\in [0,\infty)^{E}\setminus\{\vec0\}} \frac{\sum\limits_{e=\{i,j\}\in E}x_e(c_i+c_j)}{\sum_{e\in E}x_e}=\max\limits_{\vec x\in[0,\infty)^{E}\setminus\{\vec0\}} \frac{\sum_{i\in V}c_i\sum_{e\ni i}x_e}{\sum_{e\in E}x_e}=\max\limits_{\vec x\in \R^{E}\setminus\{\vec0\}} \frac{\sum_{i\in V}c_i\left|\sum_{e\ni i}x_e\right|}{\sum_{e\in E}|x_e|}.
$$
Thus, for any family $\E\subset\power(E)$ such that $E'\in \E$ $\Rightarrow$ $E'\supset \{\{e\}:e\in E\}$, we have
$$\max\limits_{\{i,j\}\in E}(c_i+c_j)=\max\limits_{\vec x\in \R^{E}\setminus\{\vec0\}} \frac{\sum_{i\in V}c_i\left|\sum_{e\ni i}x_e\right|}{\sum_{e\in E}|x_e|}=\max\limits_{E'\in\E}\frac{\sum_{\{i,j\}\in E'}(c_i+c_j)}{\# E'},$$
which recovers the interesting equality \eqref{eq:mainJM19}  by taking $c_i=\frac{1}{\deg i}$ and $\E$ the collections of all edge sets of bipartite subgraphs.

A similar simple trick gives
\begin{equation*}
    \min_{(v,w)}\frac{\bigl|\mathcal{N}(v)\cap \mathcal{N}(w)\bigr|}{\max\{\deg v,\deg w\}}=\min\limits_{\vec x\in \R^{E}\setminus\{\vec0\}} \frac{\sum_{i\in V}\sum_{e\ni i}\left|x_e\right|\cdot\#\text{ triangles containing }e}{\sum_{e=\{i,j\}\in E}|x_e|\max\{\deg i,\deg j\}}.
\end{equation*}

 By our general spectral theory 
 for discrete structures 
\cite{JostZhang}, we immediately obtain the $k$-way Cheeger inequality and the $k$-way dual  Cheeger inequality involving  the graph 1-Laplacian \cite{CSZ17,JostZhang}. For more results 
on various types of Cheeger constants on hypergraphs,  we refer the reader to  \cite{JostZhang}.  

\subsection{Frustration in signed networks}
\label{sec:frustration}
In this section, we apply our theory to signed graphs, a concept first introduced by Harary \cite{Harary55}.
\begin{defn}
  A \emph{signed graph} $\Gamma$ consists of a vertex set $V$ and a set $E$ of undirected edges with a sign function
\begin{equation}\label{sign1}
s:E \to \{+1,-1\}.
\end{equation}
The adjacency matrix of $(\Gamma,s)$, is denoted by  $\mathrm{A}^s:=(s_{ij})_{i,j\in V}$, where $s_{ij}:=s(e)$ if $e=\{i,j\}\in E$, and  $s_{ij}:=0$ otherwise.
\end{defn}
When we replace the sign function $s$ by $-s$, we shall call the resulting graph \emph{antisigned}.
\begin{defn}
 The signed cycle $C_m$ (consisting of $m$ vertices that are cyclically connected by $m$ edges) is \emph{balanced} if 
\begin{equation}\label{sign3}
\prod_{i=1}^m s(e_i)=1.
\end{equation}
A  signed graph $(\Gamma,s)$ is \emph{balanced}  if every cycle contained in it is balanced.\\
$(\Gamma,s)$ is \emph{antibalanced}  if $(\Gamma,-s)$ is balanced. \\
The frustration index of  a signed graph $\Gamma=(V,E)$ is  
\begin{equation}\label{eq:frustration}
\min_{x_i\in \{-1,1\},\forall i} \sum_{\{i,j\}\in E}|x_i-s_{ij}x_j|,
\end{equation}  where $s_{ij}\in\{-1,1\}$  indicates the sign of the edge $(i,j)$.
\end{defn}
The frustration index then vanishes iff the graph is balanced.\\
\begin{defn}
 The  (normalized) Laplacian $\Delta^s$ of a signed graph is defined by
\begin{equation}
\label{sign5}
(\Delta^s \vec x)_i:= x_i-\frac{1}{\deg i}\sum_{j \sim i}s_{ij}x_j=\frac{1}{\deg i}\sum_{j \sim i}(x_i-s_{ij}x_j)
\end{equation}
for a vector $\vec x\in\R^V$. 
\end{defn}
\begin{remark}
The Laplacian thus is of the form $\Delta^s =\mathrm{id}-\mathrm{A}^s$, and 
 when we change the signs of all the edges, that is, go from a signed graph to the corresponding antisigned graph, the operator becomes $\Delta^{-s} =\mathrm{id}+\mathrm{A}^s$. Therefore, the eigenvalues simply  change from $\lambda$ to $2-\lambda$ (and therefore, also the ordering gets reversed). 
\end{remark}

By Proposition \ref{pro:Lovasz-eigen}, it is easy to verify that every eigenvalue of the function pair $(F,G)$ has an  eigenvector in  $\{-1,0,1\}^n$,  where $F(\vec x)=\sum_{\{i,j\}\in E}|x_i-s_{ij}x_j|$ and $G(\vec x)=\|\vec x\|_\infty$. One may  relax \eqref{eq:frustration} as  \begin{equation}\label{eq:frustration-relax}
\min_{\vec x\in \{-1,0,1\}^n\setminus\{\vec0\}} \sum_{(i,j)\in E}|x_i-s_{ij}x_j|.    
\end{equation}

This suggests the  eigenvalue problem of $(F(\vec x),\|\vec x\|_\infty)$ on a signed graph, 
where $F(\vec x)=\sum_{\{i,j\}\in E}|x_i-s_{ij}x_j|$. Below,  we show some key properties.





\begin{itemize}
\item 
The coordinate form of the eigenvalue problem $\nabla \sum_{\{i,j\}\in E}|x_i-s_{ij}x_j| \cap \lambda\nabla\|\vec x\|_\infty\ne\varnothing$ reads as

 $\exists\,z_{ij}\in \Sgn(x_i-s_{ij}x_j) \mbox{ with }z_{ij}+s_{ij}z_{ji}=0$~such that
\begin{align}
\label{eq:1LinftyN1}&\sum\limits_{j\sim i} z_{ij} = 0, &i\in D_0(\vec x),\\
\label{eq:1LinftyN2}&\sum\limits_{j\sim i} z_{ij} \in \lambda\ \mathrm{sign}(x_i)\cdot [0,1],& i\in D_\pm(\vec x),\\
\label{eq:1LinftyN3}&\sum\limits_i^n \big|\sum\limits_{j\sim i} z_{ij} \big|=\lambda,&
\end{align}
where $D_\pm(\vec x) =\{i\in V\big| \pm x_i = \|\vec x\|\}$, 
and $D_0(\vec x)=\{i\in V\big| |x_i|<\|\vec x\|\}$.
\item All eigenvalues are integers in  $\{0,1,\cdots,\vol(V)\}$. And each eigenvalue has an eigenvector in $\{-1,0,1\}^n$.

Proof: This is a direct consequence of Proposition \ref{pro:Lovasz-eigen}. 

\item The largest eigenvalue has an eigenvector in $\{-1,1\}^n$.

Proof: Let $\vec 1_A-\vec 1_B$ be an eigenvector  w.r.t. the largest eigenvalue. Note that $\vec 1_A-\vec 1_B=\frac12(\vec1_A-\vec1_{V\setminus A}+\vec1_{V\setminus B}-\vec1_B)$. By the convexity of $F$, we have $F(\vec 1_A-\vec 1_B)\le\max\{F(\vec1_A-\vec1_{V\setminus A}),F(\vec1_{V\setminus B}-\vec1_B)\}$. Hence, either $\vec1_A-\vec1_{V\setminus A}$ or $\vec1_{V\setminus B}-\vec1_B$ is an eigenvector w.r.t. the largest eigenvalue. 

\item \textbf{The frustration index  is an eigenvalue}. 
However, in general, we don't know which eigenvalue the  frustration  index is.

 Proof: We shall check that for any $A\subset V$, the binary vector $\vec x:=\vec1_A-\vec1_{V\setminus A}$ is an eigenvector w.r.t. the eigenvalue $\lambda:=2(|E_+(A,V\setminus A)|+|E_-(A)|+|E_-(V\setminus A)|)$, where  $|E_+(A,V\setminus A)|$ indicates the number of positive edges lying between  $A$ and $V\setminus A$, while  $|E_-(A)|$ denotes the number of negative edges lying in $A$.  Indeed, $D_+(\vec x)=A$ and $D_-(\vec x)=V\setminus A$.   For $i\in A$, taking $z_{ij}=1$ if $s_{ij}x_j<0$; and $z_{ij}=0$ if $s_{ij}x_j>0$.  Similarly, for $i\in V\setminus A$, letting $z_{ij}=0$ if $s_{ij}x_j<0$; and $z_{ij}=-1$ if $s_{ij}x_j>0$. It is easy to see that $z_{ij}\in \mathrm{Sgn}(x_i-s_{ij}x_j)$ and $z_{ij}+s_{ij}z_{ji}=0$ for any edge $ij$. Next, we verify the conditions \eqref{eq:1LinftyN2} and \eqref{eq:1LinftyN3}.

Note that  $\sum_{j\sim i}z_{ij}=\#(\{j\in A:ij\text{ is negative}\}\cup \{j\in V\setminus A:ij\text{ is positive}\})\in[0,\lambda]$ for $i\in A$, and  $\sum_{j\sim i}z_{ij}=-\#(\{j\in A:ij\text{ is positive}\}\cup \{j\in V\setminus A:ij\text{ is negative}\})\in[-\lambda,0]$ for $i\in V\setminus A$. 
Therefore, $\sum_{i\in V}|\sum_{j\sim i}z_{ij}|=2(|E_+(A,V\setminus A)|+|E_-(A)|+|E_-(V\setminus A)|)=\lambda$.

In particular, for $\vec x\in\{-1,1\}^n$ that realizes the frustration index, $\vec x$ must be an eigenvector, and the frustration index is the corresponding eigenvalue. This fact can also be derived by  Proposition  \ref{pro:set-pair-infty-norm}.  

\item We can use the the   Dinkelbach-type scheme in Section \ref{sec:algo}  directly to  calculate the smallest eigenvalue. When we get an eigenvector $\vec x$, we can take $\vec 1_{D_+(\vec x)}-\vec 1_{D_-(\vec x)}$ instead of $\vec x$. 
\item We construct a recursive method to approximate the frustration index:
\begin{itemize}
\item Input a signed graph $G$,  and use the   Dinkelbach-type algorithm to get a subpartition $(U_{+},U_{-})$ where $U_+=D_+(\vec x)$ and $U_-=D_-(\vec x)$ with $\vec x$ being an eigenvector w.r.t. the smallest eigenvalue. 
\item Let $G$ be the signed graph induced by $V\setminus(U_+\cup U_-)$, and let $(U_+',U_-')$ be the subpartition found by the  Dinkelbach-type  algorithmm; return $(U_+\cup U_+',U_-\cup U_-')$ or $(U_+\cup U_-',U_-\cup U_+')$, whichever is better.
\item Repeat the above process, until we get a partition $(V_+,V_-)$ of $V$, which derives an  approximate solution of the  frustration index. There are at most $n$ iterations.
\end{itemize}
In other words, the relaxation problem  \eqref{eq:frustration-relax} can approximate the frustration index \eqref{eq:frustration} in a recursive way.  This is inspired by the recursive spectral cut  algorithm for the   maxcut problem proposed by Trevisan  \cite{Trevisan2012}.
\end{itemize}

Next, we show some equivalent continuous representations of the  frustration index. Let $E_+$ (resp. $E_-$) collect all the positive (resp. negative) edges of $(V,E)$.
  Note that up to a scale factor,  \eqref{eq:frustration} is equivalent to solve $\min\limits_{A\subset V}|E_+(A,V\setminus A)|+|E_-(A)|+|E_-(V\setminus A)|$, where $|E_+(A,V\setminus A)|$ denotes the number of  positive edges between $A$ and $V\setminus A$, while $|E_-(A)|$ indicates the number of negative edges in $A$. By Lov\'asz extension,  the frustration index  is equivalent to
$$|E_-|+\min\limits_{\vec x\ne0}\frac{\sum_{\{i,j\}\in E_+}|x_i-x_j|+\sum_{i\in V}\deg_i|x_i|-\sum_{\{i,j\}\in E_-}(|x_i-x_j|+|x_i+x_j|)}{\|x\|_\infty}.$$ Also, \eqref{eq:frustration} is equivalent to $|E_-|+\min\limits_{A\subset V}(|E_+(A,V\setminus A)|-|E_-(A,V\setminus A)|)$, and by Lov\'asz extension, the frustration index equals 
$$|E_-|+\min\limits_{\vec x\ne0}\frac{\sum_{\{i,j\}\in E_+}|x_i-x_j|-\sum_{\{i,j\}\in E_-}|x_i-x_j|}{2\|x\|_\infty}.$$
One can then apply  the   Dinkelbach-type scheme in Section \ref{sec:algo}   straightforwardly  to  compute the  frustration index.

 


\begin{remark}
We should point out that the notion $|E_+(A)|$ (resp. $|E_-(A)|$) indicates the number of positive (resp. negative)  edges (unordered pairs) whose vertices are in $A$. Therefore, in   our paper, the values of  $|E_+(A)|$ and  $|E_-(A)|$ are  
half of those of 
Atay-Liu \cite{AtayLiu}, in which they count the  ordered pairs. 
\end{remark}
Besides, by Theorem \ref{thm:tilde-fg-equal-PQ} (or Theorem \ref{thm:tilde-H-f}), we can derive another continuous formulation of the frustration index:
$$\min\limits_{A\subset V}|E_+(A,V\setminus A)|+|E_-(A)|+|E_-(V\setminus A)|=\min\limits_{x\ne0}\frac{\sum\limits_{\{i,j\}\in E_+}|x_i-x_j|^\alpha+\sum\limits_{\{i,j\}\in E_-}(2\|\vec x\|_\infty-|x_i-x_j|)^\alpha}{(2\|x\|_\infty)^\alpha}$$
whenever $0<\alpha\le 1$. It is interesting that by taking $\alpha\to 0^+$, we immediately get 
$$\min\limits_{A\subset V}|E_+(A,V\setminus A)|+|E_-(A)|+|E_-(V\setminus A)|=\min\limits_{x\ne0}\sum\limits_{\{i,j\}\in E_+}\mathrm{sign}(|x_i-x_j|)+\sum\limits_{\{i,j\}\in E_-}\mathrm{sign}(2\|\vec x\|_\infty-|x_i-x_j|).$$

\subsection{Modularity measure}\label{sec:modularity-measure}

For a weighted graph $(V,(w_{ij})_{i,j\in V})$, the   modularity measure \cite{TMH18} is defined as
$$Q(A)=\sum_{i,j\in A}w_{ij}-\frac{\vol(A)^2}{\vol(V)},\;\;\text{where }A\subset V,$$
and it satisfies the following equalities (see Theorem 3.7 and Theorem 3.9 in \cite{TMH18}, respectively)
\begin{equation}\label{eq:Q-modularity-measure}
\max\limits_{A\subset V}Q(A)=\max\limits_{x\ne 0}\frac{\sum_{i,j\in V}(\frac{\deg(i)\deg(j)}{\vol(V)}-w_{ij})|x_i-x_j|}{4\|\vec x\|_\infty}
\end{equation}
and
\begin{equation}\label{eq:Q/mu-modularity-measure-mu}
\max\limits_{A\in\power(V)\setminus\{\varnothing,V\}}\frac{Q(A)}{\mu(A)\mu(V\setminus A)}=\max\limits_{\sum_{i\in V} \mu_ix_i=0}\frac{\sum_{i,j\in V}(\frac{\deg(i)\deg(j)}{\vol(V)}-w_{ij})|x_i-x_j|}{\mu(V)\sum_{i\in V}\mu_i|x_i|}.
\end{equation}

It is clear that  \eqref{eq:Q-modularity-measure} can be obtained more directly by Theorem \ref{thm:tilde-fg-equal}. We shall also state a new analog of \eqref{eq:Q/mu-modularity-measure-mu}: 
\begin{equation}\label{eq:Q/mu-modularity-measure}
\max\limits_{A\in\power(V)\setminus\{\varnothing,V\}}\frac{Q(A)}{\mu(A)\mu(V\setminus A)}=\max\limits_{\sum_{i\in V} x_i=0}\frac{\sum_{i,j\in V}(\frac{\deg(i)\deg(j)}{\vol(V)}-w_{ij})|x_i-x_j|}{\sum_{i,j\in V}\mu_i\mu_j|x_i-x_j|}
\end{equation}
which can be derived straightforwardly by Theorem \ref{thm:tilde-fg-equal}. 

By Proposition \ref{pro:Lovasz-f-pre}, we immediately obtain  Theorem 1 in \cite{CRT20}, i.e., 
for any $a,b>0$, 
$$ \max\limits_{-a\le x_i\le b,\forall i}\frac12\sum_{i,j\in V}(\frac{\deg(i)\deg(j)}{\vol(V)}-w_{ij})|x_i-x_j|=(a+b)\max\limits_{A\subset V}Q(A).$$

\textbf{A relation with the frustration index}

 For a signed weighted graph with real weights $(w_{ij})_{i,j\in V}$ and signs $s_{ij}=\mathrm{sign}(w_{ij})$, we define the frustration index as \begin{equation}\label{eq:frustration-weight}
\min_{x_i\in \{-1,1\},\forall i} \sum_{\{i,j\}}|w_{ij}|\cdot|x_i-s_{ij}x_j|.
\end{equation} 

The following result reveals an interesting relation between the modularity
measure and the frustration index.
\begin{pro}For a weighted graph $(V,(w_{ij})_{i,j\in V})$, let $\tilde{w}_{ij}=w_{ij}-\frac{\deg(i)\deg(j)}{\vol(V)}$. 
In the signed weighted graph $(V,(\tilde{w}_{ij})_{i,j\in V})$, 
 $\{i,j\}$ is a positive (resp. negative) edge if $\tilde{w}_{ij}>0$ (resp. $\tilde{w}_{ij}<0$). Then, the frustration index of $(V,(\tilde{w}_{ij})_{i,j\in V})$ equals 
$2\left(\sum_{\{i,j\}:\tilde{w}_{ij}<0}|\tilde{w}_{ij}|-\max\limits_{A\subset V}Q(A)\right)$.
\end{pro}

\begin{proof}
We know from Section \ref{sec:frustration} (or by Theorem \ref{thm:tilde-fg-equal}) that the frustration index of $(V,(\tilde{w}_{ij})_{i,j\in V})$ equals 
$$2\left(\sum_{\{i,j\}:\tilde{w}_{ij}<0}|\tilde{w}_{ij}|+\min\limits_{\vec x\ne0}\frac{\sum_{i,j\in V}\tilde{w}_{ij}|x_i-x_j|}{4\|x\|_\infty}\right).$$
The proof is then completed by \eqref{eq:Q-modularity-measure}.
\end{proof}

\subsection{Chromatic number}\label{sec:chromatic-number}
The chromatic number (i.e., the smallest vertex coloring number) of a graph is the smallest number of colors needed to color the vertices 
so that no two adjacent vertices share the same color.
Given a simple connected graph $G=(V,E)$ with $\#V=n$, its chromatic number $\gamma(G)$ can be expressed as a global optimization on the $n$-power set of vertices:
\begin{equation}\label{eq:coloring-number-sum}
\gamma(G)=\min\limits_{(A_1,\cdots,A_n)\in\power_n(V)}\left\{n\sum_{i=1}^n\#E(A_i)
+\sum_{i=1}^n\sgn(\#A_i)+n\left(n-\sum_{i=1}^n\#A_i\right)^2\right\}
\end{equation}
and similarly, we get the following
\begin{pro}The chromatic number $\gamma(G)$ of a finite simple graph $G=(V,E)$ satisfies
\begin{equation}\label{eq:coloring-number}
\gamma(G)=\min\limits_{(A_1,\cdots,A_n)\in\power(V)^n}\left\{n\sum_{i=1}^n\#E(A_i)+\sum_{i=1}^n\sgn(\#A_i)+n\left(n-\#\bigcup_{i=1}^n A_i\right)\right\}
\end{equation}
\end{pro}
\begin{proof}
 Let $f:\power(V)^n\to\R$ be defined by 
 \begin{equation*}\label{eq:color-f}
f(A_1,\cdots,A_n)=n\sum_{i=1}^n\#E(A_i)+\sum_{i=1}^n\sgn(\#A_i)+n\left(n-\#\bigcup_{i=1}^n A_i\right).     
 \end{equation*}
 Let $\{C_1,\cdots,C_{\gamma(G)}\}$ be a proper coloring class of $G$, and set  $C_{\gamma(G)+1}=\cdots=C_n=\varnothing$. Then we have $E(C_i)=\varnothing$, $\#\cup_{i=1}^n C_i=n$,  $\#C_i\ge1$ for $1\le i\le \gamma(G)$, and $\#C_i=0$ for $i> \gamma(G)$. In consequence, $f(C_1,\cdots,C_n)=\gamma(G)$. Thus, it suffices to prove $f(A_1,\cdots,A_n)\ge \gamma(G)$ for any $(A_1,\cdots,A_n)\in\power(V)^n$.

If $\bigcup_{i=1}^n A_i\ne V$, then $f(A_1,\cdots,A_n)\ge n+1> \gamma(G)$.

If there exist at least $\gamma(G)+1$  nonempty sets $A_1,\cdots,A_{\gamma(G)+1}$, then $f(A_1,\cdots,A_n)\ge \gamma(G)+1> \gamma(G)$.

So we focus on the case that $\bigcup_{i=1}^n A_i= V$ and $A_{\gamma(G)+1}=\cdots=A_n=\varnothing$. If there further exists $i\in\{1,\cdots,\gamma(G)\}$ such that $A_i=\varnothing$, then by the definition of  the chromatic number, there is $j\in \{1,\cdots,\gamma(G)\}\setminus\{i\}$ with $E(A_j)\ne \varnothing$. So $f(A_1,\cdots,A_n)\ge n+1> \gamma(G)$.
Accordingly, each of $A_1,\cdots,A_{\gamma(G)}$ must be nonempty, and thus $f(A_1,\cdots,A_n)\ge\gamma(G)$.

Also, when the equality $f(A_1,\cdots,A_n)=\gamma(G)$ holds, one may see from the above discussion that $A_1,\cdots,A_{\gamma(G)}$ are all independent sets of $G$ with $\bigcup_{i=1}^n A_i\ne V$.
\end{proof}

Let $\hat{f}:\power_2(V)^n\to\R$ be defined by 
\begin{equation*}\label{eq:color-f-pair}
\hat{f}(A_1^+,A_1^-,\cdots,A_n^+,A_n^-)=\sum_{i=1}^n(\sgn(\# (A_i^+\cup A_i^-))+n\#E(A_i^+\cup A_i^-))+n\left(n-\#\bigcup_{i=1}^n A_i^+\cup A_i^-\right).     
 \end{equation*}
 and based on \eqref{eq:coloring-number}, it is clear that $\gamma(G)=\min\limits_{(A_1^+,A_1^-,\cdots,A_n^+,A_n^-)\in\power_2(V)^n}\hat{f}(A_1^+,A_1^-,\cdots,A_n^+,A_n^-)$. 
Note that
$$\#\bigcup_{i=1}^n V^t(\vec x^i)=\#\{j\in V:\exists i \text{ s.t. }x_{i,j}>t\}=\sum_{j=1}^n\max\limits_{i=1,\cdots,n} 1_{x_{i,j}>t}=\sum_{j=1}^n 1_{\max\limits_{i=1,\cdots,n}x_{i,j}>t}$$
So the $n$-way Lov\'asz extension of $\#\bigcup_{i=1}^n A_i$ is
\begin{align*}
\int_{\min\vec x}^{\max\vec x}\#\bigcup_{i=1}^n V^t(\vec x^i)dt+\min\vec x\#\bigcup_{i=1}^n V(\vec x^i)&=\sum_{j=1}^n \int_{\min\vec x}^{\max\vec x}1_{\max\limits_{i=1,\cdots,n}x_{i,j}>t}dt+\min\vec x\#V
\\&=\sum_{j=1}^n(\max\limits_{i=1,\cdots,n}x_{i,j}-\min\vec x)+n\min\vec x
\\&=\sum_{j=1}^n\max\limits_{i=1,\cdots,n}x_{i,j}
\end{align*}
And the $n$-way disjoint-pair Lov\'asz extension of $\#\bigcup_{i=1}^n A_i^+\cup A_i^-$ is $\sum\limits_{j=1}^n\max\limits_{i=1,\cdots,n}|x_{i,j}|=\sum\limits_{j=1}^n\|\vec x^{,j}\|_\infty$.

The $n$-way Lov\'asz extension of $\sgn(\#A_i)$ is
\begin{align*}
\int_{\min\vec x}^{\max\vec x}\sgn(\#V^t(\vec x^i))dt+\min\vec x\sgn(\#V(\vec x^i))&= \int_{\min\vec x}^{\max\vec x^i}1dt+\min\vec x\sgn(\#V)
\\&= \max\limits_{j=1,\cdots,n}x_{i,j}-\min\vec x+\min\vec x= \max\limits_{j=1,\cdots,n}x_{i,j}
\end{align*}
and the $n$-way disjoint-pair  Lov\'asz extension of $\sgn(\#(A_i^+\cup A_i^-))$ is $\|\vec x^{i}\|_\infty$. Similarly, the $n$-way disjoint-pair Lov\'asz extension of $\#E(A_i^+\cup A_i^-)$ is $\sum_{j\sim j'}\min\{|x_{i,j}|,|x_{i,j'}|\}$. 
Thus, 
\begin{align*}
\hat{f}^L(\vec x)&=n\sum_{i=1}^n\sum_{j\sim j'}\min\{|x_{i,j}|,|x_{i,j'}|\} +\sum_{i=1}^n\|\vec x^{i}\|_\infty+n\left(n\|\vec x\|_\infty-\sum_{j=1}^n\|\vec x^{,j}\|_\infty\right)
\\&
=
n\sum_{i=1}^n\sum_{j\sim j'}\frac{2|x_{i,j}|+2|x_{i,j'}|-|x_{i,j}+x_{i,j'}|-|x_{i,j}-x_{i,j'}|}{2}+\sum_{i=1}^n\|\vec x^{i}\|_\infty+n^2\|\vec x\|_\infty-n\sum_{j=1}^n\|\vec x^{,j}\|_\infty
\\&
=
n^2\|\vec x\|_\infty
+n\sum_{i,j=1}^n\deg_j|x_{i,j}|
-n\frac{\sum\limits_{i=1}^n\sum\limits_{\{j,j'\}\in E}(|x_{i,j}+x_{i,j'}|+|x_{i,j}-x_{i,j'}|)}{2}+\sum_{i=1}^n\|\vec x^{i}\|_\infty-n\sum_{j=1}^n\|\vec x^{,j}\|_\infty
\end{align*}
According to Proposition \ref{pro:fraction-f/g} in the context  of the multi-way  disjoint-pair  Lov\'asz extension, we obtain
\begin{align*}\label{eq:chromatic-continuous}
\gamma(G)&=\min\limits_{(A_1^+,A_1^-,\cdots,A_n^+,A_n^-)\in\power_2(V)^n}\frac{\hat{f}(A_1^+,A_1^-,\cdots,A_n^+,A_n^-)}{1}= \min\limits_{\vec x\in \R^{n^2}\setminus\{\vec 0\}} \frac{\hat{f}^L(\vec x)}{\|\vec x\|_\infty}
\\&=n^2-\max\limits_{\vec x\in\R^{n^2}\setminus\{\vec 0\}}\sum\limits_{k\in V}\frac{n\sum\limits_{\{i,j\}\in E}(|x_{ik}-x_{jk}|+|x_{ik}+x_{jk}|)+2n\|\vec x^{,k}\|_{\infty}-2n\deg_k\|\vec x^{,k}\|_1- 2\|\vec x^{k}\|_{\infty}}{2\|\vec x\|_\infty}.
\end{align*}
%

\paragraph{Clique covering number }

The clique covering number of a graph $G$ is the minimal number of cliques in $G$ needed to cover the vertex set. It is equal to the chromatic number of the graph complement of $G$. Consequently, we can explicitly write down the continuous representation of a clique covering number by employing Theorem \ref{thm:graph-numbers}.

\section{Conclusions and Discussion}

The firm bridge between the discrete data world and the continuous mathematical  field with well-established mathematics such as analytic techniques, topological schemes and algebraic structures should be tremendously helpful. 
In 
\cite{JostZhang,JZ-prepare21} 
and in this  paper, we build these fruitful connections in a variety of areas through  Lov\'asz  extension and some more general  discrete-to-continuous extensions. 
Our contribution is two-fold:  the  theoretical  framework  for  Lov\'asz-type extensions and the corresponding  spectral theory; and
their practical applications  to the computation of the resulting optimization and eigenvalue problems. 
 Let us describe  
the contributions  of this paper in more specific terms. 

\vspace{0.2cm}

\noindent \textbf{Contributions to   optimization}.  Continuous  approaches for solving combinatorial optimization problems have been widely used in practice, such as spectral clustering and its recursive versions,   SDP-type techniques, and  polynomial methods. 
Overall, continuous approaches can be  roughly  classified into continuous relaxations and  continuous reformulations, where the continuous reformulations are also called the equivalent continuous representations (or  tight relaxations) of the original combinatorial problems. 
However, most of these approaches require certain additional  rounding techniques, even for many  equivalent continuous formulations.    
In addition, although some tight relaxations (i.e., equivalent continuous formulation) have been constructed for certain combinatorial optimization problems,  many of the constructions  are specific  and   not general enough to be applied to a wider range of   combinatorial optimization problems. 

Our constructions based on the multi-way Lov\'asz  extensions  overcome these inconveniences.  In fact, the 
equivalent continuous optimization problem we obtained 
fully  inherits all the local optimal data of the original combinatorial objective function,  and therefore fits better with the original combinatorial optimization problem. Therefore, our discrete-to-continuous framework   
is  more convenient and appropriate  for obtaining new  relaxations and reformulations than   
many other  approaches.   
Also, the reformulations  obtained  by multi-way Lov\'asz  extension are of  simple ratio form,  which offer new possibilities for designing continuous optimization algorithms for combinatorial problems in practical terms.  
In particular, we provide the mixed IP-SD scheme to  confirm the effectiveness of our discrete-to-continuous framework, which has worked well in many practical combinatorial optimization problems. 
For example, in \cite{SZZmaxcut}  we  proposed  a simple iterative  algorithm for maxcut, which is based on a previous specific version of the mixed IP-SD scheme and which  performs very well in numerical experiments. 
This method can also be used to  find specific  eigenvalues of a  function pair (see \cite{JostZhang}).   
We believe that the mixed IP-SD algorithm is one of the best continuous iterative schemes  for solving certain combinatorial optimization problems like independence number, coloring number and frustration index,  because it fully exploits some new equivalent continuous formulations. 
It is expected that  further efficient  continuous optimization algorithms will be designed for combinatorial optimization problems based on our discrete-to-continuous framework.

\vspace{0.2cm}

\noindent \textbf{Contributions to nonlinear  eigenvalue problems
}.  Nonlinear eigenvalue problems arise in many contexts, including quantum chemistry, physics, engineering, and image processing. Recently, nonlinear operators and the associated spectral theories have  allowed for more general, accurate and efficient models and techniques  for handling network problems. 
For example, the  1-Laplace operator on graphs has been successfully applied to  spectral clustering with a spectrum that has many good properties and is closely related to  multi-way Cheeger constants. 
However, it is not entirely clear why the  1-Laplacian is good, and whether similar properties can be generalized to other nonlinear operators. 

Our framework on multi-way Lov\'asz extensions  establishes a  systematic and deep spectral theory for a class of nonlinear operators. 
We prove that the spectrum of the function pair obtained by the multi-way Lov\'asz extension encodes all the key  data of the original combinatorial functions, and we particularly characterize the second eigenvalue  in terms of combinatorial quantities.   
This  generalizes  the important fact that the second eigenvalue of the  graph 1-Laplacian equals the graph Cheeger constant. We also provide many other applications, for example, we  found that  the min-cut and max-cut problems are equivalent to solving  the first nontrivial  eigenvalue and the largest eigenvalue of a certain nonlinear eigenvalue  problem provided by the Lov\'asz extension,  respectively.  
Further progress is collected in \cite{JostZhang},  and based on these fundamental  results, we can analyze the structure of eigenspaces in depth.

\vspace{0.2cm}

\noindent \textbf{Relations to other works and further remarks}. There are many other applications of Lov\'asz extension beyond this paper, for example, critical point theory for combinatorial  functions can be studied with the help of Lov\'asz extension. In \cite{JZ-prepare21}, we build the relationship between the Morse theory of a discrete Morse function and its Lov\'asz extension. We also propose a combinatorial version of the  Lusternik-Schnirelman category on abstract simplicial complexes  to bridge the classical Lusternik-Schnirelman theorem and its discrete analog on finite simplicial complexes. 

For further applications, we introduce the piecewise multilinear extension in  \cite{JostZhang}, and we provide several min-max relations based on such general extension. 
The mountain pass characterizations, linking theorems, nodal domain inequalities, inertia bounds, duality theorems and distribution of eigenvalues for pairs of $p$-homogeneous functions are derived.  In particular, we show a simple one-to-one correspondence between the nonzero eigenvalues of the vertex $p$-Laplacian and the edge $p^*$-Laplacian of a graph. We  also apply the extension theory to Cheeger inequalities and $p$-Laplacians on oriented hypergraphs and simplicial complexes, which contribute to the field of expander graph and spectral graph theory.   In addition, these results  have some applications on tensor eigenvalues, providing a strong spectral estimate for  the adjacency tensor of a hypergraph.

\vspace{1em}

{\bf Acknowledgements.} 
 Much of this work was done when the second author was working  at the Max Planck Institute for Mathematics in the Sciences (MPI MiS). 
He is very grateful to the Max Planck Institute for the excellent working conditions and atmosphere. 
This work is supported by grants from  Fundamental Research Funds for the Central Universities (No. 7101303088).

{ \linespread{0.95} 
}

\end{document}